\documentclass[a4paper,reqno,11pt]{amsart}

\usepackage[left=1 in, right=1 in,top=1 in, bottom=1 in]{geometry}

\usepackage{amsfonts}
\usepackage{amssymb}
\usepackage{amsthm}
\usepackage{amsmath}
\usepackage{mathrsfs}
\usepackage{color}
\allowdisplaybreaks

\usepackage{pdfsync} 

\numberwithin{equation}{section} 

\newcommand*\rmd{\mathop{}\!\mathrm{d}}

\newcommand{\ubz}{{\genfrac{(}{)}{0pt}{1}{0}{\tilde{u}^\flat}}}

\newcommand{\ubb}{ { \left( \genfrac{(}{)}{0pt}{1}{0}{\tilde{u}^\flat} \right) } }
\newcommand{\tu}{\tilde{u}}
\newcommand{\tub}{\tilde{u}^\flat}
\newcommand{\ta}{\widetilde{A}}
\newcommand{\tb}{\widetilde{B}}

\newcommand{\td}{\widetilde{D}}
\newcommand{\twd}{\widetilde{WD}}
\newcommand{\ttk}{{K_0}}
\newcommand{\ttkk}{{K_0^2}}
\newcommand{\tk}{{\widetilde{K}_0}}
\newcommand{\tl}{\widetilde{\mathcal{L}}^\flat}
\newcommand{\tn}{\widetilde{\mathcal{N}}^\flat}
\newcommand{\tnf}{\widetilde{\mathcal{N}}^1}
\newcommand{\tng}{\widetilde{\mathcal{N}}^2}
\newcommand{\tnc}{\widetilde{\mathcal{N}}^\mathtt{C}}
\newcommand{\tnd}{\widetilde{\mathcal{N}}^\mathtt{D}}
\newcommand{\teta}{\tilde{\eta}}
\newcommand{\tr}{\widetilde{R}}
\newcommand{\tH}{\widetilde{H}}

\newcommand{\tq}{\widetilde{Q}}
\newcommand{\ts}{\widetilde{S}}
\newcommand{\ttt}{\widetilde{T}}
\newcommand{\tw}{\widetilde{W}}
\newcommand{\tv}{\widetilde{V}}
\newcommand{\tx}{\widetilde{X}}
\newcommand{\ty}{\widetilde{Y}}
\newcommand{\tz}{\widetilde{Z}}
\newcommand{\tga}{\widetilde{\Gamma}}
\newcommand{\utra}[2]{u^{(1)}(#1,#2)}
\newcommand{\pic}{\Pi^\mathtt{C}}

\newcommand{\tmmd}{\widetilde{M}^\mathtt{D}}

\newcommand{\ntu}{\nabla_{\! \tilde u}}
\newcommand{\nuu}{\nabla_{\! u}}

\newcommand{\mtd}{\mathtt{D}}
\newcommand{\mtc}{\mathtt{C}}
\newcommand{\tuc}{\tilde{u}^{\mtc}}
\newcommand{\tud}{\tilde{u}^{\mtd}}
\newcommand{\thd}{{ {\Theta}^{\mtd} }}

\newcommand{\wik}{{\tilde{w}_{i,k}}}
\newcommand{\wiik}{{\tilde{w}_{i'\! ,k}}}
\newcommand{\wikk}{{\tilde{w}_{i,k'}}}
\newcommand{\wiikk}{{\tilde{w}_{i'\! ,k'}}}
\newcommand{\wok}{{\tilde{w}_{1,k}}}
\newcommand{\wokk}{{\tilde{w}_{1,k'}}}

\newcommand{\wjjk}{{\tilde{w}_{j'\!,k}}}

\newcommand{\tvv}{{\tilde{v}}}
\newcommand{\trr}{{\tilde{r}}}
\newcommand{\tll}{{\tilde{l}}}
\newcommand{\tlam}{\tilde{\lambda}}

\newcommand{\re}{\mathrm{Re}}

\newcommand{\pd}{{p_{{\!}_d}}}

\newcommand{\ps}{p_{{}_S}}
\newcommand{\Div}{\mathrm{div}}

\newcommand{\asmp}[1]{{$(\mathfrak{#1})$}} 

\newcommand{\sek}{{ \mathtt{e}_k }}
\newcommand{\sekk}{{ \mathtt{e}_{k'} }}

\newcommand{\hs}{{ \quad }} 

\providecommand{\abs}[1]{\left\vert#1\right\vert}
\providecommand{\norm}[1]{\left\Vert#1\right\Vert}
\providecommand{\pnorm}[2]{\left\Vert#1\right\Vert_{L^{#2}}}

\providecommand{\Rn}[1]{\mathbb{R}^{#1}}
\providecommand{\rd}{\Rn{d}}

\newcommand{\resp}{{resp.~}} 

\newcommand{\na}{\nabla}
\newcommand{\al}{\alpha}
\newcommand{\dt}{\partial_t}
\newcommand{\ddt}{ \frac{\rmd}{\rmd t}}
\def\hal{\frac{1}{2}}

\def\ls{\lesssim}

\def\p{\partial}

\def\dis{\displaystyle}
\def\La{\Lambda}

\def\fj1{\mathcal{J}^{-1}}

\def\m{\mathcal{M}}

\def\lds{\Lambda^s}
\def\ldso{\Lambda^{s-1}}
\def\ldll{\Lambda^\ell}
\def\ldl1{\Lambda^{\ell-1}}

\newtheorem{lem}{Lemma}[section]

\newtheorem{prop}[lem]{Proposition}
\newtheorem{thm}{Theorem}
\newtheorem{remark}[lem]{Remark}

\begin{document}

\title[Partially dissipative hyperbolic systems]{Global Classical Solutions to Partially Dissipative Hyperbolic Systems Violating the Kawashima Condition}

\author{Peng Qu}
\address{School of Mathematical Sciences\\ Fudan University\\
Shanghai 200433, China
\newline \indent and
\newline \indent The Institute of Mathematical Sciences\\
The Chinese University of Hong Kong\\
Shatin, NT, Hong Kong}
\email[P. Qu]{pqu@fudan.edu.cn}
\thanks{P. Qu was supported by Yang Fan Foundation of Shanghai on Science and Technology (No. 15YF1401100)}

\author{Yanjin Wang}
\address{School of Mathematical Sciences\\
Xiamen University\\
Xiamen, Fujian 361005, China
\newline \indent and
\newline \indent The Institute of Mathematical Sciences\\
The Chinese University of Hong Kong\\
Shatin, NT, Hong Kong}
\email[Y. J. Wang]{yanjin$\_$wang@xmu.edu.cn}
\thanks{Y. J. Wang was supported by the National Natural Science Foundation of China (No. 11201389)}





\begin{abstract}
This paper considers the Cauchy problem for the quasilinear hyperbolic system of balance laws in {$\rd$, $d\ge 2$.
The system is partially dissipative in the sense that there is an eigen-family violating the Kawashima condition.} By imposing certain supplementary degeneracy conditions with respect to the non-dissipative eigen-family,
global unique smooth solutions near constant equilibria are constructed.
The proof is based on the introduction of the partially normalized coordinates, a delicate structural analysis,
a family of scaled energy estimates with {minimum fractional} derivative counts
and a refined decay estimates of the dissipative components of the solution.
\end{abstract} 

\maketitle


\section{Introduction}\label{intro sec}

Consider the following $n$-component quasilinear hyperbolic system of balance laws in $d$ space dimensions:
\begin{equation}\label{1.1_Sys}
\begin{cases}
\displaystyle {\color{black}\partial_t \left( G(u)\right) }+ \sum_{k=1}^d  \left( F^k(u)\right)_{x_k} = S(u) ,
\\ \displaystyle
{u\mid_{t=0} = u_0.}
\end{cases}
\end{equation}
Here $u=(u_1,\dots,u_n)^T: \mathbb{R}\times \mathbb{R}^d \rightarrow \mathbb{R}^n $ is the unknown,
{\color{black} $G = (G_1, \dots, G_n)^T : \mathbb{R}^n \rightarrow \mathbb{R}^n$, $F^k = (F^k_1, \dots, F^k_n)^T : \mathbb{R}^n \rightarrow \mathbb{R}^n $ $(k=1,\dots,d)$ and  $S = (S_1, \dots, S_n)^T : \mathbb{R}^n \rightarrow \mathbb{R}^n $ are given smooth functions. It is assumed that $\nuu G(u)$ is invertible, hence, by the chain rule, the system can be rewritten into}
\begin{equation} \label{sys}
\begin{cases}
\displaystyle \partial_t u + \sum_{k=1}^d  A^k(u) u_{x_k} = Q(u) ,
\\ \displaystyle
u\mid_{t=0} = u_0.
\end{cases}
\end{equation}
{\color{black}Here the coefficient matrices $A^k(u) = (\nuu G(u))^{-1} \nuu F^k(u)$ $(k=1,\dots,d)$ and the inhomogeneous term $Q(u) =  (\nuu G(u))^{-1}  S(u)$.}
By hyperbolicity, for any $u\in\mathbb{R}^n$ and any $\omega = (\omega_1, \dots, \omega_d)^T \in \mathbb{S}^{d-1}$ the matrix
\begin{equation}\label{A matrix}
A(u,\omega) := \sum_{k=1}^d \omega_k A^k(u)
 \end{equation}
has $n$ real eigenvalues $\lambda_1(u,\omega),\dots,\lambda_n(u,\omega)$ and a complete set of left (\resp right) eigenvectors $l_i(u,\omega) = (l_{i1}(u,\omega), \dots, l_{in}(u,\omega)) $ (\resp $r_i(u,\omega) = (r_{1i}(u,\omega), \dots, r_{ni}(u,\omega))^T$) $(i=1,\dots,n)$. It is assumed that $\lambda_i(u,\omega)\, (i=1,\dots,n)$ are smooth with respect to $u$ and $\omega$, while $l_i(u,\omega),\, r_i(u,\omega)\, (i=1,\dots,n)$ are {all smooth functions of $u$ for any given $\omega \in \mathbb{S}^{d-1}$}.
One may normalize the eigenvectors so that
\begin{equation}
l_i(u,\omega) r_{i'}(u,\omega) \equiv \delta_{ii'}, \quad   i, i'=1,\dots, n. \label{1.4}
\end{equation}
It is well known that for any given suitably smooth initial data $u_0$, the system \eqref{1.1_Sys} has a unique local smooth solution \cite{kato1975local,majda1984local},
but in general the solution develops singularities in finite time even for small and smooth initial data \cite{dafermos2005book,li1994bookglobal,majda1984local}.

{It has been known that there are two structures that can prevent the finite-time singularity formation.
{\color{black} One of them} is the (weak) linear degeneracy for the homogeneous systems, which was firstly introduced in 1D study {\color{black}(see \cite{bressan1988global} and \cite{li1994bookglobal})}.
The other one is the dissipation effect induced by {\color{black}inhomogeneous terms}.
A well-known assumption in this aspect is given by the strict dissipation, which requires a {\color{black}damping term} to enter into each of the equations of the system, {\color{black} see \cite{li1994bookglobal}}.
Notice that for some physical problems, this condition is too strong, while in many applications and as the interest in this paper $Q(u)$ takes the form
\begin{equation} \label{1.2}
Q(u) = (0,\dots, 0, Q_{r+1}(u), \dots, Q_{n}(u))^T,
\end{equation}
which seems to produces the dissipation effect only in the last $n-r$ equations.} In this case, the system \eqref{1.1_Sys} is referred to as a partially dissipative hyperbolic system in many literatures, {\color{black} such as \cite{beauchardzuazua2011ARMA,bianchini2007CPAM,hanouzetnatalini2003ARMA,yong2004ARMA}}.

A constant vector $u_\ast\in \mathbb{R}^n$ is called an equilibrium state of the system \eqref{1.1_Sys} {\color{black}if $S(u_\ast)=0$, or equivalently, $Q(u_\ast)=0$}.
The main subject of this paper is to search for suitable structural conditions, under which \eqref{1.1_Sys} has a unique global smooth solution near $u_e$.
Without loss of generality, one may take $u_\ast=0$, $i.e.$, $Q(0)=0$. {\color{black}One may also suppose that $G(0) = 0$ and $\nuu G(0) = I_n$.}
A useful assumption (see \cite{yong2004ARMA}) is the following entropy dissipation condition:
\begin{itemize}
\item[\asmp{A1}] The matrix $\left( \frac{\partial Q_p}{\partial u_{p'}}(0)\right)_{p,p'=r+1}^n$ is invertible.
\item[\asmp{A2}] {\color{black}There exist a strictly convex smooth entropy function $\bar\eta(G)$ and $d$ smooth entropy flux functions $\bar\psi^k(G)$ $(k=1,\dots,d)$ for the system \eqref{1.1_Sys} in terms of $G$, such that for all $G$,
\[
 (\nabla_G \bar \psi^k (G))^T = (\nabla_G \bar \eta(G))^T \nabla_G (F^k(u(G))).
\]
This implies that if one defines ${\eta}(u) = \bar\eta(G(u)) - \bar\eta(0) - \nabla_G \bar\eta(0) \cdot G(u)$ and ${\psi}^k(u) = \bar\psi^k(G(u)) - \nabla_G \bar\eta(0) \cdot F^k(u)$, then $\eta(0) = 0$, $\nuu \eta(0) = 0$, $\nabla_{u}^2 \eta(0)$ is positive definite and}
\begin{equation}
{\color{black}(\nuu \psi^k(u))^T = (\nuu \eta(u))^T A^k(u), \hs k=1,\dots,d.}
\end{equation}
\item[\asmp{A3}] There exists a positive constant $c_e > 0$ such that for all $u$,
\begin{equation} \label{1.9_EnDis}
{ \nuu \eta(u)  \cdot Q(u) \leq -c_e |Q(u)|^2.}
\end{equation}
\end{itemize}
It is clear that the entropy dissipation condition \asmp{A1}--\asmp{A3} is too weak to prevent the formation of singularities,
for instance, a decoupled system composed of an inviscid Burgers' equation and a damped linear equation satisfies \asmp{A1}--\asmp{A3},
but generally would develop singularities in a finite time for non-degenerate initial data.

A natural supplementary coupling condition, the Shizuta--Kawashima stability condition, or ``Kawashima condition'' for short which was first formulated in \cite{kawashima1984thesis,shizutakawashima1985systems} for the hyperbolic-parabolic systems with a source term, was introduced in \cite{hanouzetnatalini2003ARMA,yong2004ARMA}.
This coupling condition implies that due to the wave propagation,
the partial dissipation term indeed produces a kind of complete dissipation effect with respect to all the components of the solution, which yields the global existence of the small smooth solutions.
The Kawashima condition has several equivalent formulations, and one of them reads as
\begin{itemize}
\item[\asmp{K}] The kernel of the Jacobian $\nuu Q(0)$ contains no eigenvector of $A(0,\omega)$ for any $\omega \in \mathbb{S}^{d-1}$.
\end{itemize}
Under the assumptions \asmp{A1}--\asmp{A3} and \asmp{K}, \cite{hanouzetnatalini2003ARMA,yong2004ARMA} prove the global smooth solutions for small initial data.
{\color{black} See also \cite{beauchardzuazua2011ARMA,bianchini2007CPAM,kawashimayong2004ARMA,kawashimayong2009ZAA,ruggeriserre2004stability} for more discussions and
\cite{dafermos2013JDE} for the result on global entropy solutions to systems satisfying Kawashima condition with zero mass initial data}.

It should be pointed out that although the name partial dissipation is used in this case, {all the components of the solution are dissipative due to the Kawashima coupling condition \asmp{K}.
On the other hand, any other kind of partial dissipations that certain part of the solution may not be linearly affected by the dissipation and hence is non-dissipative,} which can be represented by the violation of the Kawashima condition \asmp{K}, is an interesting structure to study (see remarks in \cite{bianchini2007CPAM,mascianatalini2010ARMArelaxation,yong2004ARMA}).
In fact, there are many physical models that the Kawashima condition \asmp{K} is violated, and the eigen-family that violates it is linearly degenerate, such as {the damping full Euler system of adiabatic flow} (see, for instance, {\color{black}\cite{hsiaoserre1996global,tan2013global})} and {\color{black}the thermal non-equilibrium flow (see \cite{zeng1999gas,zeng2010gas,zeng2015gas}).}
Some works have discussed several aspects on the system violating the condition \asmp{K}, such as \cite{mascianatalini2010ARMArelaxation} for special structures and {\color{black}\cite{liuqu2012JMPA}} for partial strict dissipations in 1D. 

In this paper we will continue the study on the system \eqref{1.1_Sys} {with $d\ge 2$ and $r\ge 2$ (and hence $n\ge3$) in \eqref{1.2} violating \asmp{K}. More precisely, we assume}
\begin{itemize}
\item[\asmp{A4}] For each $\omega \in \mathbb{S}^{d-1}$ one and only one family of right eigenvectors of $A(0,\omega)$ are not contained in the kernel of the Jacobian $\nuu Q(0)$, that is, possibly after relabeling,
\begin{equation} \label{1.10_Kaw}
\nuu Q(0) r_j(0,\omega) \neq 0,\hs j=2,\dots,n,
\end{equation}
but
\begin{equation}
\nuu Q(0) r_1(0,\omega) = 0.
\end{equation}
\end{itemize}
One may refer to the assumption \asmp{A4} that the last $n-1$ right eigenvectors satisfy the Kawashima condition and the first right eigenvector violates it.
Note that \asmp{A1}--\asmp{A4} is still too weak to prevent the formation of singularities, and some additional structural conditions are needed in order to guarantee the global existence of the small smooth solutions.
The goal of this paper is to identify one set of such conditions and prove the corresponding results.

For notational simplicity, $L^p$ and $H^s$ are used to denote the usual $L^p$ and Sobolev spaces on $\mathbb{R}^d$.
The notation $\Lambda ^{s }$ is defined by $\Lambda^s f :=  (|\xi|^s  \hat{f})^\vee $, where $\hat{\cdot}$ is the Fourier transform and $(\cdot)^\vee$ is its inverse. $C$ denotes the positive constants, whose value may change from line to line,
and $f \lesssim g$ is used for $f \le C g$.


\section{Main results} \label{sec:2+}

As mentioned in Section \ref{intro sec}, the aim of this paper is to impose certain structural conditions in supplement with the assumptions \asmp{A1}--\asmp{A4} so that the system \eqref{1.1_Sys} admits a unique global smooth solution for the small initial data.
First, for the technical point and as suggested by many physical applications, one may assume
\begin{itemize}
\item[\asmp{B}] The first family of eigenvectors satisfy an isotropic condition:
\begin{equation}\label{1.12_iso_r}
l_1(u,\omega) \equiv   l_1(u) \text{ and }
r_1(u,\omega) \equiv   r_1(u), \hs \forall\, \omega \in \mathbb{S}^{d-1}.
\end{equation}
\end{itemize}
In the spirit of the normalized coordinates method developed in {\color{black}\cite{li1994bookglobal,lizhoukong1994global}} for the one dimensional quasilinear hyperbolic systems,
thanks to the isotropy of $r_1$ in \eqref{1.12_iso_r}, the partially normalized coordinates can be introduced as follows.
After the linear transformation performed in Appendix \ref{linear trans}, let
$\utra{s}{u^0}$ be the $1$st characteristic trajectory passing through the point $u^0$, which is defined by
\begin{equation} \label{1.13}
\begin{cases}
\displaystyle \frac{\rmd}{\rmd s} \utra{s}{u^0} = r_1(\utra{s}{u^0}), \\
\utra{0}{u^0} = u^0.
\end{cases}
\end{equation}
For each point $\tilde{u} = (\tilde{u}_1, \dots, \tilde{u}_n)^T$, map it to $u = (u_1, \dots, u_n)^T$ by
\begin{equation} \label{2.5_trans}
u = u(\tilde u) := \utra{\tu_1}{(0,\tu_2, \dots, \tu_n)^T}.
\end{equation}
It will be verified in Section \ref{sec:2} that the mapping \eqref{2.5_trans} is a local diffeomorphism near $\tu=0$. This then allows to define the partially normalized coordinates $\tu$ by the inverse mapping $\tu=\tu(u)$ for each point $u$ near $u(0)=u^{(1)}(0,0)=0$ in the phase space, and hence the system \eqref{1.1_Sys} can be equivalently reformulated in terms of $\tu$.
This reformulated system will be studied in details in Section \ref{sec:2}.
Throughout the paper the following notational convention is applied for vectors: for a vector $f=(f_1,\dots,f_n)^T\in \mathbb{R}^n$, denote
\begin{equation}\label{jjkk}
f^\flat =\left(f_{2},\dots,f_n \right)^T\in \mathbb{R}^{n-1},\  f^\mathtt{C}=\left(f_{2},\dots,f_r \right)^T\in \mathbb{R}^{r-1}\text{ and } f^\mathtt{D} =\left(f_{r+1},\dots,f_n \right)^T\in \mathbb{R}^{n-r}.
\end{equation}
As one shall see, the decomposition \eqref{jjkk} is closely related to the structure of the system \eqref{1.1_Sys} under assumptions \asmp{A1}--\asmp{A4} and \asmp{B}.

Since the assumption \asmp{A4} implies that the first right eigenvector violates the Kawashima condition, the natural way is to supplement certain degeneracy conditions with respect to the first eigen-family.
One may use the following weak degeneracy conditions:
\begin{itemize}
\item[\asmp{WD1}] $Q(u)$ is weakly degenerate along the first characteristic trajectory that passes through $0$:
\[
Q(u^{(1)}(s,0)) \equiv 0, \hs \forall\, \abs{s} \text{ small},
\]
which is equivalent to
\begin{equation}\label{1.12_VioKaw_w}
\left. \left( \nuu Q(u) r_1(u) \right) \right|_{u=\utra{s}{0}} \equiv 0,\quad \forall\, |s|\text{ small}.
\end{equation}
\item[\asmp{WD2}] The first eigen-family is weakly linearly degenerate {\color{black}(see \cite{li1994bookglobal,lizhoukong1994global})}:
\[
\lambda_1(u^{(1)}(s,0),\omega) \equiv \lambda_1(0,\omega), \hs \forall\, \abs{s} \text{ small},\ \forall\, \omega \in \mathbb{S}^{d-1},
\]
which is equivalent to
\begin{equation} \label{1.13_LD}
\left. \left( \nuu \lambda_1(u,\omega) \cdot r_1(u) \right) \right|_{u = \utra{s}{0}}  \equiv 0,\quad \forall\, |s|\text{ small},\ \forall\,\omega \in \mathbb{S}^{d-1} .
\end{equation}
\end{itemize}
It is easy to check that the conditions \asmp{WD1}--\asmp{WD2} are necessary
since if any one of them is missing, the solution generally will blow up in finite time.
The typical examples can be constructed in the spirit of the Riccati equation and the Burgers' equation (see \cite{li1994bookglobal} and \cite{liuqu2012JMPA} for instance).


The first result of this paper is that under the assumptions \asmp{A1}--\asmp{A4}, \asmp{B} and \asmp{WD1}--\asmp{WD2}, the system \eqref{1.1_Sys} admits a unique global solution for each suitably smooth and small initial data.

\begin{thm} \label{thm:1.1}
Suppose that the assumptions \asmp{A1}--\asmp{A4}, \asmp{B} and \asmp{WD1}--\asmp{WD2} hold.
Let $d \geq 3$, $\ell>d/2+1$ and $1\le p\le 2,\ 1\le p<d/2 $. Denote $p^\ast=\min\{2, {dp }/({d-p })\}$ and $ s_1 ^\ast= d(1-1/p )+1$.

There exists a sufficiently small $\varepsilon_0>0$ such that if
\begin{equation} \label{K00}
K_0:=\norm{u_0}_{H^\ell}+\norm{   \tuc_0 }_{L^{p }}+\norm{\tud_0 }_{L^{p^\ast}} \leq \varepsilon_0,
\end{equation}
then there exists a unique global solution $u(t,x)$ to the system \eqref{1.1_Sys} satisfying
\begin{equation}\label{th11}
\norm{u(t)}_{H^\ell}^2 + \int_0^t \left( \norm{\tud(\tau)}_{H^\ell}^2+\norm{\Lambda  \tuc(\tau)}_{H^{\ell-1}}^2 \right) \rmd \tau \le  CK_0^2.
\end{equation}
Moreover, for all $s\in [ 0,   s_1^\ast]\cap [0,d/2)$,
\begin{equation}
\norm{\lds\tuc(t) }_{H^{\ell-s}}\le  C K_0(1+t)^{-\frac{d}{2}\left(\frac{1}{p}-\frac{1}{2}\right)-\frac{s}{2} }
\end{equation}
and for all $s\in [ 0,   s_1^\ast-1]\cap [0,d/2-1)$,
\begin{equation}\label{th13}
\norm{\lds\tud(t) }_{H^{\ell-s}}\le  C K_0(1+t)^{-\frac{d}{2}\left(\frac{1}{p}-\frac{1}{2}\right)-\frac{s}{2}-\frac{1}{2} }.
\end{equation}
\end{thm}
\begin{remark}
Note that if $d\ge 5$, then one can take $p=2$ in Theorem \ref{thm:1.1} which yields the unique global solution to the system \eqref{1.1_Sys} for the small initial data in $H^\ell$ with $\ell>d/2+1$; but if $d=3$ or $4$, in order to obtain the global solutions one needs the low frequency assumption of the initial data as stated in \eqref{K00} with $1\le p<d/2$ (if $p\ge 2d/(d+2)$ and hence $p^\ast=2$, then one needs only the low frequency assumption of $\tuc_0$).
\end{remark}

To improve the result of Theorem \ref{thm:1.1}, one may use the following stronger degeneracy conditions:
\begin{itemize}
\item[\asmp{D1}] $Q(u)$ is degenerate along \emph{each} first characteristic trajectory:
\begin{equation}\label{1.14_VioKaw}
\nuu Q(u) r_1(u) \equiv 0,\quad \forall\,u.
\end{equation}
\item[\asmp{D2}] The first eigen-family is linearly degenerate (see \cite{lax1957hyperbolic}):
\begin{equation} \label{1.15_LD}
\nuu \lambda_1(u,\omega) \cdot r_1(u) \equiv 0,\quad \forall\,u,\ \forall\, \omega \in \mathbb{S}^{d-1}.
\end{equation}
\end{itemize}
Then the second result of this paper can be stated as follows.

\begin{thm} \label{thm:1.2}
Suppose that the assumptions \asmp{A1}--\asmp{A4}, \asmp{B} and \asmp{D1}--\asmp{D2} hold.
Let $d \geq 3$, $\ell>d/2+1$ and $1\le p\le 2$, require further that $p < 2d/(2(3-\ell)+d)$ if $\ell\le 3$.
Denote $ s_2^\ast=\min\{d(1-1/p)+1,\ell-1\}$.

There exists a sufficiently small $\varepsilon_0>0$ such that if
$K_0 \leq \varepsilon_0$,
then there exists a unique global solution $u(t,x)$ to the system \eqref{1.1_Sys} satisfying
\begin{equation}\label{th21}
\norm{u(t)}_{H^\ell}^2 + \int_0^t \left( \norm{\tud(\tau)}_{H^\ell}^2+\norm{\Lambda  \tuc(\tau)}_{H^{\ell-1}}^2 \right) \rmd \tau \le  CK_0^2.
\end{equation}
Moreover, for all $s\in [ 0,   s_2^\ast]\cap [0,d/2+1)$,
\begin{equation}
\norm{\lds\tuc(t) }_{H^{\ell-s}}\le  C K_0(1+t)^{-\frac{d}{2}\left(\frac{1}{p}-\frac{1}{2}\right)-\frac{s}{2} }
\end{equation}
and for all $s\in [ 0,   s_2^\ast-1]\cap [0,d/2 )$,
\begin{equation}\label{th23}
\norm{\lds\tud(t) }_{H^{\ell-s}}\le  C K_0(1+t)^{-\frac{d}{2}\left(\frac{1}{p}-\frac{1}{2}\right)-\frac{s}{2}-\frac{1}{2} }.
\end{equation}
\end{thm}
\begin{remark}
Note that now one can take $p=2$ in Theorem \ref{thm:1.2} for any $d\ge 3$ which yields the unique global solution to the system \eqref{1.1_Sys} for the small initial data in $H^\ell$ with $\ell > \max\{d/2+1,3\}$ and without any low frequency assumption of the initial data.
\end{remark}

Note that the result of Theorem \ref{thm:1.2} excludes the case $d=2$. To obtain the global solution to the system \eqref{1.1_Sys} in two dimensions, one more assumption should be made on the structure of the system:
\begin{itemize}
\item[\asmp{D3}] The first eigen-family gets no linear effect from the source term:
\begin{equation} \label{C3}
l_1(0) \nuu Q(0) = 0.
\end{equation}
\end{itemize}
And the third result of this paper is stated as follows.

\begin{thm} \label{thm:1.3}
Suppose that the assumptions \asmp{A1}--\asmp{A4}, \asmp{B}, \asmp{D1}--\asmp{D2} and \asmp{D3} hold.
Let $d \geq 2$, $\ell>d/2+1$ and $1\le p\le q\le 2,\ 1\le p\le q<d $, require further that $p < 2d/(2(3-\ell)+d)$ if $\ell\le 3$.
Denote $  s_3^\ast=\min\{d(1/2+1/q-1/p)+1, d(1-1/p)+2,\ell-1\}$.

There exists a sufficiently small $\varepsilon_0>0$ such that if
\begin{equation}
\widetilde{K}_0:=\norm{u_0}_{H^\ell}+\norm{   \tuc_0 }_{L^{p}}+\norm{\tud_0 }_{L^{p^\ast }} + \norm{\tu_{0,1}}_{L^{q }} \leq \varepsilon_0,
\end{equation}
then there exists a unique global solution $u(t,x)$ to the system \eqref{1.1_Sys} satisfying
\begin{equation}\label{th31}
\norm{u(t)}_{H^\ell}^2 + \int_0^t \left( \norm{\tud(\tau)}_{H^\ell}^2+\norm{\Lambda  \tuc(\tau)}_{H^{\ell-1}}^2 \right) \rmd \tau \le  C\widetilde{K}_0^2.
\end{equation}
Moreover, for all $s\in [ 0,   s_3^\ast]\cap [0,d/2+1)$,
\begin{equation}
\norm{\lds\tuc(t) }_{H^{\ell-s}}\le  C \widetilde{K}_0(1+t)^{-\frac{d}{2}\left(\frac{1}{p}-\frac{1}{2}\right)-\frac{s}{2} }
\end{equation}
and for all $s \in [ 0,   s_3^\ast -1]\cap [0,d/2 )$,
\begin{equation}\label{th33}
\norm{\lds\tud(t) }_{H^{\ell-s}}\le  C \widetilde{K}_0(1+t)^{-\frac{d}{2}\left(\frac{1}{p}-\frac{1}{2}\right)-\frac{s}{2}-\frac{1}{2} }.
\end{equation}
\end{thm}

\begin{remark}\label{Rem2.3}%
{\color{black}It is verified in Section \ref{App_C}} that the damping full Euler system of adiabatic flow satisfies all the assumptions \asmp{A1}--\asmp{A4}, \asmp{B} and \asmp{D1}--\asmp{D3} and hence Theorem \ref{thm:1.3} can be applied. This improves the result of \cite{tan2013global} in 3D and produces the first result of global unique solution in 2D which seems necessary to require the low frequency assumption of the initial data. {\color{black}However, our results could not cover the multidimensional thermal non-equilibrium flow \cite{zeng2015gas}.}
\end{remark}


{{Now let us illustrate the main difficulties encountered in proving these results of Theorems \ref{thm:1.1}--\ref{thm:1.3} and explain our strategies to overcome them.
Since the components of the solution behave totally distinctively as already seen in the theorems,
it is essential to divide the system, at least in the linear level, into the dissipative part and the non-dissipative one.
Different from the specific models which usually have certain physical variables to give reasonable ways of doing this,
for the general system considered in this paper these two parts are generally strongly coupled in an implicit manner.
In \cite{liuqu2012JMPA}, the methods of normalized coordinates and formulas of wave decomposition are used to identify the two parts for the one dimensional systems with partial strict dissipation.
However, for our multidimensional systems, these methods,
which are originally designed only for one dimensional problems, are just not applicable.
Meanwhile, when discussing the dissipation structures given by the entropy dissipation condition together with Kawashima condition instead of the strict dissipation, the way of applying these methods would generally destroy the structure to identify the faster decaying variables from the slower ones. To get around this difficulty, the partially normalized coordinates \eqref{2.5_trans} is introduced under the assumption \asmp{B}.
It is shown in Section \ref{sec:2} that by reformulating the system \eqref{1.1_Sys} in terms of this new coordinates into the system \eqref{2.12}, the two parts are successfully identified, $i.e.$, the dissipative part \eqref{2.47_Sys_2n} for $\tub=(\tuc,\tud)$ and the rest part \eqref{2.45_Sys_1} for the non-dissipative $\tu_1$. Also, those structural conditions can be reformulated in terms of new variables in a more applicable way.

The main part of proving Theorems \ref{thm:1.1}--\ref{thm:1.3} is to derive the energy estimates. In the following, we explain the strategy by first dealing with the case under the assumptions of Theorem \ref{thm:1.1}; Theorems \ref{thm:1.2}--\ref{thm:1.3} are proved in the same fashion and we only need to explain the additional arguments for the proof. It is crucial to analyze the structures of the linearized system and the nonlinear terms. It is with the help of the partially normalized coordinates that the linear part of \eqref{2.47_Sys_2n} is shown to be dissipative in the sense of satisfying the entropy dissipation condition and the Kawashima condition. By putting aside the nonlinear effect, one can follow the arguments of \cite{hanouzetnatalini2003ARMA,yong2004ARMA} to derive the basic $L^2$ estimate by using the entropy dissipation condition \asmp{A1}--\asmp{A3} as well as the high order energy estimates for $\tub$ which include the damping dissipation of $\tud$; the partial Kawashima condition \asmp{A4} can be applied to recover the degenerate dissipation of $\tuc$. The remaining difficulty is to handle with the nonlinear terms. It is again with the help of the partially normalized coordinates that under the degenerate assumption \asmp{WD1} there is no nonlinear term resulting from the source term composed solely by the non-dissipative part, that is, those nonlinear terms are of form $O(1)\tu\tub$, at least. This feature provides relatively nice structure so that we can complete the estimate of the dissipative part $\tub$.

The next natural step is to derive the energy estimates for the remaining non-dissipative part $\tu_1$. Under the weakly linearly degenerate assumption \asmp{WD2}, there is no nonlinear term resulting from the convection term in \eqref{2.45_Sys_1} composed solely by the non-dissipative part. However, there would be a loss of derivative if one directly performs the energy estimates by using \eqref{2.45_Sys_1}. To get around this difficulty, we return back to the original full system \eqref{2.12} and resort to generalize the method of wave decomposition (see \cite{john1974blowup}) to our multidimensional case. The isotropy assumption \asmp{B} will be used again. After the delicate analysis of the nonlinear term under the degenerate conditions \asmp{WD1}--\asmp{WD2}, we can finish the full energy estimates as
\begin{equation}\label{llqq}
  \norm{\tu(t)}_{H^\ell}
 \ls \norm{\tu_0}_{H^\ell} +\int_0^t  \norm{     \tub  }_{L^\infty}+\cdots  \rmd\tau.
\end{equation}

Note carefully that we cannot control the time integral in \eqref{llqq} by the dissipation estimates. To bound this time integral, we will turn to derive a strong enough decay rate of $\norm{     \tub  }_{L^\infty}+\cdots$. This step is extremely delicate. We will need to highly explore a refined linear decay estimates which requires weaker low frequency assumption of $\tud_0$ than $\tuc_0$, the stronger degeneracy of $\tq^\mathtt{C}$ than $\tq^\mathtt{D}$ and the faster decay of $\tud$ than $\tuc$. The technical point thus lies in the way of balancing these despairs. Taking into account these facts, by a bootstrap estimates, which implies that the decay of higher order norms can be deduced from the decay of lower order norms, and a use of the Duhamel formula, which implies that the decay of lower order norms is bounded in terms of the decay of higher order norms, we are able to derive a decay of the solution as stated in Theorem \ref{thm:1.1} in a recursive way by the smallness of the solution. It then follows from the Sobolev interpolation that $\norm{     \tub  }_{L^\infty}$ is integrable in time {\color{black} as long as the parameters $p$ and $d$ satisfy the required restriction given in the theorem}. We remark here that the introduction of the fractional derivative in the study enlarges the {\color{black} available} range of $p$ and $d$. This closes the energy estimates and then completes the proof of Theorem \ref{thm:1.1}.

To prove Theorem \ref{thm:1.2}, it is observed that under the stronger degenerate conditions \asmp{D1}--\asmp{D2}, the nonlinear terms behave better. For example, under \asmp{D1}, nonlinear terms resulting from the source term are of the form $O(1)\tu\tud$, or so. Based on the stronger degeneracy, we can improve the full energy estimates as
\begin{equation}
  \norm{\tu(t)}_{H^\ell}
 \ls \norm{\tu_0}_{H^\ell} +\int_0^t  \norm{     \tud  }_{L^\infty}+\cdots  \rmd\tau.
\end{equation}
Since $\tud$ decays at a faster $1/2$ rate than $\tuc$, this ultimately enlarges the range of $p$ and $d$ as stated in Theorem \ref{thm:1.2}. Note that we can mostly deduce from \eqref{th23} that $\norm{     \tud}_{L^\infty}$ decays at the rate of $(1+t)^{-{d}/4 - {1}/{2}+\varepsilon }$ for any $\varepsilon>0$, which restricts Theorem \ref{thm:1.2} only hold for $d\ge 3$.

The main goal of Theorem \ref{thm:1.3} is to include the case $d=2$. Note that even in view of the linear decay estimates $\norm{     \tud  }_{L^\infty}$ decays at the rate of $(1+t)^{-{d}/{(2p)} - {1}/{2} }$, and then we need to require the low frequency assumption of $\tub_0$ when $d=2$. However, due to the nonlinear effect we cannot achieve this rate, and Theorem \ref{thm:1.2} basically tells that $\norm{     \tud  }_{L^\infty}$ would decay at the same rate of $(1+t)^{-{d}/4 - {1}/{2}+\varepsilon }$ for all $p$ if  there is no further additional assumption. Thereby, the $L^q$ assumption of $ \tu_{0,1}$ comes into the role. The idea is to derive the $L^q$ estimate of $\tu_1$ so that we can improve a bit some nonlinear estimates when using the Duhamel formula, for example the estimates of $\norm{\tu_1\tud}_{L^p}$. However, we can not derive the $L^q$ estimate of $\tu_1$ by using \eqref{2.45_Sys_1}, and we again need to return back to the original full system \eqref{2.12} and introduce the wave decomposition to consider instead the first wave $\tilde{v}_1$. It is the derivation of the $L^q$ estimate of $\tilde{v}_1$ that we need the assumption \asmp{D3}. With the help of the $L^q$ estimate of $\tilde{v}_1$, we can improve the nonlinear estimates so that we can derive the decay rates of the solution as stated in Theorem \ref{thm:1.3}. Note that now the decay rate of $\norm{     \tud}_{L^\infty}$ is improved to be $(1+t)^{-{d}/(2q) - {1}/{2}+\varepsilon}$, which allows the case $d=2$ for $q<2$. One may then notice the necessity of the $L^q$ assumption of $ \tu_{0,1}$.

 The rest of the paper is organized as follows. In Section \ref{sec:2}, partially normalized coordinates will be used to reformulate the system \eqref{1.1_Sys} and those structural conditions.  In Section \ref{sec:4} the structure of the reformulated system is analyzed carefully as the foundation for the further analysis.
The delicate energy estimates is performed in Sections \ref{en1111111}--\ref{en2222222},  the low frequency estimates for two dimensional case is presented in Section \ref{sec:6+}, and the refined decay estimates is given in Section \ref{en3333333}. Finally, those estimates are combined in Section \ref{999} to prove Theorems \ref{thm:1.1}--\ref{thm:1.3}.}}


\section{Partially normalized coordinates} \label{sec:2}

After a linear transformation if necessary (see Appendix \ref{linear trans} for details), one may assume
\begin{equation} \label{1.10}
\frac{\p Q_i}{\p u_{q}}(0) = 0, \hs q = 1, \dots, r; \ i = 1,\dots, n,
\end{equation}
and
\begin{equation} \label{1.11}
r_1(0) = e_1,
\end{equation}
where $e_i\, (i=1,\dots,n)$ stands for the $i$th unit vector in $\Rn{n}$.

The first step is to verify that the nonlinear mapping $u=u(\tu)$ in the phase space, defined by \eqref{2.5_trans} and \eqref{1.13}, is invertible near $\tu=0$.
Denote the Jacobian matrix of the transformation $u=u(\tu)$ by
\begin{equation}
J(\tilde{u}) := \frac{\partial u}{\partial \tilde{u}}(\tilde{u}).
\end{equation}
By the definition, it holds
\begin{equation} \label{2.7_ub}
u \ubb = u^{(1)}(0, \ubz)= \ubz
\end{equation}
and
\begin{equation}\label{2.8}
\frac{\partial u}{\partial \tilde{u}_1}(\tilde{u}) \equiv r_1(u(\tilde{u})).
\end{equation}
Hence,
\begin{equation}
J \ubb = \left( r_1 \ubb, e_2, \dots, e_n \right). \label{2.9}
\end{equation}
This together with \eqref{1.11} yields
\begin{equation}
J(0) = \left( r_1 (0), e_2, \dots, e_n \right)= I_n, \label{2.3}
\end{equation}
which implies that the transformation $u = u(\tilde u)$ is a local diffeomorphism near $\tu=0$. In the partially normalized coordinates $\tu$ of the phase space, the system \eqref{1.1_Sys} {\color{black}(and \eqref{sys}) is transformed into
\begin{equation} \label{2.12}
\begin{cases}
\displaystyle
\partial_t \tilde{u} + \sum_{k = 1}^{d} \widetilde{A}^k(\tilde{u}) \tilde u_{x_k} = \widetilde{Q}(\tilde{u}),
\\
 \tilde{u}\mid_{t=0} = \tilde{u}_0.
\end{cases}
\end{equation}
Here the coefficient matrices
\begin{equation} \label{3.11}
\ta^k(\tu) = J^{-1}(\tilde{u}) A(u(\tilde{u})) J(\tilde{u}), \hs k=1,\dots,d
\end{equation}
and the inhomogeneous term}
\begin{equation}
 {\color{black}\widetilde{Q}(\tilde{u}) = J^{-1}(\tilde{u}) Q(u(\tilde{u})).}  \label{2.14}
\end{equation}

Now it is important to study the structure of $\widetilde{A}^k(\tilde{u})$, which will exhibit the advantage of our partially normalized coordinates $\tu$.
First of all, it is direct to check that $\ta^k(\tu)$ possesses the following expression
\begin{equation}
\widetilde{A}^k(\tilde{u}) = \widetilde{A}(\tilde{u},  \mathtt{e}_k),\quad k=1,\dots, d
\end{equation}
with $\sek$ the $k$th unit vector in $\rd$ and
\begin{equation} \label{2.15}
\widetilde{A}(\tilde{u},\omega)  := J^{-1}(\tilde{u}) A(u(\tilde{u}), \omega) J(\tilde{u}),\ \omega \in \mathbb{S}^{d-1},
\end{equation}
where the matrix $A(u, \omega)$ is defined by \eqref{A matrix}.
Therefore, it suffices to study the structure of $\widetilde{A}(\tilde{u},\omega)$. Recalling that $l_i(u,\omega)$ (\resp $r_i(u,\omega)$) is the $i$th left (\resp right) eigenvector of $A(u,\omega)$ with the corresponding eigenvalue $\lambda_i(u,\omega)$, $i=1,\dots,n$, one could then choose
\begin{align}
\tilde{\lambda}_i(\tilde{u},\omega) & :=  \lambda_i(u(\tilde{u}),\omega), \label{2.22}\\
\tilde{r}_i(\tilde{u},\omega) & = \left( \tilde{r}_{1i}(\tilde{u}), \dots, \tilde{r}_{ni}(\tilde{u}) \right)^T := J^{-1}(\tilde{u}) r_i(u(\tilde{u}),\omega),  \label{2.21}  \\
\tilde{l}_i(\tilde{u},\omega) & = \left( \tilde{l}_{i1}(\tilde{u}), \dots, \tilde{l}_{in}(\tilde{u}) \right):= l_i(u(\tilde{u}),\omega) J(\tilde{u})   \label{2.23}
\end{align}
as the eigenvalues and a complete set of left (\resp right) eigenvectors of $\ta(\tu,\omega)$, that is,
\begin{align}
&\tilde{l}_i(\tilde{u}, \omega) \widetilde{A}(\tilde{u},\omega) = \tilde{\lambda}_i(\tilde{u},\omega) \tilde{l}_i(\tilde{u},\omega),\hs i=1,\dots,n,\\
&\widetilde{A}(\tilde{u},\omega) \tilde{r}_i(\tilde{u}, \omega) = \tilde{\lambda}_i(\tilde{u},\omega) \tilde{r}_i(\tilde{u},\omega), \hs i=1,\dots,n.\label{tezhen}
\end{align}
From the normalization \eqref{1.4}, it follows
\begin{equation} \label{2.28}
\tilde{l}_i(\tilde{u},\omega) \tilde{r}_{i'}(\tilde{u},\omega) \equiv \delta_{ii'}, \hs    i,i' =1,\dots, n.
\end{equation}
Moreover, by the assumption \asmp{B} and \eqref{2.8}, respectively, for any $\omega \in \mathbb{S}^{d-1}$,
\begin{equation} \label{2.330}
\tilde{l}_1(\tilde{u},\omega)=l_1(u(\tilde{u})) J(\tilde{u}) \equiv   \tilde{l}_1(\tilde{u})
\end{equation}
and
\begin{equation}
\tilde{r}_1(\tilde{u},\omega)=J^{-1}(\tilde{u}) r_1(u(\tilde{u})) \equiv  e_1. \label{2.33}
\end{equation}
Combining \eqref{2.28}--\eqref{2.33}, it follows
\begin{equation}
\tilde{l}_{11}(\tilde{u}) \equiv  1 \text{ and }
\tilde{l}_{j1}(\tilde{u},\omega) \equiv   0,\hs j=2,\dots,n,\ \forall\, \omega\in \mathbb{S}^{d-1}. \label{2.35}
\end{equation}
More importantly, taking $i = 1$ in \eqref{tezhen} and noting \eqref{2.33}, leads to
\begin{equation}
\widetilde{A}_{11}(\tilde{u},\omega) = \tilde{\lambda}_1(\tilde{u},\omega)  \text{ and }
\widetilde{A}_{j1}(\tilde{u},\omega) \equiv 0,\hs  j=2,\dots,n, \ \forall\, \omega \in \mathbb{S}^{d-1}.\label{2.40}
\end{equation}
Throughout the paper the following notational convention are used for matrices: for an $n\times n$ matrix $M=\left(M_{ii'}\right)_{i,i'=1}^n$, denote
\begin{equation}
 {M}^\sharp = (M_{12},\dots,M_{1n}),\  M^\flat=\left(M_{jj'}\right)_{j,j'=2}^n  \text{ and } M^{\mathtt{D}} = \left(M_{pp'}\right)_{p,p'=r+1}^n.
\end{equation}
Then by \eqref{2.40}, one has the following structure of $\widetilde{A}(\tilde{u},\omega)$:
\begin{equation} \label{2.42}
\widetilde{A}(\tilde{u},\omega) = \begin{pmatrix}
\tilde{\lambda}_1(\tilde{u},\omega) & \widetilde{A}^\sharp(\tilde{u},\omega) \\
0 & \widetilde{A}^\flat(\tilde{u},\omega)
\end{pmatrix}.
\end{equation}
Accordingly, the system \eqref{2.12} indeed takes the form
\begin{align}
 &\partial_t \tilde{u}_1 + \sum_{k=1}^{d} \tilde{\lambda}_1(\tilde{u}, \mathtt{e}_k)   (\tilde{u}_1)_{x_k} + \sum_{k=1}^{d} \widetilde{A}^{k,\sharp}(\tilde{u})  \tilde{u}^\flat_{x_k} = \widetilde{Q}_1(\tilde{u}), \label{2.45_Sys_1}
\\\label{2.47_Sys_2n}
&\partial_t \tilde{u}^\flat + \sum_{k=1}^{d} \widetilde{A}^{k,\flat}(\tilde{u})  \tilde{u}^\flat_{x_k} = \widetilde{Q}^\flat(\tilde{u}).
\end{align}
Note that the system \eqref{2.12} is split into two subsystems in some sense.
We will see that the subsystem \eqref{2.47_Sys_2n} of $\tub$ is dissipative in the sense of satisfying the entropy dissipation condition and the Kawashima condition and the main part of the subsystem \eqref{2.45_Sys_1} of the non-dissipative $\tu_1$ is indeed (weakly) linearly degenerate.
This structure property is crucial for the whole analysis in this paper.

Recall that from the assumption \asmp{B}, the isotropy \eqref{2.330}--\eqref{2.33} hold in the partially normalized coordinates, and we restate them as follows
\begin{itemize}
\item[\asmp{\tb}] The first family of eigenvectors satisfy an isotropic condition:
\begin{equation}\label{1.12_iso_r2}
\tilde{l}_1(\tilde{u},\omega)\equiv   \tilde{l}_1(\tilde{u}) \text{ and }
\tilde{r}_1(\tilde{u},\omega) \equiv   e_1, \hs \forall\, \omega \in \mathbb{S}^{d-1}.
\end{equation}
\end{itemize}
The next task is to derive the corresponding forms of the assumptions \asmp{A1}--\asmp{A4}, \asmp{WD1}--\asmp{WD2}, \asmp{D1}--\asmp{D2} and \asmp{D3}.
First, recalling \eqref{2.14} and \eqref{1.2} it trivially holds that $\widetilde{Q}(0) = 0$ and
\begin{equation} \label{2.270}
\sum_{i=1}^{n} J_{qi}(\tu) \tq_i(\tu) \equiv 0,\hs q=1,\dots, r.
\end{equation}
Moreover, by \eqref{2.14} and \eqref{2.3}, one has
\begin{equation} \label{2.28+}
\frac{\p\tq}{\p \tu_i}(0) = \left. \left( \frac{\p J^{-1}}{\p \tu_i}(\tu) \right)\right|_{\tu = 0} Q(0) + J^{-1}(0)  \left. \left( \sum_{i'=1}^{n}  \frac{\p u_{i'}}{\p \tu_i} \frac{\p Q}{\p u_{i'}} (u(\tu)) \right)\right|_{\tu = 0} = \frac{\p Q}{\p u_i}(0).
\end{equation}
Hence the assumption \asmp{A1} is translated to be
\begin{itemize}
\item[\asmp{\ta 1}]  The matrix $\left( \ntu \tq\right)^\mathtt{D}(0)$ is invertible.
\end{itemize}

Next,  {\color{black}set}
\begin{equation} \label{entropy}
\tilde{\eta}(\tilde{u}) = \eta(u(\tilde{u})),
\end{equation}
then 
\begin{equation} \label{2.55}
\tilde{\eta}(0) = 0\text{ and } \ntu \tilde{\eta}(0) = 0.
\end{equation}
Moreover, since
\begin{equation}
\frac{\partial^2 \tilde{\eta}}{\partial \tilde{u}_i \partial \tilde{u}_{i'}}(\tilde{u}) = \sum_{r,r'=1}^{n} \frac{\partial^2 \eta}{\partial u_r \partial u_{r'}}(u(\tilde{u})) \frac{\partial u_r}{\partial \tilde{u}_i}(\tilde{u}) \frac{\partial u_{r'}}{\partial \tilde{u}_{i'}}(\tilde{u}) + \sum_{r=1}^{n} \left( \frac{\partial \eta}{\partial u_r}(u(\tilde{u})) - \frac{\partial \eta}{\partial u_r}(0)  \right)  \frac{\partial^2 u_r}{\partial \tilde{u}_i \partial \tilde{u}_{i'}}(\tilde{u}),
\end{equation}
by \eqref{2.3} it holds
\begin{equation} \label{2.53}
\ntu^2 \tilde{\eta}(0) = J^T(0) \nuu^2 \eta(0) J(0) = \nuu^2 \eta(0).
\end{equation}
Set 
$\tilde{\psi}^k(\tu) = \psi^k(u(\tu))$,
then the assumption \asmp{A2} leads to
\begin{itemize}
\item[\asmp{\ta 2}] The strictly convex smooth entropy function $\tilde{\eta}(\tu)$ and the $d$ smooth entropy flux functions $\tilde{\psi}^k(\tu)\, (k=1,\dots,d)$ satisfy that for all $\tu$,
\begin{equation}\label{2.32}
{\color{black}(\ntu \tilde{\psi}^k(\tilde{u}))^T = (\ntu \tilde{\eta}(\tilde{u}))^T\ta^k(\tu).}
\end{equation}
\end{itemize}

Now, noting \eqref{2.3}, \eqref{2.14}, \eqref{2.21} and \eqref{2.28+}, it follows from the assumptions \asmp{A3} and \asmp{A4} that
\begin{itemize}
\item[\asmp{\ta 3}] There exists a positive constant $\tilde c_e > 0$ such that for all $\tu$,
\begin{equation} \label{2.56}
\ntu \tilde{\eta} (\tilde{u}) \cdot \widetilde{Q}(\tilde{u}) \leq -\tilde{c}_e |\widetilde{Q}(\tilde{u})|^2.
\end{equation}
\item[\asmp{\ta 4}] The last $n-1$ right eigenvectors of $\ta(0,\omega)$ are
not contained in the kernel of the Jacobian $\ntu \widetilde{Q}(0)$ for any $\omega \in \mathbb{S}^{d-1}$, that is,
\begin{equation} \label{2.57}
\ntu \widetilde{Q}(0) \tilde{r}_j (0,\omega) \neq 0,\hs j=2,\dots,n,
\end{equation}
but
\begin{equation*}
\ntu \tq(0) \trr_1(0) = 0.
\end{equation*}
\end{itemize}

For the assumptions \asmp{WD1}--\asmp{WD2},
we notice that the definition \eqref{2.5_trans} directly leads to
\begin{equation}
u(\tu_1 e_1) = \utra{\tu_1}{0}.
\end{equation}
Then by \eqref{2.14}, \eqref{2.22} and \eqref{1.12_iso_r2}, one can express them as
\begin{itemize}
\item[\asmp{\twd 1}] $\widetilde{Q}(\tilde{u})$ is degenerate along the first characteristic trajectory passing through $0$:
\begin{equation} \label{2.52+}
\frac{\partial(J\tq)}{\partial \tilde{u}_1} (\tu_1 e_1) \equiv 0,\quad \forall\, |\tu_1|\text{ small}.
\end{equation}
\item[\asmp{\twd 2}] The first eigen-family is weakly linearly degenerate:
\begin{equation} \label{2.65+}
\frac{\p \tilde{\lambda}_1}{\p \tu_1}(\tu_1e_1,\omega)=0,\quad \forall\, |\tu_1|\text{ small},\ \forall\,\omega \in \mathbb{S}^{d-1} .
\end{equation}
\end{itemize}
Similarly, the assumptions \asmp{D1}--\asmp{D2} are translated to be
\begin{itemize}
\item[\asmp{\td 1}] $\widetilde{Q}(\tilde{u})$ is degenerate along \emph{each} first characteristic trajectory:
\begin{equation} \label{2.52++}
\frac{\partial(J\tq)}{\partial \tilde{u}_1} (\tu )\equiv   0,\quad \forall\,\tu.
\end{equation}
\item[\asmp{\td 2}] The first eigen-family is linearly degenerate:
\begin{equation} \label{2.65++}
\frac{\p \tilde{\lambda}_1}{\p \tu_1}(\tilde{u},\omega)  \equiv 0,\quad \forall\,\tu,\ \forall\, \omega \in \mathbb{S}^{d-1}.
\end{equation}
\end{itemize}

At last, by \eqref{2.3}, \eqref{2.330} and \eqref{2.28+}, the assumption \asmp{D3} can be expressed as
\begin{itemize}
\item[\asmp{\td 3}] The first eigen-family gets no linear effect from the source term
\begin{equation} \label{C3t}
\tll_1(0) \ntu \tq(0) = 0.
\end{equation}
\end{itemize}

\begin{remark}
Since the assumptions \asmp{A1}--\asmp{A4} and \asmp{B} are made in all the three Theorems \ref{thm:1.1}--\ref{thm:1.3},
\asmp{\ta 1}--\asmp{\ta 4} and \asmp{\tb} are always supposed to be hold in the rest of the paper.
\end{remark}

\section{Structural analysis} \label{sec:4}

In this section, useful algebraic consequences are established for the structural assumptions \asmp{\ta 1}--\asmp{\ta 4}, \asmp{\twd 1} and \asmp{\td 1}.

\subsection{Structure of the source term $\tq$}
In this subsection the source term $\widetilde{Q}(\tilde{u})$ is analyzed under the assumption \asmp{\twd 1} or \asmp{\td 1} by the Taylor expansion in an appropriate way.
\subsubsection{The expansion of $\tq(\tu)$ under \asmp{\twd 1}: }
First, assume \asmp{\twd 1}, then it holds that $\widetilde{Q}(\tilde{u}_1 e_1) \equiv J^{-1}(\tu_1e_1)J(0)\tq(0)= 0$.
Using this fact, one may first expand $\widetilde{Q}(\tilde{u})$ around $\tu_1 e_1$, that is, fixing $\tu_1$ and viewing $\widetilde{Q}(\tilde{u})$ as a function of $\tub=(\tu_2,\dots,\tu_n)^T$ and then expanding it around $\tub=0$, to have
\begin{equation}\label{eexx1}
\widetilde{Q}_i(\tilde{u}) = \sum_{j'=2}^{n} \frac{\partial \widetilde{Q}_i}{\partial \tilde{u}_{j'}}(\tilde{u}_1 e_1) \tilde{u}_{j'} + \sum_{j',j''=2}^{n} \ty_{ij'j''}(\tu) \tilde{u}_{j'} \tilde{u}_{j''},\hs i=1,\dots,n.
\end{equation}
Expanding further the functions of $\tu_1$ inside the first summation in \eqref{eexx1} around $\tu_1=0$, one can obtain
\begin{equation}\label{eexx2}
\begin{split}
\widetilde{Q}_i(\tilde{u}) = & \sum_{j'=2}^{n} \frac{\partial \widetilde{Q}_i}{\partial \tilde{u}_{j'}}(0) \tilde{u}_{j'} + \sum_{j'=2}^{n} \frac{\partial^2 \widetilde{Q}_i}{\partial \tilde{u}_1\partial \tilde{u}_{j'}}(0) \tilde{u}_{j'} \tilde{u}_1 + \sum_{j',j''=2}^{n} \ty_{ij'j''}(\tu) \tilde{u}_{j'} \tilde{u}_{j''} \\
& + \sum_{ j' =2}^n \widetilde{Z}_{ij'11}(\tilde u) \tilde u_{j'} \tilde u_1^2,\hs i=1,\dots,n.
\end{split}
\end{equation}

Now examine the first two summations in \eqref{eexx2}. First, it follows from \eqref{2.28+}, \eqref{1.10} and \eqref{1.2} that
\begin{equation}
\frac{\partial \widetilde{Q}_i}{\partial \tilde{u}_q}(0) = \frac{\partial \widetilde{Q}_q}{\partial \tilde{u}_i}(0) = 0, \hs  q=1,\dots,r;\ i=1,\dots, n. \label{2.50}
\end{equation}
On the other hand, by \eqref{2.14}, it holds
\begin{equation}
\begin{split}
\frac{\p^2 \tq}{\p \tu_{i'} \p \tu_i}(\tu) = &   \frac{\p^2 J^{-1}}{\p \tu_{i'} \p \tu_i}(\tu)   J(\tu) \tq(\tu)  +   \frac{\p J^{-1}}{\p \tu_{i'} }(\tu)   \frac{\p J}{\p \tu_i}(\tu)   \tq(\tu) +  \frac{\p J^{-1}}{\p \tu_{i'}}(\tu)   J(\tu)  \frac{\p \tq}{\p \tu_i}(\tu)   \\
& +   \frac{\p J^{-1}}{\p \tu_i }(\tu)   \frac{\p J}{\p \tu_{i'}}(\tu)  \tq(\tu) +  \frac{\p J^{-1}}{\p \tu_i}(\tu)  J(\tu)  \frac{\p\tq}{\p \tu_{i'}}(\tu) + J^{-1}(\tu)  \frac{\p^2\left( Q(u(\tu)) \right)}{\p \tu_{i'} \p \tu_i} .
\end{split}
\end{equation}
Hence, by the fact $\tq(0) = 0$, \eqref{2.50}, \eqref{2.3} and \eqref{1.2}, one has
\begin{equation}\label{upup}
 \frac{\p^2 \tq_{q}}{\p \tu_{q'} \p \tu_{q''}}(0) =  \left.   \frac{\p^2}{\p \tu_{q'} \p \tu_{q''}} \left( Q_{q}(u(\tu))\right)  \right|_{\tu = 0} = 0,\quad q,q', q''=1,\dots r.
\end{equation}
Denote the constant matrix
\begin{gather}\label{theta1}
{\Theta}  = \nabla_{\tu}\tq(0)
\end{gather}
and for $i,i',i'' = 1,\dots, n$, 
\begin{gather}\label{theta111}
{\Upsilon}_{ii'i''} =   \frac{\partial^2 \widetilde{Q}_i}{\partial \tilde{u}_{i'} \partial \tilde{u}_{i''}}(0).
\end{gather}
Then \eqref{2.50} and \eqref{upup} imply that
\begin{gather}
\Theta_{qi} = \Theta_{iq} = 0, \hs   i  =1,\dots, n;\ q =1,\dots, r, \label{3.10++22}
\\ \label{theta}
{\Upsilon}_{qq'q''} = 0, \quad q, q',q''=1,\dots r.
\end{gather}
Accordingly, it follows from \eqref{eexx2} that
\begin{equation}
\label{for decay1}
\widetilde{Q}_q(\tilde{u}) =   \sum_{ p' =r+1 }^n \Upsilon_{qp'1} \tilde{u}_{p'}\tilde{u}_1 + \sum_{j',j''=2}^{n} \ty_{qj'j''}(\tu) \tilde{u}_{j'} \tilde{u}_{j''}  + \sum_{ j' =2}^n \widetilde{Z}_{qj'11}(\tilde u) \tilde u_{j'} \tilde u_1^2, \hs q=1,\dots, r
\end{equation}
and
\begin{equation}\label{for decay2}
\widetilde{Q}_p(\tilde{u}) =   \sum_{p'=r+1}^n \Theta_{pp'} \tilde u_{p'} + \sum_{j'=2}^n \sum_{i''=1}^{n} \ty_{pj'i''}(\tu) \tilde{u}_{j'} \tilde{u}_{i''}  , \hs p= r+1,\dots, n,
\end{equation}
where $\ty_{pj'1}(\tu) = \Upsilon_{p1j'} + \tz_{pj'11}(\tu) \tu_1$.

\begin{remark}
By the matrix $\Theta$ defined in \eqref{theta1}, \asmp{\ta 1} is equivalent to
\begin{equation}
\det \Theta^\mathtt{D}\neq 0. \label{3.10++}
\end{equation}
\end{remark}

\subsubsection{The expansion of $\tq(\tu)$ under \asmp{\td 1}}

Since  \asmp{\td 1} implies \asmp{\twd 1}, the expansions \eqref{for decay1}--\eqref{for decay2} and all the computations in the previous subsubsection hold in this case.
But under the stronger assumption \asmp{\td 1}, it is reasonable to expect a better expansion.

Again, using the fact $\widetilde{Q}(\tilde{u}_1 e_1) \equiv 0$, one may first expand $\widetilde{Q}(\tilde{u})$ around $\tu_1 e_1$ to have
\begin{equation}\label{eexx112}
\begin{split}
\widetilde{Q}_i(\tilde{u}) = &\sum_{j'=2}^{n} \frac{\partial \widetilde{Q}_i}{\partial \tilde{u}_{j'}}(\tilde{u}_1 e_1) \tilde{u}_{j'} + \sum_{j',j''=2}^{n} \frac{1}{2} \frac{\partial^2 \widetilde{Q}_i}{\partial \tilde{u}_{j'} \partial \tilde{u}_{j''}}(\tilde{u}_1 e_1) \tilde{u}_{j'} \tilde{u}_{j''} \\&+ \sum_{j',j'',j^* = 2}^n \widetilde{Z}'_{ij'j''j^*}(\tilde u) \tilde u_{j'} \tilde u_{j''} \tilde u_{j^*},\hs i=1,\dots,n.
\end{split}
\end{equation}
Due to the assumption \asmp{\td 1} and the fact $\tq(\tu_1 e_1) =\tq(0) = 0$, it holds
\begin{equation}\label{hhhh122}
\begin{split}
 \tq(\tu_1 e_1 + \tu_q e_q) - \tq(\tu_1 e_1)
= & J^{-1} (\tu_1 e_1 + \tu_q e_q) J (\tu_q e_q) \tq(\tu_q e_q)  \\
= & J^{-1} (\tu_1 e_1 + \tu_q e_q) J (\tu_q e_q) \left( \tq(\tu_q e_q) - \tq(0) \right) .
\end{split}
\end{equation}
This together with \eqref{2.50} yields
\begin{equation}
\frac{\p \tq}{\p \tu_q}(\tu_1 e_1) = J^{-1}(\tu_1 e_1) J(0) \frac{\p \tq}{\p \tu_q}(0) = 0, \hs   q =1,\dots, r.
\end{equation}
Thus, from \eqref{eexx112}, one can get
\begin{equation}\label{eexx3}
\begin{split}
\widetilde{Q}_i(\tilde{u}) = &\sum_{p'=r+1}^{n} \frac{\partial \widetilde{Q}_i}{\partial \tilde{u}_{p'}}(\tilde{u}_1 e_1) \tilde{u}_{p'} + \sum_{j',j''=2}^{n} \frac{1}{2} \frac{\partial^2 \widetilde{Q}_i}{\partial \tilde{u}_{j'} \partial \tilde{u}_{j''}}(\tilde{u}_1 e_1) \tilde{u}_{j'} \tilde{u}_{j''} \\&+ \sum_{j',j'',j^* = 2}^n \widetilde{Z}'_{ij'j''j^*}(\tilde u) \tilde u_{j'} \tilde u_{j''} \tilde u_{j^*},\hs i=1,\dots,n.
\end{split}
\end{equation}
Expanding further the functions of $\tu_1$ inside the first two summations in \eqref{eexx3} around $\tu_1=0$, one has
\begin{equation}\label{eexx2123}
\begin{split}
\widetilde{Q}_i(\tilde{u}) = & \sum_{p'=r+1}^{n} \frac{\partial \widetilde{Q}_i}{\partial \tilde{u}_{p'}}(0) \tilde{u}_{p'} + \sum_{p'=r+1}^{n} \ty'_{ip'1}(\tu) \tilde{u}_{p'} \tilde{u}_{1} \\&  + \frac{1}{2} \sum_{j',j''=2}^{n} \frac{\partial^2 \widetilde{Q}_i}{\partial \tilde{u}_{j'} \partial \tilde{u}_{j''}}(0) \tilde{u}_{j'} \tilde{u}_{j''} + \sum_{j',j'' =2}^n \sum_{i^* =1}^{n} \widetilde{Z}'_{ij'j''i^*}(\tilde u) \tilde u_{j'} \tilde u_{j''} \tilde u_{i^*} ,\hs i=1,\dots,n.
\end{split}
\end{equation}
Recall $\Theta$ from \eqref{theta1} and $\Upsilon$ from \eqref{theta111}, and note \eqref{3.10++22} and \eqref{theta}, then it holds
\begin{equation}\label{eexx4}
\widetilde{Q}_q(\tilde{u}) =   \sum_{p'=r+1}^n  \sum_{i'' =1}^{n} \ty'_{qp'i''}(\tu) \tilde{u}_{p'} \tilde{u}_{i''}  +  \sum_{j',j'' =2}^n \sum_{i^* =1}^{n}  \widetilde{Z}'_{qj'j''i^*}(\tilde u) \tilde u_{j'} \tilde u_{j''} \tilde u_{i^*},\hs q=1,\dots,r
\end{equation}
and
\begin{equation}\label{eexx44}
\widetilde{Q}_p(\tilde{u}) =  \sum_{p'=r+1}^{n} \Theta_{pp'} \tilde{u}_{p'}  + \sum_{p'=r+1}^{n} \ty'_{pp'1}(\tu) \tilde{u}_{p'} \tilde{u}_{1} + \sum_{j',j''=2}^{n} \ty'_{pj'j''}(\tu) \tilde{u}_{j'} \tilde{u}_{j''} ,\hs p=r+1,\dots,n,
\end{equation}
where $\ty'_{qp'j''}(\tu) = \hal \Upsilon_{qp'j''}$ and $\ty'_{pj'j''}(\tu) = \hal \Upsilon_{pj'j''} + \sum_{i^*=1}^{n} \tz'_{pj'j''i^*}(\tu) \tu_{i^*}$.

\subsection{Structure of the subsystem \eqref{2.47_Sys_2n}}

In this subsection the subsystem \eqref{2.47_Sys_2n} is analyzed under the assumptions \asmp{\ta 1}--\asmp{\ta 4}.

\subsubsection{Symmetrizer}
The first task is to search for a positive definite $(n-1) \times (n-1)$ symmetric matrix to symmetrize the subsystem \eqref{2.47_Sys_2n}.
Indeed,
\begin{equation} \label{3.1}
\widetilde{A}^{0,\flat}(\tilde{u}) := \left( \frac{\partial^2 \tilde{\eta}}{\partial \tilde{u}_j \partial \tilde{u}_{j'}} (\tilde{u}) - \sum_{i,i'=1}^{n} \frac{\partial \tilde{\eta}}{\partial \tilde{u}_i}(\tilde{u})
\left( \frac{\partial G}{\partial \tu} \right)^{-1}_{ii'}
\frac{\partial^2 G_{i'}}{\partial \tilde{u}_j \partial \tilde{u}_{j'}}(\tilde{u}) \right)_{j,j'=2}^n
\end{equation}
is just the desired symmetrizer.
\begin{lem}\label{a0b}
$\widetilde{A}^{0,\flat}(\tilde{u})$ is symmetric and positive definite near $\tilde{u} = 0$. Moreover,
\begin{equation} \label{3.2}
\widetilde{A}^{0,\flat}(\tilde{u}) \widetilde{A}^{k,\flat}(\tilde{u})  = \widetilde{A}^{k,\flat}(\tilde{u})^T \widetilde{A}^{0,\flat}(\tilde{u}), \hs k=1,\dots,d.
\end{equation}
\end{lem}
\begin{proof}
It is direct to check that $\widetilde{A}^{0,\flat}(\tilde{u})$ is symmetric.
Meanwhile, by \eqref{2.55}, it holds
\begin{equation} \label{3.2+}
\ta^{0,\flat}(0) = \left( \ntu^2 \tilde{\eta}(0) \right)^\flat.
\end{equation}
On the other hand, according to the assumption \asmp{\ta 2}, $\tilde{\eta}$ is strictly convex, then it follows from \eqref{3.2+} that $\widetilde{A}^{0,\flat}(\tilde{u})$ is positive definite near $\tilde{u} = 0$.
Moreover,
it follows from \eqref{3.11} 
and  \eqref{2.32} that, {\color{black}setting $\widetilde{F}^k(\tu) = F^k(u(\tu))$,}
{\allowdisplaybreaks
\begin{align*}
0 = & \frac{\p }{\p \tu_{i'}} \left( \frac{\p \tilde{\psi}^k }{\p \tu_i} \right) - \frac{\p }{\p \tu_{i}} \left( \frac{\p \tilde{\psi}^k }{\p \tu_{i'}} \right) \\
= & \sum_{r,l=1}^{n} \left( \frac{\p^2 \teta}{\p \tu_{i'} \p \tu_r}
{\left( \frac{\partial G}{\partial \tu} \right)^{-1}_{rl} } 
\frac{\p \widetilde{F}^k_l}{\p \tu_i} + \frac{\p \teta}{\p \tu_r} \frac{\p}{\p \tu_{i'}} \left( {\left( \frac{\partial G}{\partial \tu} \right)^{-1}_{rl} } \right) \frac{\p \widetilde{F}^k_l}{\p \tu_i} +  \frac{\p \teta}{\p \tu_r} {\left( \frac{\partial G}{\partial \tu} \right)^{-1}_{rl} } \frac{\p^2 \widetilde{F}^k_l}{\p \tu_i \p \tu_{i'}} \right) \\
& - \sum_{r,l=1}^{n} \left( \frac{\p^2 \teta}{\p \tu_{i} \p \tu_r} {\left( \frac{\partial G}{\partial \tu} \right)^{-1}_{rl} } \frac{\p \widetilde{F}^k_l}{\p \tu_{i'}} + \frac{\p \teta}{\p \tu_r} \frac{\p}{\p \tu_{i}} \left( {\left( \frac{\partial G}{\partial \tu} \right)^{-1}_{rl} } \right) \frac{\p \widetilde{F}^k_l}{\p \tu_{i'}} +  \frac{\p \teta}{\p \tu_r}  {\left( \frac{\partial G}{\partial \tu} \right)^{-1}_{rl} }  \frac{\p^2 \widetilde{F}^k_l}{\p \tu_{i'} \p \tu_{i}} \right) \\
= & {\sum_{r,r',l=1}^{n} \left( \frac{\p^2 \teta}{\p \tu_{i'} \p \tu_r} {\left( \frac{\partial G}{\partial \tu} \right)^{-1}_{rl} } \frac{\p G_l}{\p \tu_{r'}} \ta_{r'i}^k + \frac{\p \teta}{\p \tu_r} \frac{\p}{\p \tu_{i'}} \left( {\left( \frac{\partial G}{\partial \tu} \right)^{-1}_{rl} } \right) \frac{\p G_l}{\p \tu_{r'}} \ta_{r'i}^k \right)} \\
& - {\sum_{r,r',l=1}^{n} \left( \frac{\p^2 \teta}{\p \tu_{i} \p \tu_r} {\left( \frac{\partial G}{\partial \tu} \right)^{-1}_{rl} } \frac{\p G_l}{\p \tu_{r'}} \ta_{r'i'}^k + \frac{\p \teta}{\p \tu_r} \frac{\p}{\p \tu_{i}} \left( {\left( \frac{\partial G}{\partial \tu} \right)^{-1}_{rl} } \right) \frac{\p G_l}{\p \tu_{r'}} \ta_{r'i'}^k \right)} \\
= & \left( \sum_{r=1}^{n} \frac{\p^2 \teta}{\p \tu_{i'} \p \tu_r} \ta_{ri}^k - \sum_{r,r',l=1}^{n} \frac{\p \teta}{\p \tu_r} {\left( \frac{\partial G}{\partial \tu} \right)^{-1}_{rl} } {\frac{\p^2 G_l}{\p \tu_{i'} \p \tu_{r'}}} \ta_{r'i}^k \right) \\
& - \left( \sum_{r=1}^{n} \frac{\p^2 \teta}{\p \tu_{i} \p \tu_r} \ta_{ri'}^k - \sum_{r,r',l=1}^{n} \frac{\p \teta}{\p \tu_r} {\left( \frac{\partial G}{\partial \tu} \right)^{-1}_{rl} } {\frac{\p^2 G_l}{\p \tu_{i} \p \tu_{r'}}} \ta_{r'i'}^k \right).
\end{align*}    } 
This together with \eqref{2.40} gives rise to \eqref{3.2}.
\end{proof}

\subsubsection{Dissipative structure}
In this part the linear dissipative structure is investigated for the subsystem \eqref{2.47_Sys_2n}, which yields the damping effect on the component $\tud$ of the system \eqref{2.12}.
\begin{lem}\label{diss}
Recall $\widetilde{A}^{0,\flat}(0)$ from \eqref{3.1} and $\Theta$ from \eqref{theta1}. Then it holds
\begin{equation} \label{3.10}
\widetilde{A}^{0,\flat}(0) \begin{pmatrix}
0 & 0 \\ 0 & \thd
\end{pmatrix} + \begin{pmatrix}
0 & 0 \\ 0 & \left(\thd\right)^T
\end{pmatrix} \widetilde{A}^{0,\flat}(0) = - \begin{pmatrix}
0 & 0 \\ 0 & \tmmd
\end{pmatrix}
\end{equation}
for a constant $(n-r)\times(n-r)$ symmetric matrix $\tmmd$, which is positive definite
\begin{equation}\label{mm}
\tmmd \geq \tilde c_m I_{n-r},
\end{equation}
for a constant $\tilde{c}_m > 0$.
\end{lem}
\begin{proof}
We employ a modification of the proof of Theorem 2.1 of \cite{yong2004ARMA}. {\color{black}Recall that $\nuu \tq(0) = \begin{pmatrix}
0 & 0 \\ 0 & \thd
\end{pmatrix}$, and block accordingly
$
   \nuu^2 \teta(0)= \begin{pmatrix}
  B_1 & B_2^T \\
  B_2 & B_4
  \end{pmatrix}.
$
It follows that}
\begin{align}
\nuu^2 \teta(0) \nuu \tq(0) + (\nuu \tq(0))^T \nuu^2 \teta (0) = & \begin{pmatrix}
  B_1 & B_2^T \\
  B_2 & B_4  \end{pmatrix} \begin{pmatrix}
  0 & 0 \\ 0 & \thd
  \end{pmatrix} + \begin{pmatrix}
  0 & 0 \\ 0 & (\thd)^T
  \end{pmatrix} \begin{pmatrix}
    B_1 & B_2^T \\
    B_2 & B_4 \end{pmatrix} \notag \\
= & \begin{pmatrix}
 0 & B_2^T \thd \\ (\thd)^T B_2 & B_4 \thd + (\thd)^T B_4
\end{pmatrix}. \label{4.3_1}
\end{align}
Since $(\nabla_{\tilde u} \teta({\tilde u}))^T \tq({\tilde u})$ achieves its local maximum at ${\tilde u} = 0$ {\color{black}by \asmp{\ta 3}, one gets that} $\nabla_{\tilde u}^2 \teta(0) \nabla_{\tilde u} \tq(0) + (\nabla_{\tilde u} \tq(0))^T \nabla_{\tilde u}^2 \teta (0)$ is semi-negative definite. It then follows from \eqref{4.3_1} that
\begin{equation}\label{r.kerk}
\nuu^2 \teta(0) \nuu \tq(0) + (\nuu \tq(0))^T \nuu^2 \teta (0) = -\begin{pmatrix}
0 & 0 \\ 0 & \tmmd
\end{pmatrix},
\end{equation}
where $\tmmd =-B_4 \thd - (\thd)^T B_4$ is symmetric and semi-positive definite.

Note that \eqref{3.10} follows by \eqref{r.kerk} and \eqref{3.2+}. It remains to prove that $\tmmd$ is indeed positive definite. Note that it is equivalent to prove that $\tmmd$ is invertible.
Suppose that $\tmmd$ has a zero eigenvalue, one may take a corresponding right eigenvector ${\tilde u}^\mathtt{D}_* \neq 0$  and set ${\tilde u}_* = \begin{pmatrix}
0 \\ {\tilde u}^\mathtt{D}_*
\end{pmatrix}$.
Then by a Taylor expansion, since $\tilde Q(0)=0=\nabla_{\tilde u} \teta(0)$,
\begin{align*}
0 \geq & (\nabla_{\tilde u} \teta({\tilde u}_*))^T \tq({\tilde u}_*) + c_e |\tq({\tilde u}_*)|^2 \\
= & u_*^T \big( \nabla_{\tilde u} ^2 \teta(0) \nabla_{\tilde u} \tq(0) + (\nabla_{\tilde u} \tq(0))^T \nabla_{\tilde u}^2 \teta(0) \big) {\tilde u}_\ast
+ \tilde c_e |\nabla_{\tilde u} \tq(0) \tilde u_*|^2 + o(|\tilde u_*|^2) \\
= &  -({\tilde u}^\mathtt{D}_*)^T 2 \tmmd  {\tilde u}^\mathtt{D}_* + \tilde c_e |\thd {\tilde u}^\mathtt{D}_*|^2 + o(|{\tilde u}^\mathtt{D}_*|^2) \\
= & \tilde c_e |\thd {\tilde u}^\mathtt{D}_*|^2 + o(|{\tilde u}^\mathtt{D}_*|^2),
\end{align*}
which yields the contradiction {\color{black}since $\thd$ is invertible by \asmp{\ta 1},} if ${\tilde u}^\mathtt{D}_*$ has been taken sufficiently small.
\end{proof}

\subsubsection{Partial Kawashima condition}
Finally, the partial Kawashima condition inside the assumption \asmp{\ta 4} for the subsystem \eqref{2.47_Sys_2n} is explored, which implies that the component $\tuc$ of the system \eqref{2.12} is indeed affected by the dissipation linearly.
\begin{lem}\label{kaw}
For any $\omega\in \mathbb{S}^{d-1}$, recall $\widetilde{A}^\flat(0,\omega)$ from \eqref{2.42}. Then there exists a constant $c_k > 0$ and an $(n-1) \times (n-1)$ skew-symmetric real matrix $K = K(\omega) \in C^\infty (\mathbb{S}^{d-1})$ with
\begin{equation} \label{3.13}
K(-\omega) = - K(\omega)
\end{equation}
such that
\begin{equation} \label{3.14}
K(\omega) \widetilde{A}^\flat(0,\omega) - \left(\widetilde{A}^\flat(0,\omega)\right)^T K(\omega) \geq 2 c_k I_{n-1} - 2 \mathrm{diag}(0,I_{r-1}), \hs \forall\ \omega \in \mathbb{S}^{d-1}.
\end{equation}
\end{lem}
\begin{proof}
Consider the following linear system 
\begin{equation}\label{hahaaa}
\widetilde{A}^{0,\flat}(0) V_t^\flat + \sum_{k=1}^d \widetilde{A}^{0,\flat}(0) \widetilde{A}^{k,\flat}(0) V_{x_k}^\flat = - \begin{pmatrix}
0 & 0 \\ 0 & \tmmd
\end{pmatrix} V^\flat.
\end{equation}
Due to \eqref{2.50} and \eqref{2.57}, it holds that $\tilde{r}_j^\mtd (0,\omega) \neq 0$ and hence
\begin{equation}
\begin{pmatrix}
0 & 0 \\ 0 & \tmmd
\end{pmatrix} \tilde r_j^\flat(0,\omega) \neq 0,\hs  j=2,\dots, n.
\end{equation}
On the other hand, by \eqref{tezhen} and \eqref{2.42},
$\left\{\tilde r_j^\flat(0,\omega)\right\}_{j =2 }^n$ is a set of right eigenvectors of the matrix $\widetilde{A}^{\flat}(0,\omega)$ with $\tilde{\lambda}_j(0,\omega)$ $(j = 2,\dots, n)$ as the corresponding eigenvalues.
Moreover, by \eqref{2.28} and \eqref{2.35}, $\left\{\tilde r_j^\flat(0,\omega)\right\}_{j =2 }^n$ are also linearly independent for each given $\omega \in \mathbb{S}^{n-1}$.
This verifies the Kawashima condition \asmp{K} for the system \eqref{hahaaa}.
Noting \eqref{3.2}, one can conclude the lemma from Theorem~1.1 of Shizuta--Kawashima \cite{shizutakawashima1985systems}.
\end{proof}


\section{Energy estimates with full dissipation}\label{en1111111}

This section is devoted to derive the energy estimates for the system \eqref{2.12} which includes the full dissipative estimates of the component $\tub$, under the assumption \asmp{\twd 1} or \asmp{\td 1}.
Throughout Sections \ref{en1111111}--\ref{en3333333} it will be assumed {\it a priori} that for sufficiently small $\delta>0$,
\begin{equation}\label{apriori}
\norm{\tu(t)}_{H^\ell}\ls \delta.
\end{equation}

\subsection{Entropy dissipation estimates}
One may first derive the $L^2$ energy estimates for the solution $\tu$ by using the entropy $\teta=\teta(\tu)$, defined by \eqref{entropy}.

\begin{lem}\label{enes1}

The following estimates hold.

$(i)$ Under   \asmp{\twd 1},
\begin{equation}\label{est0}
 \ddt\int \teta(\tu) +  \norm{ \tud}_{L^2}^2  \ls \delta^2 \norm{ \tub}_{L^\infty}^2.
\end{equation}

$(ii)$ Under    \asmp{\td 1},
\begin{equation}\label{est00}
 \ddt\int \teta(\tu) +  \norm{ \tud}_{L^2}^2  \ls \norm{ \tub}_{L^2}^2 \norm{ \tub}_{L^\infty}^2.
\end{equation}
\end{lem}
\begin{proof}
Applying $\left(\ntu \teta(\tu)\right)^T$ to \eqref{2.12}, the entropy assumption \asmp{\ta 2} implies
\begin{equation}\label{en1}
\partial_t \teta(\tu) + \sum_{k = 1}^{d} \left(\widetilde{\psi}^k(\tilde{u})\right)_{x_k} = \ntu \teta(\tu)\cdot\widetilde{Q}(\tilde{u}).
\end{equation}
Integrating \eqref{en1} over $\rd$, by the entropy dissipation assumption \asmp{\ta 3}, one has
\begin{equation}\label{en2}
 \ddt\int \teta(\tu) + \tilde{c}_e\norm{\tq(\tu)}_{L^2}^2 \le 0.
\end{equation}
Since trivially $\abs{\tq^\mathtt{D}(\tu)}^2\le \abs{\tq(\tu)}^2$ and recalling $\Theta$ defined in \eqref{theta1}, it follows from \eqref{en2} that
\begin{equation}\label{en3}
 \ddt\int \teta(\tu) + \norm{\thd \tud}_{L^2}^2  \ls \norm{ \tq^\mathtt{D}(\tu)-\thd \tud}_{L^2}^2.
\end{equation}

Now one may prove $(i)$ under \asmp{\twd 1}. By the expansion \eqref{for decay2} and the a priori assumption \eqref{apriori}, using Sobolev's inequality, it holds that
\begin{equation}\label{en4}
 \norm{ \tq^\mathtt{D}(\tu)-\thd \tud}_{L^2}^2  \ls \norm{  \ty(\tu)\tub\tu }_{L^2}^2
 \ls  \norm{ \tu}_{L^2}^2 \norm{ \tub}_{L^\infty}^2\ls  \delta^2 \norm{ \tub}_{L^\infty}^2.
\end{equation}
Plugging the estimates \eqref{en4} into \eqref{en3}, by the assumption \asmp{\ta 1} (or equivalently, \eqref{3.10++}), one deduces \eqref{est0}.

One then turns to prove $(ii)$ under \asmp{\td 1}. By the expansion \eqref{eexx44}, it holds that
\begin{equation}\label{en5}
 \norm{ \tq^\mathtt{D}(\tu)-\thd \tud}_{L^2}^2  \ls \norm{  \ty(\tu)\tud\tu }_{L^2}^2+\norm{  \ty(\tu)\tub\tub }_{L^2}^2
 \ls  \delta^2 \norm{ \tud}_{L^2}^2+ \norm{ \tub}_{L^2}^2 \norm{ \tub}_{L^\infty}^2.
\end{equation}
Plugging the estimates \eqref{en5} into \eqref{en3}, since $\delta$ is small, one deduces \eqref{est00}.
\end{proof}

\subsection{Energy estimates for $\tub$}

One may next derive the higher-order energy estimates for the component $\tub$ by using the structure of the subsystem \eqref{2.47_Sys_2n}.

\begin{lem}\label{enes2}

For $1\le s\le \ell$, the following estimates hold.

$(i)$ Under    \asmp{\twd 1},
\begin{equation}\label{est1}
\begin{split}
& \ddt\int (\lds\tub)^T\ta^{0,\flat}(\tu)\lds\tub +  \norm{\lds\tud}_{L^2}^2
 \\&\quad\ls  \delta \left(\norm{\lds \tub}_{L^2}^2+\norm{  \tub  }_{L^\infty}^2+\norm{  \na \tub  }_{L^\infty}^2\right).
 \end{split}
\end{equation}

$(ii)$ Under   \asmp{\td 1},
\begin{equation}\label{est11}
 \begin{split}
& \ddt\int (\lds\tub)^T\ta^{0,\flat}(\tu)\lds\tub +  \norm{\lds\tud}_{L^2}^2
   \\&\quad\ls  \delta \left(\norm{\lds \tub}_{L^2}^2+\norm{  \tud  }_{L^\infty}^2+\norm{  \na \tub  }_{L^\infty}^2\right)+\norm{\lds \tub}_{L^2}\norm{ \tub  }_{L^\infty}^2.
 \end{split}
 \end{equation}
\end{lem}
\begin{proof}
For $1\le s\le \ell$, applying $\Lambda^s$ to \eqref{2.47_Sys_2n} and employing the commutator notation \eqref{commutator} yields
\begin{equation} \label{s order}
\partial_t (\lds\tub) + \sum_{k=1}^{d} \ta^{k,\flat}(\tu) (\lds\tub)_{x_k} =-\sum_{k=1}^{d} \left[\lds,\ta^{k,\flat}(\tu)\right] \tub_{x_k}+ \lds\tq^\flat(\tu).
\end{equation}
Applying $(\lds\tub)^T\ta^{0,\flat}(\tu)$ to \eqref{s order}, since by Lemma \ref{a0b} $ \ta^{0,\flat}(\tu)$ and $ \ta^{0,\flat}(\tu)\ta^{k,\flat}(\tu)$ are symmetric, one has
\begin{equation} \label{s order1}
\begin{split}
&\partial_t \left((\lds\tub)^T\ta^{0,\flat}(\tu)\lds\tub\right) +
\sum_{k=1}^{d} \left((\lds\tub)^T\ta^{0,\flat}(\tu)\ta^{k,\flat}(\tu)\lds\tub\right) _{x_k}
\\&\quad= (\lds\tub)^T\left(\partial_t\left(\ta^{0,\flat}(\tu)\right)+\sum_{k=1}^{d} \left( \ta^{0,\flat}(\tu)\ta^{k,\flat}(\tu) \right) _{x_k} \right)\lds\tub
\\&\qquad-2 \re \left( \sum_{k=1}^{d} (\lds\tub)^T\ta^{0,\flat}(\tu) \left[\lds,\ta^{k,\flat}(\tu)\right] \tub_{x_k} \right) + 2 \re \left((\lds\tub)^T\ta^{0,\flat}(\tu)\lds\tq^\flat(\tu) \right).
\end{split}
\end{equation}
Here $\re f=(f + f^T)/2$.

Now one may estimate the integration of the right hand side of \eqref{s order1} over $\rd$. Replacing $\partial_t\tu$ by using the system \eqref{2.12}, sine $\ell>d/2+1$ and by Sobolev's inequality, the first term is easily bounded by
\begin{equation}\label{en11}
 \left(\norm{\dt \tu}_{L^\infty}+\norm{\nabla \tu}_{L^\infty}\right)\norm{\lds \tub}_{L^2}^2\ls\delta\norm{\lds \tub}_{L^2}^2.
\end{equation}
By the commutator estimates \eqref{commutator estimate} of Lemma \ref{A2}, using Cauchy's inequality, one may bound the second term by
\begin{equation}
\begin{split}
& \norm{\lds \tub}_{L^2}\sum_{k=1}^{d}\norm{\left[\lds,\ta^{k,\flat}(\tu)\right] \tub_{x_k}}_{L^2}.
\\&\quad\ls \norm{\lds \tub}_{L^2}\sum_{k=1}^{d}\left(\norm{\nabla(\ta^{k,\flat}(\tu))  }_{L^\infty}\norm{\lds \tub  }_{L^2}+\norm{\lds(\ta^{k,\flat}(\tu))  }_{L^2}\norm{\na \tub  }_{L^\infty} \right)
\\&\quad\ls \delta\left(\norm{\lds \tub}_{L^2}^2+\norm{\na \tub  }_{L^\infty}^2\right).
\end{split}
\end{equation}
For the last term, one can rewrite
\begin{equation}
\begin{split}
&  (\lds\tub)^T\ta^{0,\flat}(\tu)\lds\tq^\flat(\tu)
 \\&\quad= (\lds\tub)^T\ta^{0,\flat}(\tu)\Theta^\flat\lds\tub+  (\lds\tub)^T\ta^{0,\flat}(\tu)\lds\left(\tq^\flat(\tu)-\Theta^\flat\tub\right)
 \\&\quad= (\lds\tub)^T\ta^{0,\flat}(0)\Theta^\flat\lds\tub+ (\lds\tub)^T \left(\ta^{0,\flat}(\tu)-\ta^{0,\flat}(0)\right)\Theta^\flat\lds\tub
 \\&\qquad+  (\lds\tub)^T\ta^{0,\flat}(\tu)\lds\left(\tq^\flat(\tu)-\Theta^\flat\tub\right).
 \end{split}
\end{equation}
By \eqref{3.10}--\eqref{mm} from Lemma \ref{diss}, it holds that
\begin{equation}
  \int 2\re \left( (\lds\tub)^T\ta^{0,\flat}(0)\Theta^\flat\lds\tub \right)=-\int  (\lds\tud)^T \tmmd \lds\tud \le - \tilde c_m\norm{\lds\tud}_{L^2}^2.
\end{equation}
On the other hand, it is easy to bound
\begin{equation}\label{en15}
\int 2 \re \left( (\lds\tub)^T \left(\ta^{0,\flat}(\tu)-\ta^{0,\flat}(0)\right)\Theta^\flat\lds\tub \right)
\ls \norm{\tu}_{L^\infty} \norm{\lds\tub}_{L^2}^2\ls \delta\norm{\lds\tub}_{L^2}^2.
\end{equation}
Hence, integrating \eqref{s order1} over $\rd$, it follows from \eqref{en11}--\eqref{en15} that
\begin{equation}\label{en16}
\begin{split}
 &\ddt\int (\lds\tub)^T\ta^{0,\flat}(\tu)\lds\tub +  \norm{\lds\tud}_{L^2}^2
\\&\quad \ls  \delta \left(\norm{\lds \tub}_{L^2}^2 +\norm{  \na \tub  }_{L^\infty}^2\right)+\norm{\lds\tub}_{L^2}\norm{\lds\left(\tq^\flat(\tu)-\Theta^\flat\tub\right)}_{L^2}.
 \end{split}
\end{equation}

Now one may prove $(i)$ under \asmp{\twd 1}. By the expansions \eqref{for decay1}--\eqref{for decay2} and using the product estimates \eqref{product estimate} of Lemma \ref{A2}, one may estimate
\begin{equation}\label{en41}
\begin{split}
&
\norm{\lds\left(\tq^\flat(\tu)-\Theta^\flat\tub\right)}_{L^2}   \ls  \norm{\lds\left(\ty(\tu)\tub \tu \right)}_{L^2}
\\&\quad \ls \norm{ \ty(\tu)\tu  }_{L^\infty}\norm{\lds\tub}_{L^2}+\norm{\lds\left(\ty(\tu)\tu \right)}_{L^2}\norm{\tub}_{L^\infty}
\\&\quad\ls \delta\left(\norm{\lds \tub}_{L^2} +\norm{  \tub  }_{L^\infty}  \right).
\end{split}
\end{equation}
Plugging the estimates \eqref{en41} into \eqref{en16} yields \eqref{est1}.

One then turns to prove $(ii)$ under \asmp{\td 1}.
By the expansions \eqref{eexx4}--\eqref{eexx44} and using the product estimates \eqref{product estimate} of Lemma \ref{A2}, one may estimate
\begin{equation}\label{en51}
\begin{split}
&
\norm{\lds\left(\tq^\flat(\tu)-\Theta^\flat\tub\right)}_{L^2}   \ls  \norm{\lds\left(\ty'(\tu)\tud \tu \right)}_{L^2}+\norm{\lds\left(\ty'(\tu)\tub \tub \right)}_{L^2}
\\&\quad \ls \norm{ \ty'(\tu)\tu  }_{L^\infty}\norm{\lds\tud}_{L^2}+\norm{\lds\left(\ty'(\tu)\tu \right)}_{L^2}\norm{\tud}_{L^\infty}
\\&\qquad+\norm{ \ty'(\tu)\tub  }_{L^\infty}\norm{\lds\tub}_{L^2}+\norm{\lds\left(\ty'(\tu)\tub \right)}_{L^2}\norm{\tub}_{L^\infty}
\\&\quad\ls \delta\left(\norm{\lds \tub}_{L^2} +\norm{  \tud  }_{L^\infty} \right) +\norm{ \tub  }_{L^\infty}^2 .
\end{split}
\end{equation}
Plugging the estimates \eqref{en51} into \eqref{en16} yields \eqref{est11}.
\end{proof}

\subsection{Recovering the full dissipation}

Noting that Lemmas \ref{enes1}--\ref{enes2} only contain the dissipation estimates for the component $\tud$,
one may then use the partial Kawashima condition of the subsystem \eqref{2.47_Sys_2n} to recover the dissipation estimates for the component $\tuc$ (and hence $\tub$).
\begin{lem}\label{enes3}

For $0\le s\le \ell-1$, the following estimates hold.

$(i)$ Under   \asmp{\twd 1} and that $d\ge 3$,
\begin{equation}\label{est22}
\begin{split}
&\ddt\int-i \left(\abs{\xi}^{ s }\widehat{\tub}\right)^T K \left(\abs{\xi}^{ s+1 }\widehat{\tub}\right) \rmd\xi  +  \norm{\Lambda^{s+1} \tub}_{L^2}^2
\\&\quad\ls \norm{\Lambda^{s } \tud}_{L^2}^2 +\norm{\Lambda^{s+1} \tud}_{L^2}^2 +\delta \left( \norm{  \tub  }_{L^\infty}^2+\norm{\na \tub  }_{L^\infty}^2\right).
\end{split}
\end{equation}

$(ii)$ Under   \asmp{\td 1},
\begin{equation}\label{est222}
\begin{split}
&\ddt\int-i \left(\abs{\xi}^{ s }\widehat{\tub}\right)^T K \left(\abs{\xi}^{ s+1 }\widehat{\tub}\right) \rmd\xi  +  \norm{\Lambda^{s+1} \tub}_{L^2}^2
\\&\quad\ls \norm{\Lambda^{s } \tud}_{L^2}^2 +\norm{\Lambda^{s+1} \tud}_{L^2}^2 +\delta \left( \norm{  \tud  }_{L^\infty}^2+\norm{\na \tub  }_{L^\infty}^2 + \norm{ \tub  }_{L^\infty}^4 \right) \\&\qquad +\norm{\lds\tub}_{L^2}^2\norm{ \tub  }_{L^\infty}^2.
\end{split}
\end{equation}
Here $K = K(\xi / |\xi|)$ is defined in Lemma \ref{kaw}.
\end{lem}
\begin{proof}
The subsystem \eqref{2.47_Sys_2n} can be rewritten in the following form
\begin{equation} \label{perturbed form}
\partial_t \tub  + \sum_{k=1}^{d} \ta^{k,\flat}(0)   \tub_{x_k}
= -\sum_{k=1}^{d} \left(\ta^{k,\flat}(\tu)- \ta^{k,\flat}(0)  \right) \tub_{x_k}+ \tq^\flat(\tu):=h^\flat .
\end{equation}
Taking the Fourier transform yields
\begin{equation} \label{perturbed form1}
\partial_t \widehat{\tub}  + i\sum_{k=1}^{d} \xi_k\ta^{k,\flat}(0)\widehat{\tub} = \widehat{h^\flat}.
\end{equation}
Let $K=K(\xi/|\xi|)$ be the skew-symmetric real matrix from Lemma \ref{kaw}, then
\begin{equation} \label{kk1}
K\sum_{k=1}^{d}  \xi_k\ta^{k,\flat}(0)-\sum_{k=1}^{d}  \xi_k\left(\ta^{k,\flat}(0)\right)^TK \ge 2\abs{\xi}\left(c_k I_{n-1}- {\mathrm{diag}}(0,I_{n-r})\right).
\end{equation}
Applying $-i(\widehat{\tub})^T K$ to \eqref{perturbed form1} yields
\begin{equation} \label{perturbed form2}
-i(\widehat{\tub})^T K\partial_t \widehat{\tub}  + (\widehat{\tub})^T K\sum_{k=1}^{d}\xi_k\ta^{k,\flat}(0)\widehat{\tub} =-i(\widehat{\tub})^T K \widehat{h^\flat}.
\end{equation}
It then follows from \eqref{kk1} that
\begin{equation} \label{perturbed form3}
-i\partial_t\left((\widehat{\tub})^T K \widehat{\tub}\right)  + 2 c_k\abs{\xi}\abs{\widehat{\tub}}^2 \le 2\abs{\xi}\abs{\widehat{\tud}}^2 + 2 \abs{(\widehat{\tub})^T K \widehat{h^\flat}}.
\end{equation}

For $0\le s\le \ell-1$, multiplying \eqref{perturbed form3} by $\abs{\xi}^{2s+1}$ and then integrating over the frequency space $\rd_\xi$ yields, by the Plancherel theorem,
\begin{equation} \label{perturbed form4}
\begin{split}
\ddt\int-i \abs{\xi}^{2s+1} (\widehat{\tub})^T K \widehat{\tub} \rmd\xi  + \norm{\Lambda^{s+1}  \tub}_{L^2}^2 \ls \norm{\Lambda^{s+1}  \tud}_{L^2}^2 +\norm{\Lambda^{s+1} \tub}_{L^2}\norm{\Lambda^{s} h^\flat}_{L^2}.
\end{split}
\end{equation}
To estimate $\norm{\Lambda^{s } h^\flat}_{L^2}$, one may use the product estimates \eqref{product estimate} of Lemma \ref{A2} to get
\begin{equation} \label{perturbed form8}
\begin{split}
&\norm{\Lambda^{s } \left(\left(\ta^{k,\flat}(\tu)- \ta^{k,\flat}(0)  \right) \tub_{x_k}\right)}_{L^2}
  \\&\quad \ls    \norm{\ta^{k,\flat}(\tu)- \ta^{k,\flat}(0)  }_{L^\infty}\norm{\Lambda^{s }\na \tub}_{L^{2 }}+\norm{\Lambda^{s }\left(\ta^{k,\flat}(\tu)- \ta^{k,\flat}(0)  \right)}_{L^2}\norm{\na\tub}_{L^\infty}
\\&\quad \ls \delta\left(\norm{\Lambda^{s+1} \tub}_{L^2} +\norm{ \na \tub  }_{L^\infty}\right).
\end{split}
\end{equation}
On the other hand
\begin{equation} \label{perturbed form88}
\begin{split}
\norm{\Lambda^{s } \left(\tq^\flat(\tu)\right)}_{L^2}\le \norm{ \Theta^\mathtt{D} \Lambda^{s }\tud }_{L^2}+\norm{\lds\left(\tq^\flat(\tu)-\Theta^\flat\tub\right)}_{L^2},
\end{split}
\end{equation}
Hence, by \eqref{perturbed form8}--\eqref{perturbed form88}, it follows from \eqref{perturbed form4} that
\begin{equation} \label{en23}
\begin{split}
&\ddt\int-i \left(\abs{\xi}^{ s }\widehat{\tub}\right)^T K \left(\abs{\xi}^{ s+1 }\widehat{\tub}\right) d\xi  +  \norm{\Lambda^{s+1} \tub}_{L^2}^2
\\&\quad\ls \norm{\Lambda^{s } \tud}_{L^2}^2 +\norm{\Lambda^{s+1} \tud}_{L^2}^2 +\delta \norm{\na \tub  }_{L^\infty}^2+ \norm{\Lambda^{s+1}\tub}_{L^2}\norm{\lds\left(\tq^\flat(\tu)-\Theta^\flat\tub\right)}_{L^2} .
\end{split}
\end{equation}

Now one may prove $(i)$ under \asmp{\twd 1} and that $d\ge3$. By the expansions \eqref{for decay1}--\eqref{for decay2} and the product estimates \eqref{product estimate} of Lemma \ref{A2}, it holds that
\begin{equation}\label{en42}
\begin{split}
&
\norm{\lds\left(\tq^\flat(\tu)-\Theta^\flat\tub\right)}_{L^2}   \ls  \norm{\lds\left(\ty(\tu)\tub \tu \right)}_{L^2}
\\&\quad \ls \norm{ \ty(\tu)\tu  }_{L^d}\norm{\lds\tub}_{L^ \frac{2d}{d-2}}+\norm{\lds\left(\ty(\tu)\tu \right)}_{L^2}\norm{\tub}_{L^\infty}
\\&\quad\ls \delta\left(\norm{\Lambda^{s+1}\tub}_{L^2} +\norm{  \tub  }_{L^\infty}  \right).
\end{split}
\end{equation}
Plugging the estimates \eqref{en42} into \eqref{en23}, one deduces \eqref{est22}.

One then turns to prove $(ii)$ under \asmp{\td 1}. By the expansions \eqref{eexx4}--\eqref{eexx44} and the product estimates \eqref{product estimate} of Lemma \ref{A2}, it holds that
\begin{equation}\label{en52}
\begin{split}
&
\norm{\lds\left(\tq^\flat(\tu)-\Theta^\flat\tub\right)}_{L^2}   \ls  \norm{\lds\left(\ty'(\tu)\tud \tu \right)}_{L^2}+\norm{\lds\left(\ty'(\tu)\tub \tub \right)}_{L^2}
\\&\quad \ls \norm{ \ty'(\tu)\tu  }_{L^\infty}\norm{\lds\tud}_{L^2}+\norm{\lds\left(\ty'(\tu)\tu \right)}_{L^2}\norm{\tud}_{L^\infty}
+\norm{ \ty'(\tu)\tub  }_{L^\infty}\norm{\lds\tub}_{L^2}
\\&\qquad + \norm{\lds\left( \left( \ty'(\tu) - \ty'(0) \right) \tub \right)}_{L^2}\norm{\tub}_{L^\infty}
+ \norm{\lds\left( \ty'(0) \tub \right)}_{L^2}\norm{\tub}_{L^\infty}
\\&\quad\ls \delta\left(\norm{\lds \tud}_{L^2} +\norm{  \tud  }_{L^\infty} \right) +\delta \norm{ \tub  }_{L^\infty}^2+\norm{\lds\tub}_{L^2}\norm{ \tub  }_{L^\infty}  .
\end{split}
\end{equation}
Plugging the estimates \eqref{en52} into \eqref{en23}, one deduces \eqref{est222}.
\end{proof}

\subsection{Synthesis}
Now one can combine Lemmas \ref{enes1}--\ref{enes3} to arrive the following propositions.

\begin{prop}\label{enes5}
For $1\le s\le \ell-1$, there exists an energy functional $\mathcal{E}_s(\tub)$ equivalent to $\norm{\Lambda^{s} \tub}_{H^{\ell-s}}^2$ such that the following estimates hold.

$(i)$ Under   \asmp{\twd 1} and that $d\ge 3$,
\begin{equation}\label{energy es for decay}
\begin{split}
 &\ddt\mathcal{E}_s(\tub)+  \norm{\Lambda^{s } \tud}_{H^{\ell-s }}^2 +  \norm{\Lambda^{s+1} \tuc}_{H^{\ell-s-1}}^2
\\&\quad \ls \delta \left(\pnorm{\lds \tuc }{2}^2+ \norm{  \tub  }_{L^\infty}^2+\norm{\na \tub  }_{L^\infty}^2  \right).
 \end{split}
\end{equation}

$(ii)$ Under   \asmp{\td 1},
\begin{equation}\label{energy es for decay1}
\begin{split}
 &\ddt\mathcal{E}_s(\tub)+  \norm{\Lambda^{s } \tud}_{H^{\ell-s }}^2 +  \norm{\Lambda^{s+1} \tuc}_{H^{\ell-s-1}}^2
\\&\quad \ls \delta \left(\pnorm{\lds \tuc }{2}^2+ \norm{  \tud  }_{L^\infty}^2+\norm{\na \tub  }_{L^\infty}^2 +\norm{ \tub  }_{L^\infty}^4 \right)+\norm{\lds \tub}_{L^2}\norm{ \tub  }_{L^\infty}^2.
 \end{split}
\end{equation}
\end{prop}
\begin{proof}
One may first prove $(i)$ under \asmp{\twd 1} and that $d\ge 3$ by using the $(i)$ assertions of Lemmas \ref{enes2}--\ref{enes3}.
Taking first $1\le s\le \ell-1$ and then $s=\ell$ in the estimates \eqref{est1} of Lemma \ref{enes2}, noting that $\norm{f}_{H^k}$ is equivalent to the sum $\norm{f}_{L^2}+\norm{\Lambda^k f}_{L^2}$, one has
\begin{equation}\label{conclusion1}
\begin{split}
 &\ddt\int\left( (\lds\tub)^T\ta^{0,\flat}(\tu)\lds\tub+(\Lambda^\ell\tub)^T\ta^{0,\flat}(\tu)\Lambda^\ell\tub  \right)+  \norm{\lds\tud}_{H^{\ell-s}}^2
 \\&\quad \ls  \delta \left(\norm{\lds \tub}_{H^{\ell-s}}^2+\norm{  \tub  }_{L^\infty}^2+\norm{\na \tub  }_{L^\infty}^2\right).
 \end{split}
\end{equation}
Next, taking first $1\le s\le \ell-1$ and then $s=\ell-1$ in the estimates \eqref{est22} of Lemma \ref{enes3}, one obtains
\begin{equation} \label{conclusion2}
\begin{split}
&\ddt\int-i \left( \left(\abs{\xi}^{ s }\widehat{\tub}\right)^T K \left(\abs{\xi}^{ s+1 }\widehat{\tub}\right)
+\left(\abs{\xi}^{\ell-1 }\widehat{\tub}\right)^T K \left(\abs{\xi}^{ \ell}\widehat{\tub}\right)\right)d\xi  +  \norm{\Lambda^{s+1} \tub}_{H^{\ell-s-1}}^2
\\&\quad\ls \norm{\Lambda^{s} \tud}_{H^{\ell-s}}^2   +\delta \left( \norm{  \tub  }_{L^\infty}^2+\norm{\na \tub  }_{L^\infty}^2\right).
\end{split}
\end{equation}
For sufficiently small constant $\epsilon>0$, if one defines the instant energy functional
\begin{equation}
\begin{split}
 \mathcal{E}_s(\tub)=&\int\left( (\lds\tub)^T\ta^{0,\flat}(\tu)\lds\tub+(\Lambda^\ell\tub)^T\ta^{0,\flat}(\tu)\Lambda^\ell\tub  \right)
 \\&+\epsilon\int-i \left( \left(\abs{\xi}^{ s }\widehat{\tub}\right)^T K \left(\abs{\xi}^{ s+1 }\widehat{\tub}\right)
+\left(\abs{\xi}^{\ell-1 }\widehat{\tub}\right)^T K \left(\abs{\xi}^{ \ell}\widehat{\tub}\right)\right) \rmd\xi
\end{split}
\end{equation}
for $1\le s\le \ell-1$, since $\ta^{0,\flat}(\tu)$ is positive definite near $\tu=0$ from Lemma \ref{a0b}, then it follows that $\mathcal{E}_s(\tub)$ is equivalent to $\norm{\Lambda^{s} \tub}_{H^{\ell-s}}^2$. Moreover, since both $\epsilon$ and $\delta$ are small, one can deduce \eqref{energy es for decay} from \eqref{conclusion1}--\eqref{conclusion2}.

Similarly, to prove $(ii)$ under \asmp{\td 1}, one needs only to use instead the $(ii)$ assertions of Lemmas \ref{enes2}--\ref{enes3} in the arguments above to deduce \eqref{energy es for decay1}.
\end{proof}

\begin{prop}\label{enes6}
The following estimates hold.

$(i)$ Under   \asmp{\twd 1} and that $d\ge 3$,
\begin{equation}\label{energy es for global}
\begin{split}
 &
 \norm{\tu(t)}_{L^2}^2+\norm{\Lambda  \tub(t)}_{H^{\ell-1}}^2+ \int_0^t \left( \norm{\tud(\tau)}_{H^\ell}^2+\norm{\Lambda  \tuc(\tau)}_{H^{\ell-1}}^2 \right) \rmd\tau
 \\&\quad \ls \norm{\tu_0}_{L^2}^2+\norm{\Lambda  \tub_0}_{H^{\ell-1}}^2.
 \end{split}
\end{equation}

$(ii)$ Under   \asmp{\td 1},
\begin{equation}\label{energy es for global+}
\begin{split}
 &
 \norm{\tu(t)}_{L^2}^2+\norm{\Lambda  \tub(t)}_{H^{\ell-1}}^2+ \int_0^t \left( \norm{\tud(\tau)}_{H^\ell}^2+\norm{\Lambda  \tuc(\tau)}_{H^{\ell-1}}^2 \right) \rmd\tau
 \\&\quad \ls \norm{\tu_0}_{L^2}^2+\norm{\Lambda  \tub_0}_{H^{\ell-1}}^2 + \int_0^t \left( \norm{\tub(\tau)}_{L^2}^2 \norm{\tub(\tau)}_{L^\infty}^2 + \norm{\tub(\tau)}_{L^\infty}^4 \right) \rmd \tau.
 \end{split}
\end{equation}
 \end{prop}
\begin{proof}
One may first prove $(i)$ under \asmp{\twd 1} and that $d\ge 3$.
Taking $s=1$ in the estimates \eqref{energy es for decay} of Proposition \ref{enes5}, one obtains
\begin{equation}\label{hih1}
\begin{split}
 &\ddt\mathcal{E}_1(\tub)+  \norm{\Lambda  \tud}_{H^{\ell-1 }}^2 +  \norm{\Lambda^2 \tuc}_{H^{\ell-2}}^2
\\&\quad \ls \delta \left(\pnorm{\Lambda \tub }{2}^2+ \norm{  \tub  }_{L^\infty}^2+\norm{\na \tub  }_{L^\infty}^2  \right).
 \end{split}
\end{equation}
Taking $s=0$ in the estimates \eqref{est22} of Lemma \ref{enes3} yields
\begin{equation}\label{hih2}
\begin{split}
&\ddt\int-i   (  \widehat{\tub} )^T K \left(\abs{\xi} \widehat{\tub} \right) \rmd\xi  +  \norm{\Lambda \tub}_{L^2}^2
\\&\quad\ls \norm{ \tud}_{L^2}^2 +\norm{\Lambda \tud}_{L^2}^2 +\delta \left( \norm{  \tub  }_{L^\infty}^2+\norm{\na \tub  }_{L^\infty}^2\right).
\end{split}
\end{equation}
Since $\ell>d/2+1$ and  $d\ge3$, by the interpolation estimates of Lemma \ref{A1}, it holds that
\begin{equation}
 \norm{  \tub  }_{L^\infty}^2+\norm{\na \tub  }_{L^\infty}^2\ls \norm{\Lambda  \tub }_{H^{\ell-1}}^2
 .
\end{equation}
By this, integrating in time the sum of the estimates \eqref{hih1}--\eqref{hih2} and the entropy estimates \eqref{est0} of Lemma \ref{enes1}, since $\delta$ is small, one can deduce \eqref{energy es for global}.

Similarly, to prove $(ii)$ under \asmp{\td 1}, one needs only to use instead the estimates \eqref{energy es for decay1} of Proposition \ref{enes5}, the estimates \eqref{est222} of Lemma \ref{enes3} and the estimates \eqref{est00} of Lemma \ref{enes1} to deduce \eqref{energy es for global+}.
\end{proof}

%
%
%

\section{Full energy estimates}\label{en2222222}
In view of the energy estimates \eqref{energy es for global}--\eqref{energy es for global+} of Proposition \ref{enes6} in Section \ref{en1111111}, it appears natural to derive the energy estimates for the derivatives of the remaining component $\tu_1$.
However, the evolution equation \eqref{2.45_Sys_1} of $\tu_1$ is apparently not suitable for this goal since there would be a loss of derivatives if one directly performs the energy estimate.
To get around this obstacle, despite that the component $\tub$ has already been controlled, one should still return back to derive the estimates for the derivatives of the whole $\tu$ by using the original system \eqref{2.12}.

\begin{prop}\label{energy energy}

The following estimates hold.

$(i)$ Under   \asmp{\twd 1}--\asmp{\twd 2}  and that $d\ge 3$,
\begin{equation}\label{energy es for global1212}
  \norm{\tu(t)}_{H^\ell}
 \ls \norm{\tu_0}_{H^\ell} +\int_0^t \left(\norm{     \tub  }_{L^\infty}+\norm{ \nabla   \tub  }_{L^\infty}+\norm{ \Lambda^{\ell-1}   \tub  }_{L^2}+\norm{ \Lambda^{\ell}   \tub  }_{L^2}\right) \rmd\tau.
\end{equation}

$(ii)$ Under   \asmp{\td 1}--\asmp{\td 2} and that $d\ge 2$,
\begin{equation}\label{energy es for global112}
  \norm{\tu(t)}_{H^\ell}
 \ls \norm{\tu_0}_{H^\ell} +\int_0^t \left(\norm{     \tub  }_{L^\infty}^2+\norm{     \tud  }_{L^\infty}+\norm{ \nabla   \tub  }_{L^\infty}+\norm{ \Lambda^{\ell-1}   \tub  }_{L^2}+\norm{ \Lambda^{\ell}   \tub  }_{L^2}\right) \rmd\tau.
\end{equation}
\end{prop}
\begin{proof}
Thanks to the entropy estimates in Lemma \ref{enes1}, by the Sobolev interpolation, it suffices to derive the energy estimates for  $\ldll \tu$. Note that this is equivalent to derive the energy estimates of $\ldl1 \nabla \tu $.
In the spirit of one dimensional study {\color{black}\cite{john1974blowup,li1994bookglobal,lizhoukong1994global},} one may do the wave decomposition
\begin{equation}  \label{4.32++}
  \tu_{x_k} = \sum_{i=1}^{n} \wik \tilde{r}_i (\tu,\sek),\quad k=1,\dots,d,
\end{equation}
where
\begin{equation}\label{wii}
 \wik = \tilde{l}_i(\tu, \sek) \tu_{x_k},\hs  i=1,\dots,n;\ k=1,\dots,d.
\end{equation}
Using the product estimate \eqref{product estimate} of Lemma \ref{A2}, one has
\begin{align}
\norm{\ldl1 \tu_{x_k}}_{L^2} \ls & \sum_{i=1}^{n} \norm{\wik}_{L^\infty} \norm{\ldl1 ( \tilde{r}_i(\tu,\sek) - \tilde{r}_i(0,\sek) )}_{L^2}  + \sum_{i=1}^{n} \norm{\ldl1 \wik}_{L^2} \norm{\tilde{r}_i(\tu,\sek)}_{L^\infty}\nonumber \\
\ls & \delta \norm{\tu}_{H^\ell} + \sum_{i=1}^{n} \norm{\ldl1 \wik}_{L^2},\hs k = 1,\dots,d.
\end{align}
It then reduces to derive the energy estimates for $\ldl1 \wik$. But by \eqref{2.35}, it holds that
\begin{equation}\label{wjj}
 \tilde{w}_{j,k} = \tilde{l}_j(\tu, \sek) \tu_{x_k}=\sum_{j'=2}^n \tilde{l}_{jj'}(\tu, \sek)\tu_{j'\!,x_k},\hs j=2,\dots,n;\ k = 1,\dots,d.
\end{equation}
This together with the estimates \eqref{energy es for global}--\eqref{energy es for global+} of Proposition \ref{enes6} implies that
it suffices to derive the energy estimates of $\ldl1 \wokk$ for each $k'=1,\dots,d$. Note that for the notational convenience the index notation $k$ is changed into $k'$.

Now it is crucial to derive the equation of $ \wokk$. Taking $\p_{x_{k'}}$ to the system \eqref{2.12} yields
\begin{equation}\label{hmh1}
 \p_t   \tu_{x_{k'}}  + \sum_{k=1}^{d}  \left( \ta^k(\tu)  \tu_{x_k}  \right)_{x_{k'}}  = \ntu \tq (\tu)  \tu _{x_{k'}}.
\end{equation}
Substituting the decomposition \eqref{4.32++} into \eqref{hmh1}, by \eqref{tezhen} and the system \eqref{2.12}, one can deduce
\begin{align}\label{hh1}
& \sum_{i=1}^n \p_t \wikk \tilde{r}_i(\tu,\sekk ) + \sum_{i=1}^{n} \sum_{k=1}^{d} \tilde{\lambda}_i(\tu, \sek )  \left( \wik\right)_{x_{k'}} \tilde{r}_i (\tu, \sek)\nonumber\\
& \quad =\sum_{i=1}^{n}   \wikk \ntu \tq (\tu) \tilde{r}_i(\tu, \sekk)\nonumber\\& \qquad+  \sum_{i=1}^{n} \wikk \ntu \tilde{r}_i(\tu,\sekk )
\left( \sum_{i'=1}^{n}\sum_{k=1}^{d} \tilde\lambda_{i'}(\tu,\sek) \wiik \tilde{r}_{i'}(\tu,\sek) - \tq(\tu) \right) \\
 & \qquad - \sum_{i=1}^{n} \sum_{k=1}^{d} \wik \tilde{\lambda}_i(\tu,\sek) \ntu \tilde{r}_i(\tu, \sek ) \left( \sum_{i'=1}^{n} \wiikk \tilde{r}_{i'}(\tu,\sekk ) \right)\nonumber \\
& \qquad - \sum_{i=1}^{n} \sum_{k=1}^{d}\wik \left( \ntu \tilde{\lambda}_i(\tu,\sek) \cdot \tilde{r}_{i'}(\tu,\sekk)  \right) \wiikk \tilde{r}_i(\tu,\sek ). \nonumber
\end{align}
Multiplying \eqref{hh1} by $\tilde{l}_1(\tu )$ (owing to {\asmp{\tb}}) and using \eqref{2.28}, one has
\begin{align}\label{hh2}
&   \p_t \wokk   +   \sum_{k=1}^{d} \tilde{\lambda}_1(\tu, \sek )  \left( \wok\right)_{x_{k'}} \nonumber\\
& \quad =\sum_{i=1}^{n}   \tilde{l}_1(\tu )   \ntu \tq (\tu) \tilde{r}_i(\tu, \sekk)  \wikk
 -\sum_{i=1}^{n}   \tilde{l}_1(\tu ) \ntu \tilde{r}_{i} (\tu, \sekk ) \tq(\tu )  \wikk \nonumber\\& \qquad+  \sum_{i,i'=1}^{n} \sum_{k=1}^{d}  \tilde{l}_1(\tu )  \ntu \tilde{r}_i(\tu,\sekk ) \tilde{r}_{i'}(\tu,\sek)
 \tilde\lambda_{i'}(\tu,\sek) \wikk\wiik   \\
 & \qquad - \sum_{i,i'=1}^{n} \sum_{k=1}^{d} \tilde{l}_1(\tu )   \ntu \tilde{r}_i(\tu, \sek )\tilde{r}_{i'}(\tu,\sekk ) \tilde{\lambda}_i(\tu,\sek)  \wik \wiikk \nonumber\\
& \qquad - \sum_{i=1}^{n} \sum_{k=1}^{d} \left( \ntu \tilde{\lambda}_1(\tu,\sek) \cdot \tilde{r}_{i}(\tu,\sekk)  \right)\wok \wikk:=\sum_{m=1}^5 I_m.\nonumber
\end{align}
One next needs to reexamine each term in \eqref{hh2}.
Indeed, by {\asmp{\tb}}, it holds that
\begin{equation}\label{hhhh4}
I_1 =\tilde{l}_1(\tu )    \frac{\partial\tq }{\partial \tu_1}(\tu)   \tilde{w}_{1,k'}+\sum_{j=2}^{n}   \tilde{l}_1(\tu )   \ntu \tq (\tu) \tilde{r}_j(\tu, \sekk) \tilde{w}_{j,k'} ,
\end{equation}
\begin{equation}
I_2 =-\sum_{j=2}^{n}   \tilde{l}_1(\tu ) \ntu \tilde{r}_{j} (\tu, \sekk ) \tq(\tu )  \tilde{w}_{j,k'},
\end{equation}
\begin{equation}
I_3=  \sum_{j=2}^{n}\sum_{i'=1}^{n} \sum_{k=1}^{d}  \tilde{l}_1(\tu )  \ntu \tilde{r}_j(\tu,\sekk ) \tilde{r}_{i'}(\tu,\sek)
 \tilde\lambda_{i'}(\tu,\sek) \tilde{w}_{j,k'}\wiik,
\end{equation}
 \begin{equation}
 I_4=- \sum_{j=2}^{n}\sum_{i'=1}^{n} \sum_{k=1}^{d} \tilde{l}_1(\tu )   \ntu \tilde{r}_j(\tu, \sek )\tilde{r}_{i'}(\tu,\sekk ) \tilde{\lambda}_j(\tu,\sek)  \tilde{w}_{j,k} \wiikk
\end{equation}
and
\begin{equation}\label{hhhh1}
 I_5=-  \sum_{k=1}^{d}   \frac{\partial \tilde{\lambda}_1}{\partial \tu_1}(\tu,\sek)  \wok \tilde{w}_{1,k'} - \sum_{j=2}^{n} \sum_{k=1}^{d} \left( \ntu \tilde{\lambda}_1(\tu,\sek) \cdot \tilde{r}_{j}(\tu,\sekk)  \right)\wok \tilde{w}_{j,k'} .
\end{equation}
Lastly, notice that the left hand side of \eqref{hh2} is not a transported one,
and it causes the inconvenience in doing the energy estimates.
One may artificially recover this with the following:
\begin{equation}
\begin{split}
&\left( \wokk\right)_{x_{k}} - \left( \wok\right)_{x_{k'}} = \left(\tilde{l}_1(\tu)  \tu_{x_{k'}} \right) _{x_k} - \left(\tilde{l}_1(\tu)  \tu_{x_{k}} \right) _{x_{k'}}
\\&\quad= \sum_{i,i'=1}^{n} \tilde{r}_i^T(\tu,\sek) \left( \ntu \tilde{l}_1(\tu) - \left(\ntu \tilde{l}_1(\tu )\right)^T \right) \tilde{r}_{i'}(\tu,\sekk ) \wik\wiikk.
\end{split}
\end{equation}
Moreover, for $i=i'=1$ in the summation, by {\asmp{\tb}},
\begin{equation}
\tilde{r}_1^T(\tu,\sek) \left( \ntu \tilde{l}_1(\tu) - \left(\ntu \tilde{l}_1(\tu )\right)^T \right) \tilde{r}_{1}(\tu,\sekk )
= e_1^T \left( \ntu \tilde{l}_1(\tu) - \left(\ntu \tilde{l}_1(\tu )\right)^T \right) e_1=0.
\end{equation}
Hence,
\begin{equation}\label{hhhh5}
 \left( \wok\right)_{x_{k'}}
  = \left( \wokk\right)_{x_{k}} -\sum_{i+i'\ge 3}^{n} \tilde{r}_i^T(\tu,\sek) \left( \ntu \tilde{l}_1(\tu) - \left(\ntu \tilde{l}_1(\tu )\right)^T \right) \tilde{r}_{i'}(\tu,\sekk ) \wik\wiikk
\end{equation}
Collecting these facts \eqref{hhhh4}--\eqref{hhhh1} and \eqref{hhhh5} together with \eqref{wjj}, one may thus deduce from \eqref{hh2}
that
\begin{align}\label{hhhhh111}
& \p_t \wokk + \sum_{k=1}^{d} \tilde{\lambda}_1(\tu, \sek )  \left( \wokk\right)_{x_{k}}\nonumber  \\
&\quad =\tilde{l}_1(\tu )    \frac{\partial\tq }{\partial \tu_1}(\tu)   \tilde{w}_{1,k'}-  \sum_{k=1}^{d}   \frac{\partial \tilde{\lambda}_1}{\partial \tu_1}(\tu,\sek)  \wok \tilde{w}_{1,k'}
+ \sum_{j=2}^{n}   \ts_{j,k'}(\tu) \tu_{j,x_{k'}}  \\
&\qquad+ \sum_{j=2}^{n}\sum_{i'=1}^{n} \sum_{k,k^*=1}^{d} \ttt_{ji'\!,kk^*}(\tu) \tu_{j,x_{k}}\tu_{i'\!,x_{k^*}}.\nonumber
 \end{align}

Now one may prove $(i)$ under \asmp{\twd 1}--\asmp{\twd 2} and that $d\ge3$. Assume \asmp{\twd 1}--\asmp{\twd 2}, then it holds that $\frac{\partial\tq}{\partial \tu_1}(\tilde{u}_1 e_1) \equiv 0$ and $\frac{\partial\tilde\lambda_1}{\partial \tu_1}(\tilde{u}_1 e_1,\sek) \equiv 0$.
Using this fact, one may fix $\tu_1$ and view $\frac{\partial\tq}{\partial \tu_1} (\tilde{u} )$ and $\frac{\partial \tilde\lambda_1}{\partial \tu_1} (\tilde{u},\sek)$ as functions of $\tub=(\tu_2,\dots,\tu_n)^T$.
By the mean value theorem, it holds that
\begin{equation}
\frac{\partial\tq_i}{\partial \tu_1} (\tilde{u} ) =\frac{\partial\tq_i}{\partial \tu_1} (\tilde{u}_1e_1 )+\sum_{j=2}^n\tv_{ij}(\tu )\tu_j\equiv\sum_{j=2}^n\tv_{ij}(\tu )\tu_j
\end{equation}
and
\begin{equation}
\frac{\partial\tilde\lambda_1}{\partial \tu_1} (\tilde{u},\sek) =\frac{\partial \tilde\lambda_1}{\partial \tu_1} (\tilde{u}_1e_1,\sek)+\sum_{j=2}^n\tw_{j,k}(\tu )\tu_j\equiv\sum_{j=2}^n\tw_{j,k}(\tu)\tu_j.
\end{equation}
Also,
\begin{equation}\label{mn1}
 \tilde\lambda_1 (\tilde{u},\sek) =\tilde\lambda_1 (\tilde{u}_1 e_1,\sek)+ \sum_{j=2}^n\tw^*_{j,k}(\tu )\tu_j\equiv\tilde\lambda_1 (0,\sek)+ \sum_{j=2}^n\tw^*_{j,k}(\tu )\tu_j.
\end{equation}
Hence, under \asmp{\twd 1}--\asmp{\twd 2}, it holds that
\begin{equation}\label{glgl}
 \tilde{l}_1(\tu )    \frac{\partial\tq }{\partial \tu_1}(\tu)   \tilde{w}_{1,k'} =\sum_{i=1}^n\sum_{j=2}^n  \tilde{l}_{1i}(\tu ) \tv_{ij}(\tu )\tu_j \tilde{w}_{1,k'}  ,
\end{equation}
and
\begin{equation}\label{glgl2}
 \sum_{k=1}^{d}   \frac{\partial \tilde{\lambda}_1}{\partial \tu_1}(\tu,\sek)  \wok \tilde{w}_{1,k'}=\sum_{k=1}^{d}\sum_{j=2}^n\tw_{j,k}(\tu)\tu_j\wok \tilde{w}_{1,k'}  .
\end{equation}
Substituting \eqref{glgl}--\eqref{glgl2} into \eqref{hhhhh111}, one can get
\begin{equation}\label{hhhhh11122}
\begin{split}
& \p_t \wokk + \sum_{k=1}^{d} \tilde{\lambda}_1(\tu, \sek )  \left( \wokk\right)_{x_{k}}  \\
&\quad =\sum_{j=2}^n\tx_j(\tu,\nabla\tu)\tu_j
+ \sum_{j=2}^{n}   \ts_{j,k'}(\tu) \tu_{j,x_{k'}}  + \sum_{j=2}^{n}\sum_{i'=1}^{n} \sum_{k,k^*=1}^{d} \ttt_{ji'\!,kk^*}(\tu) \tu_{j,x_{k}}\tu_{i'\!,x_{k^*}}
:=\tnf.
 \end{split}
\end{equation}
The $\ldl1 $ energy estimates on \eqref{hhhhh11122} yields
\begin{equation}\label{hhhhhhh0}
\begin{split}
& \hal\ddt\norm{\ldl1 \wokk}_{L^2}^2 + \sum_{k=1}^{d} \int \tilde{\lambda}_1(\tu, \sek )  \left(\ldl1  \wokk\right)_{x_{k}}\ldl1 \wokk  \\
&\quad = \int \ldl1  \tnf \ldl1 \wokk -\sum_{k=1}^{d}\int \left[\ldl1, \tilde{\lambda}_1(\tu, \sek )\right]  \left(\wokk\right)_{x_{k}}\ldl1 \wokk.
 \end{split}
\end{equation}
Integrating by parts, by \eqref{mn1}, one has
\begin{equation}\label{hhhhhhh1}
\begin{split}
& \left| \int \tilde{\lambda}_1(\tu, \sek )  \left(\ldl1  \wokk\right)_{x_{k}}\ldl1 \wokk \right|
=\hal \left| \int \tilde{\lambda}_1(\tu, \sek )  \left(\abs{\ldl1  \wokk}^2\right)_{x_{k}} \right| \\
&\quad = \hal \left|  \int  (\tw^*(\tu)\tub)_{x_k}  \abs{\ldl1  \wokk}^2 \right|
\\&\quad\ls \left(\norm{ \tub}_{L^\infty}+\norm{\nabla\tub}_{L^\infty} \right)\norm{\tu}_{H^\ell}^2.
 \end{split}
\end{equation}
By the commutator estimates \eqref{commutator estimate}, the product estimates \eqref{product estimate} from Lemma \ref{A2}, the expansion \eqref{mn1} and since $d \ge 3$, $\ell>d/2+1 > 2$, it yields
\begin{equation}
\begin{split}
& \norm{ \left[\ldl1, \tilde{\lambda}_1(\tu, \sek )\right]  \left(\wokk\right)_{x_{k}}}_{L^2}
\\&\quad \ls \norm{ \nabla \tilde{\lambda}_1(\tu, \sek ) }_{L^\infty}\norm{ \Lambda^{\ell-2} \left( \wokk\right)_{x_{k}}}_{L^2}+\norm{ \ldl1 \left( \tilde{\lambda}_1(\tu, \sek )\right) }_{L^\frac{2d}{d-2}}\norm{ \left(  \wokk\right)_{x_{k}}}_{L^d}
\\&\quad \ls  \norm{\nabla\left(\tw^*(\tu)\tub\right) }_{L^\infty}\norm{ \Lambda^{\ell-1}  \wokk }_{L^2}+\norm{ \Lambda^{\ell-1} \left(\tw^*(\tu)\tub\right) }_{L^\frac{2d}{d-2}}\norm{\nabla  \wokk}_{L^d}
\\&\quad\ls \left(\norm{\tub  }_{L^\infty}+ \norm{ \nabla   \tub  }_{L^\infty}+
\norm{\ldll \tub}_{L^2} \right) \norm{\tu}_{H^\ell}.
 \end{split}
\end{equation}
Lastly, one may apply the product estimates \eqref{product estimate} from Lemma \ref{A2} to obtain, by the expression of $\tnf$ in \eqref{hhhhh11122},
\begin{equation}\label{hhhhhhh5}
\begin{split}
&  \norm{ \ldl1  \tnf}_{L^2} \ls  \norm{     \tub  }_{L^\infty}+\norm{ \nabla   \tub  }_{L^\infty}+\norm{ \Lambda^{\ell-1}   \tub  }_{L^2}+\norm{ \Lambda^{\ell}   \tub  }_{L^2} .
 \end{split}
\end{equation}
In view of these estimates \eqref{hhhhhhh1}--\eqref{hhhhhhh5}, it follows from \eqref{hhhhhhh0} that
\begin{equation}\label{hhhhhhh6}
 \ddt\norm{\ldl1 \wokk}_{L^2}^2\ls \left(\norm{     \tub  }_{L^\infty}+\norm{ \nabla   \tub  }_{L^\infty}+\norm{ \Lambda^{\ell-1}   \tub  }_{L^2}+\norm{ \Lambda^{\ell}   \tub  }_{L^2}\right)\norm{\tu}_{H^\ell}.
\end{equation}

Finally, integrating the inequality \eqref{hhhhhhh6} directly in time, by \eqref{4.32++}--\eqref{wjj} together with the estimates \eqref{energy es for global}, one has
\begin{equation}
  \norm{\tu(t)}_{H^\ell}^2
 \ls \norm{\tu_0}_{H^\ell}^2 +\int_0^t \left(\norm{     \tub  }_{L^\infty}+\norm{ \nabla   \tub  }_{L^\infty}+\norm{ \Lambda^{\ell-1}   \tub  }_{L^2}+\norm{ \Lambda^{\ell}   \tub  }_{L^2}\right)\norm{\tu}_{H^\ell} \rmd\tau.
\end{equation}
By Cauchy's inequality, this implies \eqref{energy es for global1212}.

One then turns to prove $(ii)$ under \asmp{\td 1}--\asmp{\td 2} and that $d\ge2$. Assume \asmp{\td 1}--\asmp{\td 2}, then it holds that $ \frac{\partial( J\tq) }{\partial \tu_1}(\tu) \equiv 0$ and $\frac{\partial \tilde\lambda_1}{\partial \tu_1} (\tilde{u},\sek) \equiv 0$.
By this fact, it follows from \eqref{hhhhh111} that
\begin{equation}\label{hhhhh111333}
\begin{split}
& \p_t \wokk + \sum_{k=1}^{d} \tilde{\lambda}_1(\tu, \sek )  \left( \wokk\right)_{x_{k}}  \\
&\quad =\tilde{l}_1(\tu ) \frac{\partial(J^{-1}) }{\partial \tu_1}(\tu)J(\tu)\tq(\tu)  \tilde{w}_{1,k'}
+ \sum_{j=2}^{n}   \ts_{j,k'}(\tu) \tu_{j,x_{k'}}  \\
&\qquad+ \sum_{j=2}^{n}\sum_{i'=1}^{n} \sum_{k,k^*=1}^{d} \ttt_{ji'\!,kk^*}(\tu) \tu_{j,x_{k}}\tu_{i'\!,x_{k^*}}
:=\tng.
 \end{split}
\end{equation}
The $\ldl1 $ energy estimates on \eqref{hhhhh111333} yields
\begin{equation}\label{hhhhhhh02}
\begin{split}
& \hal\ddt\norm{\ldl1 \wokk}_{L^2}^2 + \sum_{k=1}^{d} \int \tilde{\lambda}_1(\tu, \sek )  \left(\ldl1  \wokk\right)_{x_{k}}\ldl1 \wokk  \\
&\quad = \int \ldl1  \tng \ldl1 \wokk -\sum_{k=1}^{d}\int \left[\ldl1, \tilde{\lambda}_1(\tu, \sek )\right]  \left(\wokk\right)_{x_{k}}\ldl1 \wokk.
 \end{split}
\end{equation}
Integrating by parts, by the linear degeneracy \asmp{\td 2}, one has
\begin{equation}\label{hhhhhhh12}
\begin{split}
& \left| \int \tilde{\lambda}_1(\tu, \sek )  \left(\ldl1  \wokk\right)_{x_{k}}\ldl1 \wokk \right|
=\hal \left| \int \tilde{\lambda}_1(\tu, \sek )  \left(\abs{\ldl1  \wokk}^2\right)_{x_{k}}  \right|  \\
&\quad = \hal \left| \sum_{j=2}^{n} \int \nabla_{\!\tu_j}\tilde{\lambda}_1(\tu, \sek )\left(\tu_j\right)_{x_{k}}  \abs{\ldl1  \wokk}^2 \right|
\\
&\quad \ls \norm{\nabla\tub}_{L^\infty} \norm{\ldl1  \wokk}_{L^2}^2.
 \end{split}
\end{equation}
By the commutator estimates \eqref{commutator estimate} from Lemma \ref{A2}, by again \asmp{\td 2}, it holds that
\begin{equation}
\begin{split}
& \norm{ \left[\ldl1, \tilde{\lambda}_1(\tu, \sek )\right]  \left(\wokk\right)_{x_{k}}}_{L^2}
\\&\quad \ls \norm{ \nabla \tilde{\lambda}_1(\tu, \sek ) }_{L^\infty}\norm{ \Lambda^{\ell-2} \left( \wokk\right)_{x_{k}}}_{L^2}+\norm{ \ldl1 \left( \tilde{\lambda}_1(\tu, \sek )\right) }_{L^\frac{2\pd}{\pd-2}}\norm{ \left(  \wokk\right)_{x_{k}}}_{L^{\pd}}
\\&\quad \ls \sum_{j=2}^{n}\norm{ \nabla_{\!\tu_j} \tilde{\lambda}_1(\tu, \sek )\nabla \tu_j }_{L^\infty}\norm{ \Lambda^{\ell-1}  \wokk }_{L^2}
\\&\qquad +\sum_{j=2}^{n} \norm{ \Lambda^{\ell-2}\left(\nabla_{\!\tu_j} \tilde{\lambda}_1(\tu, \sek )\nabla \tu_j\right) }_{L^\frac{2\pd}{\pd-2}}\norm{\nabla  \wokk}_{L^\pd}
\\&\quad\ls  \left(\norm{ \nabla   \tub  }_{L^\infty}+ \norm{ \Lambda^{\ell-1}   \tub  }_{L^2} +\norm{ \Lambda^{\ell}   \tub  }_{L^2} \right)\norm{\tu}_{H^\ell},
 \end{split}
\end{equation}
where $\pd =  d$ for $d \geq 3$; $\pd =  4$ for $d = 2$ and $l \geq  {5}/{2}$; $\pd =   {2}/({3-l})$ for $d = 2$ and $2 < l < {5}/{2}.$

Lastly, one may apply the product estimates \eqref{product estimate} from Lemma \ref{A2} to obtain, by the expressions \eqref{eexx4}--\eqref{eexx44} of $\tq$,
\begin{equation}\label{hhhhhhh52}
\begin{split}
&  \norm{ \ldl1  \tng}_{L^2} \ls  \norm{     \tub  }_{L^\infty}^2+\norm{     \tud  }_{L^\infty}+\norm{ \nabla   \tub  }_{L^\infty}+\norm{ \Lambda^{\ell-1}   \tub  }_{L^2}+\norm{ \Lambda^{\ell}   \tub  }_{L^2} .
 \end{split}
\end{equation}
In view of these estimates \eqref{hhhhhhh12}--\eqref{hhhhhhh52}, it follows from \eqref{hhhhhhh02} that
\begin{equation}\label{hhhhhhh62}
 \ddt\norm{\ldl1 \wokk}_{L^2}^2\ls \left(\norm{     \tub  }_{L^\infty}^2+\norm{     \tud  }_{L^\infty}+\norm{ \nabla   \tub  }_{L^\infty}+\norm{ \Lambda^{\ell-1}   \tub  }_{L^2}+\norm{ \Lambda^{\ell}   \tub  }_{L^2}\right)\norm{\tu}_{H^\ell} .
\end{equation}
This is an improvement of \eqref{hhhhhhh6} which then, together with the estimates \eqref{energy es for global+}, leads to \eqref{energy es for global112} in this case.
\end{proof}

\section{$L^q$ estimates under \asmp{\td 3}}\label{sec:6+}

This section is devoted to derive an $L^q$ estimates of the solution in preparation to prove Theorem \ref{thm:1.3} under \asmp{\td 1}--\asmp{\td 3} and the initial condition  $\tu_{0,1}\in L^q$ with $1\le q\le 2$.
However, one may not be able to derive the $L^q$ estimates of $\tu_1$ by directly using the subsystem \eqref{2.45_Sys_1}.
The essential difficulty lies in the presence of the second summation involving $\tub$ on the left hand side.
One feasible way is going back to the original system \eqref{2.12} of $\tu$ and then resorting to the wave decomposition.

In the spirit of the one dimensional study \cite{john1974blowup,li1994bookglobal},
one may consider, owing to \asmp{\tb},
\begin{equation} \label{6+.2}
\tvv_1 = \tll_1(\tu) \tu .
\end{equation}
But by \eqref{2.35}, it holds that
\begin{equation} \label{6+.4}
 \tvv_1= \tu_1  + \sum_{j=2}^{n}\tll_{1j}(\tu)  \tu_j.
\end{equation}
The key point is to derive the $L^q$ estimates of $\tvv_1$, and the result is stated as the following.
\begin{prop} \label{prop:6+.1}
Under \asmp{\td 1}--\asmp{\td 3} and that $d \ge 2$ and $\tvv_{1,0}\in L^q$ with $1\le q\le 2$, it holds that
\begin{equation}\label{lpp}
\begin{split}
\norm{\tvv_1}_{L^{q}} & \ls \exp \left( \int_0^t \left( \norm{\tud }_{L^\infty} + \norm{\nabla \tub }_{L^\infty} + \norm{\tub}_{L^\infty}^2 \right) \rmd\tau\right)  \\
&\quad \times \left( \norm{\tvv_{1,0}}_{L^{q}} + \int_0^t \norm{\tub }_{L^{2q}} \left( \norm{\tud }_{L^{2q}} + \norm{\nabla \tub }_{L^{2q}} + \norm{\tub }_{L^{4q}}^2 \right) \rmd\tau\right).
\end{split}
\end{equation}
\end{prop}
\begin{proof}
First of all, one needs to derive the equation of $\tvv_1$. Multiplying the system \eqref{2.12} by $\tll_1(\tu)$,
by \eqref{2.330}, \eqref{2.35} and an use of \eqref{2.12} to substitute $\partial_t\tu$, one can get
\begin{align} \label{6+.5}
& \p_t \tvv_1 + \sum_{k=1}^{d} \tlam_1(\tu,\sek) (\tvv_1)_{x_k}\nonumber \\
&\quad = \tll_1(\tu) \tq(\tu) + \partial_t\tu^T \ntu \tll_1(\tu) \tu + \sum_{k=1}^{d} \tlam_1(\tu,\sek) \tu_{x_k}^T \ntu \tll_1(\tu) \tu\nonumber \\
&\quad = \tll_1(\tu) \tq(\tu) + \sum_{j=2}^{n} \sum_{i=1}^{n} \tq_i(\tu) \frac{\p \tll_{1j}(\tu)}{\p \tu_i} \tu_{j}  \\
& \qquad  + \sum_{k=1}^{d} \sum_{j,j'=2}^{n} \sum_{i=1}^{n} \left(\tlam_1(\tu,\sek) - \tlam_{j'}(\tu,\sek) \right) \wjjk \trr_{ij'}(\tu) \frac{\p \tll_{1j}(\tu)}{\p \tu_i} \tu_j. \nonumber
\end{align}
Here in the second equality the wave composition \eqref{4.32++} for $\tu_{x_k}$ is used.

By \asmp{\td 3}, \eqref{theta1}, \eqref{3.10++22} and \eqref{3.10++}, it holds that
\begin{equation} \label{C3eq}
\tll_{1p}(0) = 0, \hs p = r+1, \dots, n,
\end{equation}
which implies that
\begin{equation} \label{6+.6}
\tll_1(\tu) \tq(\tu) = \sum_{q=1}^{r} \tll_{1q}(\tu) \tq_q(\tu) + \sum_{p=r+1}^{n} \left( \tll_{1p}(\tu) - \tll_{1p}(0) \right)  \tq_p(\tu).
\end{equation}
Substituting \eqref{6+.6} and the expansions \eqref{eexx4}--\eqref{eexx44} (due to \asmp{\td 1}) into \eqref{6+.5}, using the expression \eqref{wjj} for $\wjjk$ and \eqref{6+.4} for $\tu_1$, one may deduce
\begin{equation} \label{6+.7}
\begin{split}
&\p_t \tvv_1 + \sum_{k=1}^{d} \tlam_1(\tu,\sek) (\tvv_1)_{x_k}
 \\ &\quad= \sum_{p=r+1}^{n} \tr_{1p}(\tu) \tvv_1 \tu_p + \sum_{p=r+1}^{n} \sum_{j=2}^{n} \tr_{jp}(\tu) \tu_j \tu_p+\sum_{k=1}^{d} \sum_{j,j'=2}^{n} \tH_{jj'\!,k}(\tu) \tu_j  (\tu_{j'}) _{x_k}\\
& \qquad+ \sum_{j,j'=2}^{n} \tga_{1jj'}(\tu) \tvv_1 \tu_j \tu_{j'}   + \sum_{j,j'\!,j''=2}^{n} \tga_{jj'j''}(\tu) \tu_j \tu_{j'} \tu_{j''}.
\end{split}
\end{equation}
The key point lying in \eqref{6+.7} is that unlike \eqref{2.45_Sys_1} there is no term involving $\tub$ on the left hand side,
which makes it possible to derive the $L^q$ estimate for $\tvv_1$.

Now multiplying the equation by $|v_1|^{q-2}\, v_1$ and then integrating over $\mathbb{R}^d$, one has
\begin{equation}\label{thth}
\begin{split}
  \frac{\rmd}{\rmd t} \norm{\tvv_1}_{L^{q}}^{q} & \ls \sum_{k=1}^{d} \abs{\int  \tlam_1(\tu,\sek) ( |\tvv_1|^{q})_{x_k} } \\&\quad+ \int \left(  |\tvv_1|^{q} |\tud| +   |\tvv_1|^{q-1} |\tub|( |\tud|+|\nabla \tub|)   +  |\tvv_1|^{q} |\tub|^2 + |\tvv_1|^{q-1} |\tub|^3\right).
\end{split}
\end{equation}
By the linear degeneracy \asmp{\td 2}, it holds that
\begin{equation}\label{tt1}
\abs{  \int  \tlam_1(\tu,\sek) ( |\tvv_1|^{q})_{x_k} }= \abs{  \sum_{j=2}^{n}   \frac{\p \tlam_1}{\p \tu_j} (\tu,\sek)( \tu_j) _{x_k} |\tvv_1|^{q} } \ls \norm{\tvv_1}_{L^{q}}^{q} \norm{\nabla \tub}_{L^\infty}.
\end{equation}
By H\"older's inequality, the second integral in the right hand side of \eqref{thth} is easily bounded by
\begin{equation}\label{tt2}
\begin{split}
 & \norm{\tvv_1}_{L^{q}}^{q} \norm{\tud}_{L^\infty} + \norm{\tvv_1}_{L^{q}}^{q-1} \norm{\tub}_{L^{2q}} \left( \norm{\tud}_{L^{2q}} + \norm{\nabla \tub}_{L^{2q}} \right) \\
& \quad+ \norm{\tvv_1}_{L^{q}}^{q} \norm{\tub}_{L^\infty}^2 + \norm{\tvv_1}_{L^{q}}^{q-1} \norm{\tub}_{L^{2q}}\norm{\tub}_{L^{4q}}^2.
\end{split}
\end{equation}
Plugging the estimates \eqref{tt1}--\eqref{tt2} into \eqref{thth} yields 
\begin{equation}\label{tt3}
\begin{split}
\frac{\rmd}{\rmd t} \norm{\tvv_1}_{L^{q}} \ls& \norm{\tvv_1}_{L^{q}} \left(  \norm{\tud}_{L^\infty} + \norm{\nabla \tub}_{L^\infty} + \norm{\tub}_{L^\infty}^2 \right)
 \\&+ \norm{\tub}_{L^{2q}} \left( \norm{\tud}_{L^{2q}} + \norm{\nabla \tub}_{L^{2q}}+ \norm{\tub}_{L^{4q}}^2 \right).
\end{split}
\end{equation}
Applying the Gronwall lemma to \eqref{tt3}, one can deduce \eqref{lpp}.
\end{proof}
\begin{remark}\label{v10data}
In the statement of Proposition \ref{prop:6+.1}, it is required that $\tvv_{1,0}\in L^q$ with $1\le q\le 2$. However, this can be guaranteed under the assumptions of Theorem \ref{thm:1.3}. Indeed, by \eqref{6+.4} and \eqref{C3eq}, it holds that
\begin{equation}
\tvv_1 =\tu_1+\sum_{q=2}^r \tll_{1q}(\tu) \tu_q+\sum_{p=r+1}^{n} \left( \tll_{1p}(\tu) - \tll_{1p}(0) \right)\tu_p,
\end{equation}
which implies, since $1\le p\le q\le 2$,
\begin{equation}
\norm{\tvv_{1,0}}_{L^q}\ls \norm{\tu_{1,0}}_{L^q}+\norm{\tuc_{0}}_{L^q}+\norm{\tud_{0}}_{L^{2q}}^2\ls \norm{\tu_{1,0}}_{L^q}+\norm{\tuc_{0}}_{L^p}+\norm{\tu_{0}}_{H^{\ell}}^2.
\end{equation}
\end{remark}

\section{Decay estimates}\label{en3333333}

In order to close the estimates stated in Proposition \ref{energy energy} and Proposition \ref{prop:6+.1},
this section is devoted to derive the decay estimates for the dissipative component $\tub$.
\subsection{Bootstrap estimates}
In this subsection, one may derive the bootstrap energy estimates based on the family of scaled energy estimates with {\it minimum fractional} derivative counts derived in Section  \ref{en1111111}, which implies that the decay of the higher-order norms of $\tub(t)$ can be deduced from the decay of the lower-order norms.
\begin{lem}\label{good}

There exists a positive constant $\lambda>0$ such that the following estimates hold.

$(i)$ Under   \asmp{\twd 1} and that $d\ge 3$, for $1\le s< d/2$,
\begin{equation}\label{00000}
 \norm{\lds \tub(t)}_{H^{\ell-s}}^2\ls e^{-\lambda t}\norm{\lds \tub_0}_{H^{\ell-s}}^2 +\int_0^te^{-\lambda(t-\tau)} \norm{\lds\tuc(\tau)}_{L^2}^2 \rmd\tau.
\end{equation}

$(ii)$ Under   \asmp{\td 1}, for $1\le s\le \ell-1$ and $1\le s<d/2+1$,
\begin{equation}
\begin{split}
\norm{\lds \tub(t)}_{H^{\ell-s}}^2\ls & e^{-\lambda t}\norm{\lds \tub_0}_{H^{\ell-s}}^2  \\
& +\int_0^te^{-\lambda(t-\tau)} \left(\norm{\lds\tuc(\tau)}_{L^2}^2 + \delta  \norm{\Lambda^{s-1}  \tud (\tau)}_{L^2}^2+ \delta \norm{ \tub(\tau)  }_{L^\infty}^4 \right)  \rmd \tau. \label{000001}
\end{split}
\end{equation}
\end{lem}
\begin{proof}
One may first prove $(i)$ under \asmp{\twd 1} and that $d\ge 3$.
Recall the estimates \eqref{energy es for decay} of Proposition \ref{enes5} and then fix $1\le s< d/2$ (and hence $s<\ell-1$). Since $\ell>d/2+1$,
\begin{equation}\label{gga}
\norm{  \tub  }_{L^\infty}^2+\norm{\na \tub  }_{L^\infty}^2\ls \norm{\Lambda^s  \tub }_{H^{\ell-s}}^2,
 \end{equation}
then by the smallness of $\delta$, it follows from \eqref{energy es for decay} of Proposition \ref{enes5} that
\begin{equation}\label{ggg}
 \ddt\mathcal{E}_s(\tub)+  \norm{\Lambda^{s } \tud}_{H^{\ell-s }}^2 +  \norm{\Lambda^{s+1} \tuc}_{H^{\ell-s-1}}^2
 \ls \delta  \norm{\lds\tuc}_{L^2}^2 .
\end{equation}
Adding $\norm{\lds\tuc}_{L^2}^2$ to both sides of \eqref{ggg} yields
\begin{equation}
 \ddt\mathcal{E}_s(\tub) +  \norm{\Lambda^{s} \tub}_{H^{\ell-s}}^2
 \ls  \norm{\lds\tuc}_{L^2}^2  .
\end{equation}
Since $\mathcal{E}_s(\tub)$ is equivalent to $\norm{\lds \tub}_{H^{\ell-s}}^2$, it then holds that
\begin{equation}\label{jj1}
 \ddt\mathcal{E}_s(\tub)+ \lambda \mathcal{E}_s(\tub)
 \ls \norm{\lds\tuc}_{L^2}^2
\end{equation}
for some constant $\lambda>0$. Applying the Gronwall lemma to \eqref{jj1}, one concludes \eqref{00000}.

One then turns to prove $(ii)$ under \asmp{\td 1}.
Recall the estimates \eqref{energy es for decay1} of Proposition \ref{enes5} and then fix $s$ so that $1\le s\le \ell-1$ and $1\le s<d/2+1$.
Then it holds that
\[
\norm{\nabla \tub}_{L^\infty}^2 \ls \norm{\Lambda^s \tub}_{H^{\ell-s}}^2 \  \text{and} \  \norm{ \tud}_{L^\infty}^2 \ls \norm{\Lambda^{s-1} \tud}_{H^{\ell-s}}^2.
\]
It thus follows from \eqref{energy es for decay1} that
 \begin{equation} \label{ggg1}
\begin{split}
 &\ddt\mathcal{E}_s(\tub)+  \norm{\Lambda^{s } \tud}_{H^{\ell-s }}^2 +  \norm{\Lambda^{s+1} \tuc}_{H^{\ell-s-1}}^2
\\&\quad \ls \delta \left(\pnorm{\lds \tuc }{2}^2+ \norm{ \Lambda^{s-1} \tud  }_{L^2}^2 +\norm{ \tub  }_{L^\infty}^4 \right)+\norm{\lds \tub}_{L^2}\norm{ \tub  }_{L^\infty}^2.
 \end{split}
\end{equation}
Adding $\norm{\lds\tuc}_{L^2}^2$ to both sides of \eqref{ggg1}, by Cauchy's inequality, one can obtain
\begin{equation}
 \ddt\mathcal{E}_s(\tub) +  \norm{\Lambda^{s} \tub}_{H^{\ell-s}}^2
 \ls  \norm{\lds\tuc}_{L^2}^2+\delta \norm{ \Lambda^{s-1} \tud  }_{L^2}^2 + \delta \norm{ \tub  }_{L^\infty}^4,
\end{equation}
which implies
\begin{equation}\label{jj11}
 \ddt\mathcal{E}_s(\tub)+ \lambda \mathcal{E}_s(\tub)
 \ls   \norm{\lds\tuc}_{L^2}^2+ \delta\norm{ \Lambda^{s-1} \tud  }_{L^2}^2+ \delta \norm{ \tub  }_{L^\infty}^4  .
\end{equation}
Applying the Gronwall lemma to \eqref{jj11} yields \eqref{000001}.
\end{proof}

\subsection{Linear decay estimates}
In order to derive the decay of the lower-order norms of $\tub(t)$ appearing in Lemma \ref{good}, in this subsection one may consider the linear decay estimates for the linearized system of \eqref{2.47_Sys_2n}:
\begin{equation}\label{operator form l}
\begin{cases}
\partial_t \tub  + \tl \tub
=0,
\\ \tub\mid_{t=0}=\tub_0,
\end{cases}
\end{equation}
where the linear operator $\tl$ is defined by
\begin{equation}\label{lll}
\tl \tub =\sum_{k=1}^{d} \ta^{k,\flat}(0)   \tub_{x_k}-\Theta^\flat \tub.
\end{equation}

The solution operator ${ e}^{-t\tl}$ of the system \eqref{operator form l} has the following time-decay property:
\begin{lem}\label{linear decay lemma}
Let $d\ge 2$.
For $1\le p\le2$, $\al\ge 0$ and denote $p^\ast=\min\{2,dp/(d-p)\}$. There exists a constant $\lambda > 0$ such that
\begin{equation}\label{linear decay 1}
\norm{\Lambda^\al  \pic { e}^{-t\tl} \tub_0 }_{L^2}\ls (1+t)^{-  \frac{d}{2}\left( \frac{1}{p}-\frac{1}{2}\right)- \frac{\alpha}{2}}
\left(\norm{\pic \tub_0}_{L^p}+\norm{\Pi^\mathtt{D}\tub_0}_{L^{p^\ast}} \right) + e^{-\lambda t} \norm{\Lambda^\al \tub_0 }_{L^2},
\end{equation}
\begin{equation}\label{linear decay 1'}
\norm{\Lambda^\al  \pic { e}^{-t\tl} \tub_0 }_{L^2}\ls (1+t)^{-  \frac{d}{2}\left( \frac{1}{p}-\frac{1}{2}\right)- \frac{\alpha}{2}}
\left(\norm{\pic \tub_0}_{L^p}+\norm{\Lambda\Pi^\mathtt{D}\tub_0}_{L^{p}} \right) + e^{-\lambda t} \norm{\Lambda^\al \tub_0 }_{L^2},
\end{equation}
and
\begin{equation}\label{linear decay 2}
\norm{\Lambda^\al   \Pi^\mathtt{D}{ e}^{-t\tl} \tub_0 }_{L^2}\ls  (1+t)^{-  \frac{d}{2}\left( \frac{1}{p}-\frac{1}{2}\right)- \frac{\alpha}{2}-\hal}
\left(\norm{\pic \tub_0}_{L^p}+\norm{\Pi^\mathtt{D}\tub_0}_{L^{p^\ast}} \right) + e^{-\lambda t} \norm{\Lambda^\al \tub_0 }_{L^2},
\end{equation}
\begin{equation}\label{linear decay 2'}
\norm{\Lambda^\al  \Pi^\mathtt{D}{ e}^{-t\tl} \tub_0 }_{L^2}\ls (1+t)^{-  \frac{d}{2}\left( \frac{1}{p}-\frac{1}{2}\right)- \frac{\alpha}{2}-\hal}
\left(\norm{\pic \tub_0}_{L^p}+\norm{\Lambda\Pi^\mathtt{D}\tub_0}_{L^{p}} \right) + e^{-\lambda t} \norm{\Lambda^\al \tub_0 }_{L^2},
\end{equation}
where $\pic $ and $\Pi^\mathtt{D}$ are the projection operators defined by
\begin{equation}
 \pic a^\flat=(a_{2},\dots,a_r)^T,\ \Pi^\mathtt{D}a^\flat=(a_{r+1},\dots,a_n)^T,\quad \text{for} \ a^\flat=(a_2,\dots,a_n)^T\in \mathbb{R}^{n-1}.
\end{equation}
\end{lem}
\begin{proof}
These decay estimates follow by refining the initial condition in \cite{bianchini2007CPAM}. Indeed, since similar as in Lemma \ref{kaw} the linear system \eqref{operator form l} satisfies the Kawashima condition, from Section 4.2 of \cite{bianchini2007CPAM}, one can decompose the solution $\tub(t)$ of \eqref{operator form l} as
\begin{equation}
\tub(t)={e}^{-t\tl}\tub_0=K(t)\tub_0+\mathcal{K}(t)\tub_0,
\end{equation}
and the two components satisfy the following bounds:
\begin{equation}
\abs{\pic  \widehat{ K(t)\tub_0}  }\le C e^{-c|\xi|^2 t}\left(\abs{\pic \widehat{ \tub_0}}+|\xi|\abs{\Pi^\mathtt{D}\widehat{ \tub_0}}\right),
\end{equation}
\begin{equation}
\abs{\Pi^\mathtt{D} \widehat{ K(t)\tub_0}  }\le C e^{-c|\xi|^2 t}\left(|\xi|\abs{\pic \widehat{ \tub_0}}+|\xi|^2\abs{\Pi^\mathtt{D}\widehat{ \tub_0}}\right)
\end{equation}
and
\begin{equation}
\abs{\widehat{\mathcal{K}(t)\tub_0}}\le C e^{-ct}\abs{\widehat{ \tub_0}}
\end{equation}
for some constants $C,c>0$. Then the estimates \eqref{linear decay 1}--\eqref{linear decay 2'} follow in the standard way.
\end{proof}

\subsection{Nonlinear estimates}
In order to use the linear decay estimates of Lemma \ref{linear decay lemma},
one may rewrite the subsystem \eqref{2.47_Sys_2n} in the following perturbed operator form
\begin{equation}\label{operator form}
\partial_t \tub  + \tl \tub
=\tn,
\end{equation}
where the linear operator $\tl$ is defined by \eqref{lll}
and the nonlinear term
\begin{equation}\label{nnn}
\tn= -\sum_{k=1}^{d} \left(\ta^{k,\flat}(\tu)- \ta^{k,\flat}(0)  \right) \tub_{x_k}+\tq^\flat(\tu)-\Theta^\flat \tub  .
\end{equation}

For the convenience of the later analysis, one may record the following estimates of $\tn$.
\begin{lem}\label{non le}

The following estimates hold.

$(i)$ Under   \asmp{\twd 1} and that $d\ge 3$, for $  1\le s\le \ell-1,$
\begin{equation}  \label{nonn1}
\norm{\lds\tn}_{L^2} + \norm{\ldso\tn}_{L^2}
\ls\delta \left(\norm{\Lambda^{s+1} \tub }_{L^2}+\norm{\lds \tub }_{L^2}+\norm{\na \tub }_{L^\infty}+\norm{\tub }_{L^\infty}\right),
\end{equation}
for $1\le p<2,$
\begin{equation}\label{nonn55}
\norm{\tnc }_{L^p} \ls \delta\left(\norm{\Lambda^{d\left(1-\frac{1}{p}\right)+1} \tub }_{L^{2}}+\norm{\Lambda^{d\left(1-\frac{1}{p}\right)}\tud }_{L^{2}}+\norm{\tub }_{L^\infty}\right)+\norm{\Lambda^{\frac{d}{2}\left(1-\frac{1}{p}\right) }\tub }_{L^{2}}^2,
\end{equation}
for $1\le p<2d/(d+2)$,
\begin{equation}\label{nonn6}
\norm{\tn }_{L^\frac{dp}{d-p}}\ls \delta\left(\norm{\Lambda^{d\left(1-\frac{1}{p}\right)+2}\tub }_{L^2}+\norm{ \Lambda^{d\left(1-\frac{1}{p}\right)+1} \tub }_{L^2}\right),
\end{equation}
and
\begin{equation}\label{121}
\norm{\tn }_{L^2}  \ls \delta\left(\norm{\na \tub }_{L^{\infty}}+\norm{\tub }_{L^\infty}\right).
\end{equation}

$(ii)$ Under   \asmp{\td 1}, for $  0\le s\le \ell-1,$
\begin{equation} \label{nonn11}
\begin{split}
\norm{\lds\tn}_{L^2}  \ls&\delta \left(\norm{\Lambda^{s+1} \tub }_{L^2}+\norm{\lds \tud }_{L^2}+\norm{\na \tub }_{L^\infty}+\norm{\tud }_{L^\infty}\right)
\\&+ \norm{\tub }_{L^\infty}\left(\norm{\Lambda^{s } \tub }_{L^2}+\norm{\tub }_{L^\infty}\right) ,
\end{split}
\end{equation}
for $1\le p<2,$
\begin{equation}\label{nn5'}
\norm{\tnc }_{L^p} \ls \delta\left(\norm{\Lambda^{d\left(1-\frac{1}{p}\right)+1} \tub }_{L^{2}}+\norm{\Lambda^{d\left(1-\frac{1}{p}\right)}\tud }_{L^{2}}\right)+\norm{\Lambda^{d\left(1-\frac{1}{p}\right)}\tub }_{L^{2}}\norm{\tub }_{L^\infty},
\end{equation}
and
\begin{equation} \label{nn5}
\norm{\tnc }_{L^2} \ls \delta\left( \norm{\na \tub }_{L^{\infty}}+\norm{\tud }_{L^\infty}
\right)+  \norm{\tub }_{L^\infty}^2,
\end{equation}
for $1\le p<2d/(d+2)$,
\begin{equation}  \label{nonn61}
\norm{\tn }_{L^\frac{dp}{d-p}} \ls \delta\left(\norm{\Lambda^{d\left(1-\frac{1}{p}\right)+2}\tub }_{L^2}+\norm{ \Lambda^{d\left(1-\frac{1}{p}\right)+1} \tud }_{L^2} \right) +\norm{ \Lambda^{\frac{d}{2} \left(1-\frac{1}{p}\right)+ \frac{1}{2}} \tub }_{L^2}^2,
 \end{equation}
and
\begin{equation}\label{nonn612}
 \norm{\tn }_{L^2} \ls \delta\left(\norm{\na \tub }_{L^{\infty}}+\norm{\tud }_{L^{\infty}} \right) + \norm{\tub }_{L^{2}} \norm{\tub }_{L^{\infty}},
\end{equation}
for $2d/(d+2)< p< 2$,
 \begin{equation}\label{nonn6123}
 \begin{split}
\norm{\Lambda\tn }_{L^p} \ls &\delta\left(\norm{\Lambda^{\frac{d}{2} + 1} \tub }_{L^{2}}+\norm{\Lambda^{d\left(1-\frac{1}{p}\right)+1} \tub }_{L^{2}}+\norm{\Lambda^{d\left(1-\frac{1}{p}\right)}\tud }_{L^{2}} \right)\\
&+ \norm{\La^{d(1-\frac{1}{p})} \tub }_{L^2} \norm{\tub}_{L^\infty} ,
\end{split}
\end{equation}
and
\begin{equation}\label{nonn61235}
\norm{\Lambda\tn }_{L^2} \ls \delta\left(\norm{\Lambda^{\frac{d}{2}+1} \tub }_{L^{2}}+\norm{ \na \tub }_{L^{\infty}}+\norm{ \tud }_{L^\infty} \right)+ \norm{\tub}_{L^\infty}^2.
\end{equation}

$(iii)$ Moreover, under   \asmp{\td 1} and that $\norm{\tvv_1}_{L^{q}} \ls \delta$ with $1\le p\le q\le 2$, for $1\le p<2,$
\begin{equation} \label{7.30+}
\begin{split}
\norm{\tnc}_{L^p} & \ls \delta \left( \norm{\La^{d( \frac{1}{2} + \frac{1}{q} - \frac{1}{p} )+1} \tub }_{L^2} + \norm{\La^{d( \frac{1}{2} + \frac{1}{q} -\frac{1}{p} )} \tud }_{L^2} \right) \\
& \quad + \norm{\La^{d(1-\frac{1}{p})} \tub}_{L^2} \left(\norm{\tub}_{L^\infty} + \norm{\tud}_{L^2} + \norm{\La \tub}_{L^2} \right) ,\textrm{ if }p<q;
\end{split}
\end{equation}
\begin{equation} \label{7.300+}
\begin{split}
\norm{\tnc}_{L^p} & \ls \delta \left( \norm{\na  \tub }_{L^\infty} + \norm{ \tud }_{L^\infty} \right) \\
& \quad + \norm{\La^{d(1-\frac{1}{p})} \tub}_{L^2} \left(\norm{\tub}_{L^\infty} + \norm{\tud}_{L^2} + \norm{\La \tub}_{L^2} \right),\textrm{ if }p=q,
\end{split}
\end{equation}
and
for $2d/(d+2)< p< 2$,
\begin{align}  \nonumber
\norm{\La \tn}_{L^p} & \ls \delta \left( \norm{\La^{\frac{d}{2}+1} \tub}_{L^2} + \norm{\La^{d( \hal + \frac{1}{q} - \frac{1}{p} ) +1} \tub}_{L^2} + \norm{\La^{d( \hal + \frac{1}{q} -\frac{1}{p} )} \tud}_{L^2}+ \norm{\La^{d(1-\frac{1}{p})+1} \tud }_{L^2} \right)\\
&\quad + \norm{\La^{d(1-\frac{1}{p})}\tub}_{L^2}    \norm{\tub}_{L^\infty}+\norm{\La^{d(1-\frac{1}{p})} \tub}_{H^1}\norm{\Lambda \tub}_{L^2}
 ,\textrm{ if }p<q;\label{7.322+}
\end{align}
\begin{equation} \label{7.32+}
\begin{split}
\norm{\La \tn}_{L^p} & \ls \delta \left( \norm{\La^{\frac{d}{2}+1} \tub}_{L^2} + \norm{\na  \tub }_{L^\infty} + \norm{ \tud }_{L^\infty}+ \norm{\La^{d(1-\frac{1}{p})+1} \tud }_{L^2} \right)\\
&\quad + \norm{\La^{d(1-\frac{1}{p})}\tub}_{L^2}    \norm{\tub}_{L^\infty}+\norm{\La^{d(1-\frac{1}{p})} \tub}_{H^1}\norm{\Lambda \tub}_{L^2}
,\textrm{ if }p=q.
\end{split}
\end{equation}
\end{lem}
\begin{proof}
One may first prove $(i)$ under \asmp{\twd 1} and that $d\ge 3$. In this case, recalling from the definition \eqref{nnn} of $\tn$ and the expansion \eqref{for decay1}--\eqref{for decay2}, one has
\begin{equation}
 \tnc=\ty(\tu)(\tu \na\tub+\tud \tu +\tub\tub )+\tz(\tu) \tub\tu\tu  ,
\end{equation}
and
\begin{equation}
 \tn=\ty(\tu)(\tu \na\tub+ \tub\tu).
\end{equation}
By the product estimates \eqref{product estimate} of Lemma \ref{A2}, it holds that for $  0\le s\le \ell-1,$
\begin{equation}\label{decay ee3}
\begin{split}
 \norm{\lds\tn}_{L^2}   &\ls  \norm{\ty(\tu)\tu }_{L^\infty}\left(\norm{\Lambda^{s+1} \tub }_{L^2}+\norm{\lds \tub  }_{L^2}\right)
 \\&\quad+\norm{\lds(\ty(\tu) \tu)  }_{L^2}\left(\norm{\na \tub }_{L^\infty}+\norm{\tub }_{L^\infty}\right)
 \\&\ls\delta \left(\norm{\Lambda^{s+1} \tub }_{L^2}+\norm{\lds \tub }_{L^2}+\norm{\na \tub }_{L^\infty}+\norm{\tub }_{L^\infty}\right)
 \end{split}
\end{equation}
and that for $  1\le s\le \ell-1,$ since $d\ge 3$,
\begin{equation}\label{decay ee333}
\begin{split}
 \norm{\Lambda^{s-1}\tn}_{L^2}   &\ls  \norm{\ty(\tu)\tu }_{L^d}\left(\norm{\Lambda^{s} \tub }_{L^\frac{2d}{d-2}}+\norm{\Lambda^{s-1} \tub }_{L^\frac{2d}{d-2}}\right) \\&\quad+\norm{\La^{s-1}(\ty(\tu) \tu ) }_{L^2}\left(\norm{\na \tub }_{L^\infty}+\norm{\tub }_{L^\infty}\right)
 \\&\ls\delta \left(\norm{\Lambda^{s+1} \tub }_{L^2}+\norm{\lds \tub }_{L^2}+\norm{\na \tub }_{L^\infty}+\norm{\tub }_{L^\infty}\right).
 \end{split}
\end{equation}
For $1\le p<2$, by H\"older's and Sobolev's  inequalities, one can get
\begin{equation}\label{decay ee1}
\begin{split}
\norm{\tnc }_{L^p}&\ls \norm{\ty(\tu)\tu }_{L^2}\left(\norm{\na \tub }_{L^{\frac{2p}{2-p}}}+\norm{\tud }_{L^\frac{2p}{2-p}}\right)
\\&\quad+\norm{\ty(\tu) }_{L^\infty}\norm{\tub }_{L^{2p}}^2+\norm{\tz(\tu)\tu }_{L^{2 }}\norm{\tu }_{L^\frac{2p}{2-p}}\norm{\tub }_{L^\infty}
\\&\ls \norm{\tub }_{L^{2p}}^2+\delta\left(\norm{\na \tub }_{L^{\frac{2p}{2-p}}}+\norm{\tud }_{L^{\frac{2p}{2-p}}}+\norm{\tub }_{L^\infty}\right)
\\&\ls \norm{\Lambda^{\frac{d}{2}\left(1-\frac{1}{p}\right)}\tub }_{L^{2}}^2+\delta\left(\norm{\Lambda^{d\left(1-\frac{1}{p}\right)}\tud }_{L^{2}}+\norm{\Lambda^{d\left(1-\frac{1}{p}\right)+1} \tub }_{L^{2}}+\norm{\tub }_{L^\infty}\right) .
\end{split}
\end{equation}
For $1\le p<2d/(d+2)$ and hence ${dp}/({d-p})<2$, $d(1-1/p) + 2 < d/2+1$, it holds that
\begin{equation}\label{decay ee11122}
\begin{split}
 \norm{\tn }_{L^\frac{dp}{d-p}}& \ls \norm{\ty(\tu)\tu }_{L^2}\left(\norm{\na \tub }_{L^\frac{dp}{d-p-{dp}/{2}}}+\norm{  \tub }_{L^\frac{dp}{d-p-{dp}/{2}}}\right)
\\&\ls \delta\left(\norm{\Lambda^{d\left(1-\frac{1}{p}\right)+2}\tub }_{L^2}+\norm{ \Lambda^{d\left(1-\frac{1}{p}\right)+1} \tub }_{L^2}\right) .
\end{split}
 \end{equation}
Thus, \eqref{decay ee3}--\eqref{decay ee11122} yield \eqref{nonn1}--\eqref{nonn6} correspondingly. On the other hand, \eqref{121} can be proved with the simple modification in \eqref{decay ee11122}.

Next, one may turn to prove $(ii)$ under \asmp{\td 1}. In this case, by the expansions \eqref{eexx4}--\eqref{eexx44}, the nonlinear terms can be expressed as
\begin{equation} \label{8.42}
 \tnc=\ty(\tu)(\tu \na\tub+\tud \tu )+\tz(\tu) \tub\tub\tu
\end{equation}
and
\begin{equation} \label{8.43}
 \tn=\ty(\tu)(\tu \na\tub+\tud \tu+\tub\tub).
\end{equation}
By the product estimates \eqref{product estimate} of Lemma \ref{A2}, it holds that for $  0\le s\le \ell-1,$
\begin{equation}\label{decay ee322'}
\begin{split}
\norm{\lds\tn}_{L^2}   &\ls  \norm{\ty(\tu)\tu }_{L^\infty}\left(\norm{\Lambda^{s+1} \tub }_{L^2}+\norm{\lds \tud }_{L^2}\right)\\
&\quad+\norm{\lds (\ty(\tu)\tu)  }_{L^2}\left(\norm{\na \tub }_{L^\infty}+\norm{\tud }_{L^\infty}\right)
\\&\quad+ \norm{ \ty(\tu) }_{L^\infty} \norm{\lds\tub }_{L^2} \norm{\tub}_{L^\infty} + \norm{\lds(\ty(\tu)-\ty(0))}_{L^2}\norm{\tub }_{L^\infty}^2
\\&\ls\delta \left(\norm{\Lambda^{s+1} \tub }_{L^2}+\norm{\lds \tud }_{L^2}+\norm{\na \tub }_{L^\infty}+\norm{\tud }_{L^\infty}\right)
\\&\quad+ \norm{\Lambda^{s } \tub }_{L^2}\norm{\tub }_{L^\infty}+ \delta \norm{\tub }_{L^\infty}^2.
 \end{split}
\end{equation}
For $1\le p<2$, by H\"older's and Sobolev's  inequalities, one can get
\begin{equation}\label{decay ee1'}
\begin{split}
\norm{\tnc }_{L^p}&\ls \norm{\ty(\tu)\tu }_{L^2}\left(\norm{\na \tub }_{L^{\frac{2p}{2-p}}}+\norm{\tud }_{L^\frac{2p}{2-p}}\right)
+\norm{\tz(\tu)\tu }_{L^{2 }}\norm{  \tub }_{L^{\frac{2p}{2-p}}}\norm{\tub }_{L^\infty}
\\&\ls  \delta\left(\norm{\na \tub }_{L^{\frac{2p}{2-p}}}+\norm{\tud }_{L^{\frac{2p}{2-p}}}+\norm{  \tub }_{L^{\frac{2p}{2-p}}}\norm{\tub }_{L^\infty}\right)
\\&\ls \delta\left(\norm{\Lambda^{d\left(1-\frac{1}{p}\right)+1} \tub }_{L^{2}}+\norm{\Lambda^{d\left(1-\frac{1}{p}\right)}\tud }_{L^{2}}+\norm{\Lambda^{d\left(1-\frac{1}{p}\right)}\tub }_{L^{2}}\norm{\tub }_{L^\infty}\right).
\end{split}
\end{equation}
For $1\le p<2d/(d+2)$, it holds that
\begin{equation}\label{decay ee11122'}
\begin{split}
\norm{\tn }_{L^\frac{dp}{d-p}}  & \ls \norm{\ty(\tu)\tu }_{L^2}\left(\norm{\na \tub }_{L^\frac{dp}{d-p-{dp}/{2}}}+\norm{  \tud }_{L^\frac{dp}{d-p-{dp}/{2}}} \right) +\norm{  \tub }_{L^\frac{2dp}{d-p}}^2
\\&\ls \delta\left(\norm{\Lambda^{d\left(1-\frac{1}{p}\right)+2}\tub }_{L^2}+\norm{ \Lambda^{d\left(1-\frac{1}{p}\right)+1} \tud }_{L^2} \right) +\norm{ \Lambda^{\frac{d}{2} \left(1-\frac{1}{p}\right)+\hal} \tub }_{L^2}^2.
\end{split}
\end{equation}
For $ 2d/(d+2)\le p< 2$, by the product estimates \eqref{product estimate} of Lemma \ref{A2}, it holds that
\begin{equation}\label{decay ee1'1}
\begin{split}
\norm{\Lambda\tn }_{L^p}&\ls \norm{\ty(\tu)\tu }_{L^{\frac{dp}{d-p}}} \norm{ \na^2 \tub }_{L^d} +
\norm{\ty(\tu)\tu }_{L^2} \norm{\na\tud }_{L^\frac{2p}{2-p}}
\\&\quad+\norm{\na(\ty(\tu)\tu) }_{L^2}\left(\norm{ \na \tub }_{L^{\frac{2p}{2-p}}}+\norm{ \tud }_{L^\frac{2p}{2-p}}\right)
\\
& \quad + \norm{ \na \tub }_{L^{\frac{2p}{2-p}}}\norm{\tub }_{L^2}+ \norm{\na \ty(\tu) }_{L^2} \norm{\tub}_{L^{\frac{2p}{2-p}}}\norm{\tub}_{L^\infty}
\\&\ls \delta\left(\norm{\Lambda^{\frac{d}{2}+1} \tub }_{L^{2}}+\norm{\Lambda^{d\left(1-\frac{1}{p}\right)+1} \tub }_{L^{2}} + \norm{\Lambda^{d\left(1-\frac{1}{p}\right)}\tud }_{L^{2}} \right)\\&\quad+\delta\norm{\La^{d(1-\frac{1}{p})} \tub }_{L^2} \norm{\tub}_{L^\infty}.
\end{split}
\end{equation}
Thus, \eqref{decay ee322'}--\eqref{decay ee1'1} yields \eqref{nonn11}--\eqref{nn5'}, \eqref{nonn61}, \eqref{nonn6123}, respectively. On the other hand, \eqref{nn5}, \eqref{nonn612}, \eqref{nonn61235} can be proved with the simple modifications in \eqref{decay ee1'}--\eqref{decay ee1'1}, respectively.

Finally, one proves $(iii)$ under \asmp{\td 1} and $\norm{\tvv_1}_{L^{q}} \ls \delta$ with $1\le p\le q\le 2$.
By  \eqref{6+.4}, the expansions \eqref{8.42}--\eqref{8.43} can be rewritten as
\begin{equation} \label{8.42'}
 \tnc=\ty(\tu)\tvv_1( \na\tub+\tud   )+\ty(\tu)\tub( \na\tub+\tud)+\tz(\tu) \tub\tub\tu
\end{equation}
and
\begin{equation} \label{8.43'}
 \tn =\ty(\tu)\tvv_1( \na\tub+\tud   )+\ty(\tu)\tub( \na\tub+\tub).
\end{equation}
One may first assume $p<q$. For $1\le p< 2$, it holds that
\begin{equation}\label{t11}
\begin{split}
\norm{\tnc}_{L^p} & \ls \norm{\ty(\tu) \tvv_1}_{L^{q}} \left( \norm{\nabla \tub}_{L^{\frac{q p}{q-p}}} + \norm{\tud}_{L^{\frac{q p}{q-p}}} \right) \\
& \quad + \norm{\ty(\tu) \tub}_{L^{\frac{2p}{2-p}}} \left( \norm{\nabla \tub}_{L^2} + \norm{\tud}_{L^2} \right) + \norm{\tz (\tu) \tu}_{L^2} \norm{\tub}_{L^{\frac{2p}{2-p}}} \norm{\tub}_{L^\infty} \\
& \ls \delta \left( \norm{\La^{ d(\hal + \frac{1}{q} -\frac{1}{p}) +1 }\tub}_{L^2} + \norm{\La^{d( \hal + \frac{1}{q} - \frac{1}{p} )} \tud}_{L^2} \right)  \\
&\quad + \norm{\La^{d(1-\frac{1}{p})} \tub}_{L^2} \left( \norm{\La \tub}_{L^2} + \norm{\tud}_{L^2} + \delta \norm{\tub}_{L^\infty} \right).
\end{split}
\end{equation}
For $2d/(d+2)\le p< 2$, it holds that
\begin{align}\label{t12}
\norm{\La \tn}_{L^p} & \ls \norm{\nabla(\ty(\tu)\tvv_1)}_{L^{q}} \left(\norm{\nabla \tub}_{L^{\frac{q p}{q-p}}} +   \norm{\tud}_{L^{\frac{q p}{q-p}}} \right) \nonumber\\& \quad + \norm{\ty(\tu) \tu}_{L^{\frac{dp}{d-p}}} \norm{\na^2 \tub}_{L^d} + \norm{\ty(\tu) \tu}_{L^2} \norm{\na \tud}_{L^{\frac{2p}{2-p}}} \nonumber\\ & \quad + \norm{\ty(\tu)}_{L^\infty} \left( \norm{\na \tub}_{L^2} \norm{\na \tub}_{L^\frac{2p}{2-p}} + \norm{\na \tub}_{L^2}\norm{\tub}_{L^{\frac{2p}{2-p}}}\right)\nonumber \\
& \quad+ \norm{\na   \ty(\tu)  }_{L^\infty}\norm{\tub}_{L^{\frac{2p}{2-p}}} \norm{\nabla \tub}_{L^2} +   \norm{\na   \ty(\tu)  }_{L^2}\norm{\tub}_{L^{\frac{2p}{2-p}}}   \norm{\tub}_{L^\infty}
\\&\ls  \delta \left( \norm{\La^{\frac{d}{2}+1} \tub}_{L^2} + \norm{\La^{d( \hal + \frac{1}{q} - \frac{1}{p} ) +1} \tub}_{L^2} + \norm{\La^{d( \hal + \frac{1}{q} -\frac{1}{p} )} \tud}_{L^2}+ \norm{\La^{d(1-\frac{1}{p})+1} \tud }_{L^2} \right)\nonumber\\
&\quad+\norm{\La^{d(1-\frac{1}{p})+1} \tub}_{L^2}\norm{\na \tub}_{L^2}+ \norm{\La^{d(1-\frac{1}{p})} \tub}_{L^2}\left( \norm{\na \tub}_{L^2}+\norm{\tub}_{L^\infty}\right). \nonumber
\end{align}
Here the following bound is used
\begin{equation}
\norm{\nabla \tvv_1}_{L^{q}} \ls \norm{\tvv_1}_{L^{q}} + \norm{\nabla\tvv_1}_{L^2} \ls \delta,
\end{equation}
which follows from an use of Lemma \ref{A3}.
Thus, \eqref{t11}--\eqref{t12} yield \eqref{7.30+}, \eqref{7.322+}, respectively. On the other hand, \eqref{7.300+}, \eqref{7.32+} can be proved with the simple modifications in \eqref{t11}--\eqref{t12}, respectively.
\end{proof}

\subsection{Nonlinear decay estimates}

Now one can prove the decay estimates of $\tub$:
\begin{prop}\label{decay decay}
The following decay estimates hold.

$(i)$ Under   \asmp{\twd 1} and that $d \geq 3$ and $1\le p\le 2,\ p< d/2 $,
then for all $s\in [ 0, s_1^\ast]\cap [0,d/2)$ with $s_1^*=   d(1-1/p) + 1 $,
\begin{equation}\label{u11}
\norm{\lds\tuc(t) }_{H^{\ell-s}}\ls \ttk (1+t)^{-\frac{d}{2}\left(\frac{1}{p}-\frac{1}{2}\right)-\frac{s}{2} },
\end{equation}
and for all $s\in [ 0, s_1^\ast-1]\cap [0,d/2-1)$,
\begin{equation}\label{u12}
\norm{\lds\tud(t) }_{H^{\ell-s}}\ls \ttk (1+t)^{-\frac{d}{2}\left(\frac{1}{p}-\frac{1}{2}\right)-\frac{s}{2}-\frac{1}{2} }.
\end{equation}

$(ii)$ Under   \asmp{\td 1} and that $d \geq 3$ and $1\le p\le 2,\ p< d$; requiring further that $p < 2d/(2(3-\ell)+d) $ if $\ell\le 3$,
then for all $s\in [ 0, s_2^\ast]\cap [0,d/2+1)$ with $s_2^*= \min\{ d(1-1/p) + 1,\ell-1\} $,
\begin{equation}\label{u1112}
\norm{\lds\tuc(t) }_{H^{\ell-s}}\ls \ttk (1+t)^{-\frac{d}{2}\left(\frac{1}{p}-\frac{1}{2}\right)-\frac{s}{2} },
\end{equation}
and for all $s\in [ 0, s_2^\ast-1]\cap [0,d/2)$,
\begin{equation}\label{u1212}
\norm{\lds\tud(t) }_{H^{\ell-s}}\ls \ttk (1+t)^{-\frac{d}{2}\left(\frac{1}{p}-\frac{1}{2}\right)-\frac{s}{2}-\frac{1}{2} }.
\end{equation}

$(iii)$ Under   \asmp{\td 1}, \asmp{\td 3} and that $d \geq 2$ and $\norm{\tvv_1}_{L^{q}} \ls \delta$ with $1\le p\le q\le 2,\ q< d $; requiring further that $p < 2d/(2(3-\ell)+d) $ if $\ell\le 3$,  then for all $s\in [0,s_3^*] \cap [0,d/2+1)$ with $s_3^* = \min \{ d(1/2 + 1/q - 1/p) + 1, d(1-1/p) +2, \ell-1 \}$,
\begin{equation} \label{7.50+}
\norm{\La^s \tuc(t)}_{H^{\ell-s}} \ls \ttk (1+t)^{-\frac{d}{2}(\frac{1}{p}-\hal) -\frac{s}{2}},
\end{equation}
and for all $s \in [0,s_3^*-1] \cap [0,d/2)$,
\begin{equation} \label{7.51+}
\norm{\La^s \tud(t)}_{H^{\ell-s}} \ls \ttk (1+t)^{-\frac{d}{2}(\frac{1}{p}-\hal) -\frac{s}{2}-\hal}.
\end{equation}
Here $K_0 =\norm{u_0}_{H^\ell}+\norm{   \tuc_0 }_{L^{p }}+\norm{\tud_0 }_{L^{p^\ast}}$ is sufficiently small.
\end{prop}
\begin{proof}
One may first prove $(i)$ under \asmp{\twd 1} and that $d\ge 3$ and $1\le p\le 2$, $p<d/2$.
Denote $\sigma_p= {d}/{2} ( {1}/{p}- {1}/{2})$ and take ${\tilde{s}_1}  = \min\{ s_1^*, d/2-\epsilon \}\equiv\min\{ d(1-1/p) + 1,d/2-\epsilon \}<d/2$ for a sufficiently small positive constant $\epsilon$. 
Then define
\begin{equation}
\begin{split}
 \m_1(t)=\sup_{0\le \tau\le t}&\left\{(1+\tau)^{2\sigma_p+\tilde{s}_1  } \left( \norm{\Lambda^{\tilde{s}_1  } \tub (\tau)}_{H^{\ell-\tilde{s}_1  }}^2+ \norm{\Lambda^{\tilde{s}_1  -1} \tud  (\tau)}_{L^2}^2 \right) \right.
 \\& \left.\quad+(1+\tau)^{2\sigma_p}   \norm{  \tub (\tau)}_{L^2}^2+(1+\tau)^{2\sigma_p+1}  \norm{ \tud(\tau)  }_{L^2}^2 \right\}  .
 \end{split}
\end{equation}

First, taking $s=\tilde{s}_1  $ in \eqref{00000} yields
\begin{equation}\label{case1}
 \norm{\Lambda^{\tilde{s}_1  } \tub(t)}_{H^{\ell-\tilde{s}_1  }}^2\ls e^{-\lambda t}\ttkk  +\int_0^te^{-\lambda(t-\tau)} \norm{\Lambda^{\tilde{s}_1  }\tuc(\tau)}_{L^2}^2   \rmd\tau.
\end{equation}
On the other hand, by the Duhamel principle, it follows from \eqref{operator form} that
\begin{equation}\label{duhamel}
\tub(t)={ e}^{-t\tl} \tub_0 +\int_0^t { e}^{-(t-\tau)\tl}  \tn(\tau) \rmd\tau.
\end{equation}
Then one may apply the linear decay estimates \eqref{linear decay 1} and \eqref{linear decay 2} of Lemma \ref{linear decay lemma} to the formula \eqref{duhamel} to obtain
\begin{align}\label{haha0}
 &\norm{\Lambda^{\tilde{s}_1  } \tuc (t)}_{L^2}+ \norm{\Lambda^{\tilde{s}_1  -1} \tud  (t)}_{L^2}
\nonumber\\&\quad\ls (1+t)^{- \sigma_p- \frac{\tilde{s}_1  }{2}}\left(\norm{\tuc_0}_{L^p}+\norm{\tud_0}_{L^{p^\ast}}+\norm{\Lambda^{\tilde{s}_1  }\tub_0}_{L^2}+\norm{\Lambda^{\tilde{s}_1  -1}\tub_0}_{L^2} \right)
\\ &\qquad +\int_0^t(1+t-\tau)^{-  \sigma_p- \frac{\tilde{s}_1  }{2} }\left(\norm{\tnc(\tau)}_{L^p}+\norm{\tnd(\tau)}_{{L^{p^\ast}}} +\norm{\Lambda^{\tilde{s}_1  }\tn(\tau)}_{L^2}+\norm{\Lambda^{\tilde{s}_1  -1}\tn(\tau)}_{L^2} \right) d\tau .\nonumber
\end{align}
By the estimates \eqref{nonn1}--\eqref{121} from Lemma \ref{non le}, it holds that
\begin{equation}\label{hhal1}
\begin{split}
 & \norm{\tnc}_{L^p}+\norm{\tnd}_{{L^{p^\ast}}} +\norm{\Lambda^{\tilde{s}_1  }\tn}_{L^2}+\norm{\Lambda^{\tilde{s}_1  -1}\tn}_{L^2} 
 \\&\quad\ls   \delta \left( \norm{\Lambda^{ {\tilde{s}_1  } }\tub  }_{H^{\ell-\tilde{s}_1  }}+\norm{\Lambda^{ {\tilde{s}_1  -1} }\tud   }_{L^2}+\norm{\nabla\tub }_{L^\infty}+\norm{\tub }_{L^\infty} \right)+\norm{\Lambda^{\frac{d}{2}\left(1-\frac{1}{p}\right)} \tub}_{L^2}^2
 \\&\quad\ls  \delta \left( \norm{\Lambda^{ {\tilde{s}_1  } }\tub  }_{H^{\ell-\tilde{s}_1  }}+\norm{\Lambda^{ {\tilde{s}_1  -1} }\tud   }_{L^2}\right)+\norm{\Lambda^{\frac{d}{2}\left(1-\frac{1}{p}\right)} \tub}_{L^2}^2.
\end{split}%
\end{equation}
Here in the first inequality it has been used that $ {\tilde{s}_1}  \le d\left(1- {1}/{p} \right)+1 $ and that $d\left(1- {1}/{p} \right)+2<d/2+1< \ell$ for $1\le p<2d/(d+2)$, and in the second inequality it has been used that, since $\tilde{s}_1   <d/2$ and $\ell> d/2 + 1 $ and by the Sobolev interpolation,
\begin{equation}\label{gga11}
\norm{\na \tub  }_{L^\infty}+\norm{  \tub  }_{L^\infty}\ls \norm{\Lambda^{\tilde{s}_1  }  \tub }_{H^{\ell-{\tilde{s}_1  }}}.
\end{equation}
Then plugging the estimates \eqref{hhal1} into \eqref{haha0} yields
\begin{equation} \label{haha}
\begin{split}
 &\norm{\Lambda^{\tilde{s}_1  } \tuc (t)}_{L^2}+ \norm{\Lambda^{\tilde{s}_1  -1} \tud  (t)}_{L^2} 
\\&\quad\ls  (1+t)^{-  \sigma_p- \frac{\tilde{s}_1  }{2} } \ttk + \int_0^t(1+t-\tau)^{-  \sigma_p- \frac{\tilde{s}_1  }{2} } \delta\left(\norm{\Lambda^{ {\tilde{s}_1  } }\tub  }_{H^{\ell-\tilde{s}_1  }} +\norm{\Lambda^{ {\tilde{s}_1  -1} }\tud  }_{L^2}\right) \rmd\tau
 \\&\qquad+ \int_0^t(1+t-\tau)^{-  \sigma_p- \frac{\tilde{s}_1  }{2} } \norm{\Lambda^{\frac{d}{2}\left(1-\frac{1}{p}\right)} \tub(\tau)}_{L^2}^2  \rmd\tau.
\end{split}%
\end{equation}
Now one may estimate the two time integrals in the right hand side of \eqref{haha}. In view of the definition of $\m_1(t)$, the first time integral is bounded by
\begin{equation}\label{eses333}
\delta\int_0^t(1+t-\tau)^{-\sigma_p- \frac{\tilde{s}_1  }{2}}(1+\tau)^{-\sigma_p- \frac{\tilde{s}_1  }{2}} \sqrt{\m_1(\tau)}\,d\tau
\ls \delta (1+t)^{-\sigma_p- \frac{\tilde{s}_1  }{2}} \sqrt{\m_1(t)}.
\end{equation}
Here it has been used that $\sigma_p+\tilde{s}_1  /{2}>1$: if $\tilde{s}_1  =d(1-1/p)+1$, then $\sigma_p+\tilde{s}_1  /2=d/4+1/2>1$ since $d\ge 3$; if $\tilde{s}_1  =d/2 - \epsilon$, then $\sigma_p+\tilde{s}_1  /{2}=d/(2p) - \epsilon/2 > 1$ since $1\le p<d/2$ and one can take $\epsilon$ sufficiently small. On the other hand, since $ {d}/{2} (1- {1}/{p} )< \tilde{s}_1$, by the interpolation the second time integral is bounded by
\begin{equation}\label{eses222}
\begin{split}
 & \int_0^t(1+t-\tau)^{-\sigma_p- \frac{\tilde{s}_1  }{2}}(1+t)^{-2\sigma_p-\frac{d}{2} \left(1-\frac{1}{p}\right)}\m_1(\tau) \rmd\tau
 \\ &\quad= \int_0^t(1+t-\tau)^{-\sigma_p- \frac{\tilde{s}_1  }{2}}(1+t)^{-\sigma_p- \frac{d}{4}}\m_1(\tau)d\tau
\ls  (1+t)^{-\sigma_p- \frac{\tilde{s}_1  }{2}} {\m_1}(t)
\end{split}%
\end{equation}
since ${\tilde{s}_1}  < d/2$.
Thus substituting the estimates \eqref{eses333}--\eqref{eses222} into \eqref{haha} yields
\begin{equation}\label{eeee111}
 \norm{\Lambda^{\tilde{s}_1  } \tuc (t)}_{L^2}+ \norm{\Lambda^{\tilde{s}_1  -1} \tud  (t)}_{L^2} \ls  (1+t)^{-\sigma_p- \frac{\tilde{s}_1  }{2}}\left(\ttk + \delta\sqrt{\m_1(t)} + \m_1(t) \right).
\end{equation}
Then by \eqref{eeee111}, it follows from \eqref{case1} that
\begin{equation}\begin{split}
 \norm{\Lambda^{\tilde{s}_1  } \tub(t)}_{H^{\ell-\tilde{s}_1  }}^2&\ls e^{-\lambda t}\ttkk  +\int_0^te^{-\lambda(t-\tau)} (1+\tau)^{-2\sigma_p-\tilde{s}_1  }\left(\ttkk +\delta^2 {\m_1(\tau)} + \m_1^2(\tau) \right)  \rmd\tau\\
 &\ls (1+t)^{-2\sigma_p-\tilde{s}_1  }\left(\ttkk +\delta^2 {\m_1(t)} + \m_1^2(t) \right)
 .\end{split}
\end{equation}

Lastly and similarly, by using in addition the estimates \eqref{121}, one can get
\begin{equation}\label{eeee4441111}
\begin{split}
 \norm{\tuc(t)}_{L^2} &\ls (1+t)^{-\sigma_p}\left(\norm{\tuc_0}_{L^p}+\norm{\tud_0}_{L^{p^\ast}}+\norm{\tub_0}_{L^2}\right)
  \\&\quad+\int_0^t(1+t-\tau)^{-\sigma_p  }\left(\norm{\tnc(\tau)}_{L^p}+\norm{\tnd(\tau)}_{{L^{p^\ast}}}+\norm{\tn(\tau)}_{L^2}\right) \rmd\tau
  \\&\ls  (1+t)^{-\sigma_p}\ttk +\int_0^t(1+t-\tau)^{-\sigma_p}(1+\tau)^{-\sigma_p-\frac{\tilde{s}_1  }{2} }\left(\delta\sqrt{ \m_1(\tau)}  + \m_1 (\tau) \right) \rmd\tau
 \\&\ls  (1+t)^{-\sigma_p}\left(\ttk+\delta\sqrt{\m_1(t)} + \m_1(t) \right)
\end{split}%
\end{equation}
and
\begin{align}\label{eeee444}
\nonumber
 \norm{\tud(t)}_{L^2} & \ls  (1+t)^{-\sigma_p-\hal} \ttk +\int_0^t(1+t-\tau)^{-\sigma_p-\hal}(1+\tau)^{-\sigma_p-\frac{\tilde{s}_1  }{2}}\left( \delta\sqrt{ \m_1(\tau)}  + \m_1 (\tau) \right) \rmd\tau
 \\&\ls  (1+t)^{-\sigma_p-\hal}\left(\ttk +\delta\sqrt{ \m_1(t)}  + \m_1 (t) \right).
\end{align}
since $\tilde{s}_1  \ge1$.
Now combining the estimates \eqref{eeee111}--\eqref{eeee444} and recalling the definition of $\m_1(t)$ yields
\begin{equation}
\m_1(t)\ls  \ttkk   +\delta^2 \m_1(t)+ \m_1^2(t),
\end{equation}
which implies, since both $\delta$ and $\m_1(0)\le \ttk$ are small,
\begin{equation}\label{gadkal}
\m_1(t)\ls  \ttkk .
\end{equation}
By again the definition of $\m_1(t)$ and by the interpolation, one can deduce \eqref{u11}--\eqref{u12} from \eqref{gadkal} since $\epsilon>0$ can be taken arbitrarily small.

 One then turns to prove $(ii)$ under \asmp{\td 1} and that $d\ge 3$,  $1\le p\le 2$, $p<d$ and further that $p < {2d}/{(2(3-\ell)+d)}$ if $\ell\le 3$. Take $\tilde{s}_2  =\min\{ s_2^*, d/2+ 1 - \epsilon \}$ for a sufficiently small positive constant $\epsilon$.
Then define
\begin{equation}
\begin{split}
 \m_2(t)=\sup_{0\le \tau\le t}&\left\{(1+\tau)^{2\sigma_p+\tilde{s}_2  } \left( \norm{\Lambda^{\tilde{s}_2  } \tub (\tau)}_{H^{\ell-\tilde{s}_2  }}^2+ \norm{\Lambda^{\tilde{s}_2  -1} \tud  (\tau)}_{L^2}^2\right)\right.
\\& \left.\quad+(1+\tau)^{2\sigma_p}   \norm{  \tub (\tau)}_{L^2}^2+(1+\tau)^{2\sigma_p+1}  \norm{ \tud(\tau)  }_{L^2}^2 \right\}  .
 \end{split}
\end{equation}

First, taking $s=\tilde{s}_2  $ in \eqref{000001} yields
\begin{equation}\label{case11212}
\norm{\Lambda^{{\tilde{s}_2} } \tub(t)}_{H^{\ell-{\tilde{s}_2} }}^2\ls e^{-\lambda t}\ttkk  +\int_0^te^{-\lambda(t-\tau)} \left(\norm{\Lambda^{{\tilde{s}_2} }\tuc(\tau)}_{L^2}^2+ \norm{  \Lambda^{{\tilde{s}_2} -1}\tud (\tau)}_{L^2}^2 +\norm{ \tub(\tau)  }_{L^\infty}^4\right)  \rmd\tau.
\end{equation}
On the other hand, if $1\le p\le 2d/(d+2)$ and hence $p^\ast\equiv dp/(d-p)\le 2$, one may apply the linear decay estimates \eqref{linear decay 1} and \eqref{linear decay 2} of Lemma \ref{linear decay lemma} to the formula \eqref{duhamel} to deduce
\begin{align} \label{lll111}
 &\norm{\Lambda^{\tilde{s}_2  } \tuc (t)}_{L^2}+ \norm{\Lambda^{\tilde{s}_2  -1} \tud  (t)}_{L^2}
\nonumber\\&\quad\ls (1+t)^{- \sigma_p- \frac{\tilde{s}_2  }{2}}\left(\norm{\tuc_0}_{L^p}+\norm{\tud_0}_{L^{\frac{dp}{d-p}}}+\norm{\Lambda^{\tilde{s}_2  }\tub_0}_{L^2}+\norm{\Lambda^{\tilde{s}_2  -1}\tub_0}_{L^2} \right)
\\&\qquad
+\int_0^t(1+t-\tau)^{-  \sigma_p- \frac{\tilde{s}_2  }{2} }\left(\norm{\tnc(\tau)}_{L^p}+\norm{\tnd(\tau)}_{{L^{\frac{dp}{d-p}}}}+\norm{\Lambda^{\tilde{s}_2  }\tn(\tau)}_{L^2}+\norm{\Lambda^{\tilde{s}_2  -1}\tn(\tau)}_{L^2}\right) \rmd\tau;\nonumber
\end{align}
if $2d/(d+2)< p\le 2$, one may apply instead the linear decay estimates \eqref{linear decay 1'} and \eqref{linear decay 2'} to deduce
\begin{align}\label{lll222}
 &\norm{\Lambda^{\tilde{s}_2  } \tuc (t)}_{L^2}+ \norm{\Lambda^{\tilde{s}_2  -1} \tud  (t)}_{L^2}
\nonumber \\&\quad\ls (1+t)^{- \sigma_p- \frac{\tilde{s}_2  }{2}}\left(\norm{\tuc_0}_{L^p}+\norm{\Lambda\tud_0}_{L^p}+\norm{\Lambda^{\tilde{s}_2  }\tub_0}_{L^{2}}+\norm{\Lambda^{\tilde{s}_2  -1}\tub_0}_{L^2} \right)
\\ &\qquad   +\int_0^t(1+t-\tau)^{-  \sigma_p- \frac{\tilde{s}_2  }{2} }\left(\norm{\tnc(\tau)}_{L^p}+\norm{\Lambda\tnd(\tau)}_{{L^{p}}} +\norm{\Lambda^{\tilde{s}_2  }\tn(\tau)}_{L^2}+\norm{\Lambda^{\tilde{s}_2  -1}\tn(\tau)}_{L^2} \right) \rmd\tau . \nonumber
\end{align}
By the estimates \eqref{nonn61} from Lemma \ref{non le}, it holds that
for $1 \leq p < 2d/(d+2)$,
\begin{equation} \label{7.74}
\norm{ \tnd}_{L^{\frac{dp}{d-p}}} \leq \delta \norm{\La^{\tilde{s}_2}  \tub}_{H^{\ell-{\tilde{s}_2} }} + \norm{\La^{\frac{{\tilde{s}_2} }{2}} \tub}_{H^{\ell-{\tilde{s}_2} /2}}^2,
\end{equation}
where it has been used that $ {\tilde{s}_2}  \le d\left(1- {1}/{p} \right)+1$ and that $d(1-1/p) + 2 < d/2 + 1 < \ell$ for $1 \leq p < 2d/(d+2)$,
and by the estimates \eqref{nonn6123}--\eqref{nonn61235} that for $2d/(d+2) < p < 2$,
\begin{equation}\label{7.7444}
\norm{\La \tnd}_{L^p} \leq \delta \left( \norm{\La^{\tilde{s}_2}  \tub}_{H^{\ell-{\tilde{s}_2} }} + \norm{\La^{{\tilde{s}_2} -1} \tud}_{L^2} \right)+ \norm{\La^{{\tilde{s}_2} -1} \tub}_{L^2} \norm{\tub}_{L^\infty}.
\end{equation}
These two estimates together with the estimates \eqref{nonn11}--\eqref{nn5} and \eqref{nonn61235} imply that the two sums inside the two time integrals in the right hand side of \eqref{lll111} and \eqref{lll222} are bounded by
 \begin{equation}\label{haha012121112}
\begin{split}
 &  \delta \left( \norm{\Lambda^{ {\tilde{s}_2  } }\tub  }_{H^{\ell-\tilde{s}_2  }}+\norm{\Lambda^{ {\tilde{s}_2  -1} }\tud  }_{L^2}
 +\norm{\nabla\tub }_{L^\infty}+\norm{\tud }_{L^\infty} \right) \\&\quad + \norm{\La^{\frac{{\tilde{s}_2} }{2}} \tub}_{H^{\ell-{\tilde{s}_2} /2}}^2 + \norm{\tub }_{L^\infty}\left(
 \norm{\Lambda^{\tilde{s}_2  -1} \tub }_{L^2}+  \norm{\tub }_{L^\infty} \right)
 \\&\quad\ls \delta \left(   \norm{\Lambda^{ {\tilde{s}_2  } }\tub  }_{H^{\ell-\tilde{s}_2  }}+\norm{\Lambda^{ {\tilde{s}_2  -1} }\tud  }_{L^2}
 \right)   + \norm{\La^{\frac{{\tilde{s}_2} }{2}} \tub}_{H^{\ell-{\tilde{s}_2} /2}}^2 + \norm{\tub }_{L^\infty}
 \norm{\Lambda^{\tilde{s}_2  -1} \tub }_{H^{\ell-\tilde{s}_2+1}} .
\end{split}%
\end{equation}
Here in the first bound it has been used that $ {\tilde{s}_2}  \le d\left(1- {1}/{p} \right)+1,\ \tilde{s}_2< d/2+1 $,
and in the second bound it has been used that, since ${\tilde{s}_2}  < d/2+1 < \ell$ and by Sobolev interpolation,
 \begin{equation}
\norm{\nabla \tub}_{L^\infty} + \norm{\tud}_{L^\infty} \ls \norm{\La^{{\tilde{s}_2} } \tub}_{H^{\ell-{\tilde{s}_2} }} + \norm{\La^{{\tilde{s}_2} -1} \tud}_{L^2}.
\end{equation}
Plugging the estimates \eqref{haha012121112} into \eqref{lll111} or \eqref{lll222},  in each case one can get
\begin{equation}\label{haha01212}
\begin{split}
 &\norm{\Lambda^{\tilde{s}_2  } \tuc (t)}_{L^2}+ \norm{\Lambda^{\tilde{s}_2  -1} \tud  (t)}_{L^2}
  \\&\quad\ls  (1+t)^{-  \sigma_p- \frac{\tilde{s}_2  }{2} } \ttk +
  \int_0^t(1+t-\tau)^{-  \sigma_p- \frac{\tilde{s}_2  }{2} }  \delta \left(\norm{\Lambda^{ {\tilde{s}_2  } }\tub  }_{H^{\ell-\tilde{s}_2  }}+\norm{\Lambda^{ {\tilde{s}_2  -1} }\tud  }_{L^2}  \right)d\tau \\
 & \qquad + \int_0^t (1+t-\tau)^{-  \sigma_p- \frac{\tilde{s}_2  }{2} } \norm{\La^{\frac{{\tilde{s}_2} }{2}} \tub }_{H^{\ell - {\tilde{s}_2} /2}}^2d\tau
 \\&\qquad+ \int_0^t(1+t-\tau)^{-  \sigma_p- \frac{\tilde{s}_2  }{2} }\norm{\tub }_{L^\infty}
 \norm{\Lambda^{\tilde{s}_2  -1} \tub }_{H^{\ell-\tilde{s}_2+1}} d\tau.
\end{split}%
\end{equation}
Now one may estimate the three time integrals in the right hand side of \eqref{haha01212}. In view of the definition of $\m_2(t)$, the first time integral is bounded by
\begin{equation}\label{eses3331212}
\delta\int_0^t(1+t-\tau)^{-\sigma_p- \frac{\tilde{s}_2  }{2}}(1+\tau)^{-\sigma_p- \frac{\tilde{s}_2  }{2}} \sqrt{\m_2(\tau)}\,d\tau
\ls \delta (1+t)^{-\sigma_p- \frac{\tilde{s}_2  }{2}} \sqrt{\m_2(t)}.
\end{equation}
Here it has been used that $\sigma_p+\tilde{s}_2  /2>1$: if $\tilde{s}_2  =d(1-1/p)+1$, then $\sigma_p+\tilde{s}_2  /2=d/4+1/2>1$ since $d\ge 3$; if $\tilde{s}_2  =d/2+1 - \epsilon$, then $\sigma_p+\tilde{s}_2  /{2}=d/(2p)+1/2 - \epsilon/2 > 1$ since $1\le p<d$ and $\epsilon$ can be taken sufficiently small; if $\tilde{s}_2  =\ell-1$, then $\sigma_p+\tilde{s}_2  /{2}= {d}/{2}( {1}/{p} - 1/2) +  {(\ell-1)}/{2}>1$, which trivially holds when $\ell>3$, and follows from $1 \leq p < {2d}/{(2(3-\ell)+d)}$ when $\ell\le 3$. On the other hand, by the interpolation the second time integral is bounded by
\begin{equation} \label{8.83+}
\int_0^t (1+t-\tau)^{-\sigma_p-\frac{{\tilde{s}_2} }{2}} (1+\tau)^{-2 \sigma_p - \frac{{\tilde{s}_2} }{2}} \m_2(\tau) \rmd \tau \ls (1+t)^{-\sigma_p-\frac{{\tilde{s}_2} }{2}} \m_2(t).
\end{equation}
To bound the third time integral one may use the interpolation to estimate
 \begin{equation}\label{8.86}
 \norm{\tub }_{L^\infty}\ls \norm{\Lambda^{\min\{\frac{d}{2}- {\epsilon},\tilde{s}_2  \}}\tub }_{H^{\ell-\min\{\frac{d}{2}- {\epsilon},\tilde{s}_2  \}}}\ls (1+t)^{- \sigma_p-\frac{\min\{\frac{d}{2}- {\epsilon},\tilde{s}_2  \}}{2}} \sqrt{\m_2(t)}.
\end{equation}
Hence, the third time integral is bounded by
 \begin{equation}\label{eses2221212}
 \begin{split}
&    \int_0^t(1+t-\tau)^{-\sigma_p- \frac{\tilde{s}_2  }{2}}
  (1+\tau)^{- 2\sigma_p-\frac{\min\{\frac{d}{2} -  {\epsilon},\tilde{s}_2  \}+\tilde{s}_2  -1}{2} }
  \m_2(\tau)d\tau
 \\ &\quad\ls  (1+t)^{-\sigma_p- \frac{\tilde{s}_2  }{2}} \m_2(t).
\end{split}
\end{equation}
Here it has been used that  ${\tilde{s}_2} \ge  1$ and
$2 \sigma_p + ( {d}/{2} +  \tilde{s}_2  - 1) /2=  \sigma_p +  {{\tilde{s}_2} }/{2}+(d/p-1)/2>\sigma_p +  {{\tilde{s}_2} }/{2}$ since $1\le p<d$ and one can take $\epsilon$ sufficiently small.
Thus substituting the estimates \eqref{eses3331212}, \eqref{8.83+} and \eqref{eses2221212} into \eqref{haha01212} yields
\begin{equation}\label{eeee1111212}
 \norm{\Lambda^{\tilde{s}_2  } \tuc (t)}_{L^2}+ \norm{\Lambda^{\tilde{s}_2  -1} \tud  (t)}_{L^2} \ls  (1+t)^{-\sigma_p- \frac{\tilde{s}_2  }{2}}\left(\ttk + \delta\sqrt{\m_2(t)} + \m_2(t) \right).
\end{equation}
Then by \eqref{eeee1111212} together with an use of the estimates \eqref{eses2221212}, the bound of the third time integral in the right hand side of \eqref{haha01212}, one may deduce from \eqref{case11212} that
\begin{equation}\label{eeee1111212+}
\begin{split}
 \norm{\Lambda^{\tilde{s}_2  } \tub(t)}_{H^{\ell-\tilde{s}_2  }}^2&\ls e^{-\lambda t} \ttkk  +\int_0^te^{-\lambda(t-\tau)} (1+\tau)^{-2\sigma_p-\tilde{s}_2  }\left(\ttkk +\delta^2 {\m_2(t)} + \m_2^2(t) \right)  \rmd\tau\\
 &\ls  (1+t)^{-2\sigma_p-\tilde{s}_2  }\left(\ttkk +\delta^2 {\m_2(t)} + \m_2^2(t) \right)
 .\end{split}
\end{equation}

Finally, note that one needs to revise the estimates \eqref{eeee4441111}--\eqref{eeee444} in the case $(i)$ since now $\sigma_p-\tilde{s}_1>1$ may not be true for $d/2\le p<d$.
Then it is crucial to exploit the exponential decay property of the high frequency part in the linear decay estimates of Lemma \ref{linear decay lemma}. Indeed, applying the linear decay estimates \eqref{linear decay 1} and \eqref{linear decay 2'} and using in addition the estimates \eqref{nonn612}, one can deduce
\begin{equation} \label{ttt3}
\begin{split}
 \norm{\tuc(t)}_{L^2} &\ls (1+t)^{-\sigma_p}\left(\norm{\tuc_0}_{L^p}+\norm{\tud_0}_{L^{p^\ast}}+\norm{\tub_0}_{L^2}\right)
  \\&\quad+\int_0^t(1+t-\tau)^{-\sigma_p  }\left(\norm{\tnc(\tau)}_{L^p}+\norm{\tnd(\tau)}_{{L^{p^\ast}}}\right)+e^{-\lambda(t-\tau)}\norm{\tn(\tau)}_{L^2} \rmd\tau
  \\&\ls  (1+t)^{-\sigma_p}\ttk+\int_0^t(1+t-\tau)^{-\sigma_p}(1+\tau)^{-\sigma_p-\frac{\tilde{s}_2  }{2} }\left(\delta\sqrt{ \m_2(\tau)}  + \m_2 (\tau) \right) \rmd\tau
 \\&\quad  +\int_0^t  e^{-\lambda(t-\tau)}(1+\tau)^{-2\sigma_p-\frac{\min\{\frac{d}{2}- {\epsilon},\tilde{s}_2  \}}{2} } { \m_2(\tau)} \rmd\tau
 \\&\ls  (1+t)^{-\sigma_p}\left(\ttk+\delta\sqrt{\m_2(t)} + \m_2(t) \right), \hs \textrm{for } 1 \leq p\le  \frac{2d}{d+2},
\end{split}%
\end{equation}
\begin{equation}
\begin{split}
\norm{\tuc(t)}_{L^2} &\ls (1+t)^{-\sigma_p}\left(\norm{\tuc_0}_{L^p}+\norm{\La \tud_0}_{L^{p}}+\norm{\tub_0}_{L^2}\right)
  \\&\quad+\int_0^t(1+t-\tau)^{-\sigma_p  }\left(\norm{\tnc(\tau)}_{L^p}+\norm{\La \tnd(\tau)}_{{L^{p}}}\right)+e^{-\lambda(t-\tau)}\norm{\tn(\tau)}_{L^2} \rmd\tau
 \\&\ls  (1+t)^{-\sigma_p}\left(\ttk+\delta\sqrt{\m_2(t)} + \m_2(t) \right), \hs \textrm{for } \frac{2d}{d+2} < p \leq 2.
\end{split}
\end{equation}
and
\begin{align} \label{ttt5}
 \nonumber\norm{\tud(t)}_{L^2} & \ls  (1+t)^{-\sigma_p-\hal}\ttk+\int_0^t(1+t-\tau)^{-\sigma_p-\hal}(1+\tau)^{-\sigma_p-\frac{\tilde{s}_2  }{2} }\left(\delta\sqrt{ \m_2(\tau)}  + \m_2 (\tau) \right) \rmd\tau
 \\&\quad  +\int_0^t  e^{-\lambda(t-\tau)}(1+\tau)^{-2\sigma_p-\frac{\min\{\frac{d}{2}- {\epsilon},\tilde{s}_2  \}}{2} } { \m_2(\tau)} \rmd\tau
 \\&\ls  (1+t)^{-\sigma_p-\hal}\left(\ttk +\delta\sqrt{ \m_2(t)} + \m_2(t) \right).\nonumber
\end{align}
Here it has been used that $\tilde{s}_2  \ge1$ and that
$\sigma_p+ {d}/{4}= {d}/({2p})>1/2$ since $1\le p<d.$
Based on these estimates \eqref{eeee1111212}--\eqref{ttt5}, one can get that for $\delta$ and $\m_2(0)\le \ttk$ small,
\begin{equation}\label{ggllad}
\m_2(t)\ls  \ttkk .
\end{equation}
This implies \eqref{u1112}--\eqref{u1212}.

Finally, one turns to prove $(iii)$ under \asmp{\td 1}, \asmp{\td 3} and that $d\ge 2$ and $\norm{\tvv_1}_{L^{q}} \ls \delta$ with $1\le p\le q\le 2$, $ q<d$ and further that $p < {2d}/{(2(3-\ell)+d)}$ if $\ell\le 3$.
Take ${\tilde{s}_3}  = \min \{ s_3^*,d/2+1-\epsilon \}$ for a sufficiently small positive constant $\epsilon$. Then define
\begin{equation} \label{7.100}
\begin{split}
\m_3 (t)= \sup_{0\leq \tau \leq t} &\left\{ (1+\tau)^{2\sigma_p+{\tilde{s}_3} } \left( \norm{\La^{{\tilde{s}_3} } \tub(\tau)}_{H^{\ell-{\tilde{s}_3} }}^2 + \norm{\La^{{\tilde{s}_3} -1} \tud(\tau)}_{L^2}^2 \right)\right. \\
&\quad   \left.+ (1+\tau)^{2\sigma_p} \norm{\tub(\tau)}_{L^2}^2 + (1+\tau)^{2\sigma_p+1} \norm{\tud(\tau)}_{L^2}^2 \right\}
\end{split}
\end{equation}
One can still use the same estimates \eqref{case11212}--\eqref{lll222} in the case $(ii)$, however, the nonlinear estimates can be improved by using the estimates \eqref{7.30+}--\eqref{7.32+} from Lemma \ref{non le}. Indeed, by the estimates \eqref{7.30+}--\eqref{7.300+}, it holds that for $1\le p<2$,
\begin{equation}
\begin{split}
 \norm{\tnc }_{L^p}  \ls &\delta \left( \norm{\La^{{\tilde{s}_3} } \tub }_{H^{\ell-{\tilde{s}_3} }} + \norm{\La^{{\tilde{s}_3} -1} \tud }_{L^2}  \right) \\
&  + \norm{\La^{d(1-\frac{1}{p})} \tub }_{L^2} \left(\norm{\tub }_{L^\infty} + \norm{\La \tub }_{L^2} + \norm{\tud }_{L^2}  \right),
\end{split}
\end{equation}
and by the estimates \eqref{7.322+}--\eqref{7.32+}, it holds that for $ 2d/(d+2)< p<2$,
\begin{equation} \label{7.102}
\begin{split}
\norm{\La \tnd }_{L^p} & \ls \delta \left( \norm{\La^{{\tilde{s}_3} } \tub }_{H^{\ell-{\tilde{s}_3} }} + \norm{\La^{{\tilde{s}_3} -1} \tud }_{L^2} \right)\\
& \quad + \norm{\La^{d(1-\frac{1}{p})}\tub}_{L^2}    \norm{\tub}_{L^\infty}+\norm{\La^{d(1-\frac{1}{p})} \tub}_{H^1}\norm{\Lambda \tub}_{L^2} .
\end{split}
\end{equation}
Hence, noting further \eqref{nn5} and \eqref{nonn61235}, one may improve the estimates \eqref{haha012121112} to be
\begin{equation} \label{7.101}
\begin{split}
&  \delta \left( \norm{\La^{{\tilde{s}_3} } \tub }_{H^{\ell-{\tilde{s}_3} }} + \norm{\La^{{\tilde{s}_3} -1} \tud }_{L^2}   \right)  \\
& \quad + \norm{\tub }_{L^\infty}
 \norm{\Lambda^{\tilde{s}_3  -1} \tub }_{H^{\ell-\tilde{s}_3+1}}+ \norm{\tub }_{L^\infty}\norm{\La^{d(1-\frac{1}{p})} \tub }_{L^2} \\
& \quad+ \norm{\La^{\frac{d}{2}(1-\frac{1}{p})+\hal} \tub(\tau)}_{L^2}^2+ \norm{\La^{d(1-\frac{1}{p})} \tub }_{H^1} \left( \norm{\La \tub }_{L^2} + \norm{\tud }_{L^2}  \right).
\end{split}
\end{equation}
Plugging the estimates \eqref{7.101} to \eqref{lll111} or \eqref{lll222}, in each case it holds that
\begin{equation} \label{7.104}
\begin{split}
& \norm{\La^{{\tilde{s}_3} } \tuc(t) }_{L^2} + \norm{\La^{{\tilde{s}_3} -1} \tud(\tau)}_{L^2} \\
& \ls (1+ t)^{-\sigma_p - \frac{{\tilde{s}_3} }{2}} \ttk + \int_0^t (1+t-\tau)^{-\sigma_p - \frac{{\tilde{s}_3} }{2}} \delta \left( \norm{\La^{{\tilde{s}_3} } \tub(\tau)}_{H^{\ell-{\tilde{s}_3} }} + \norm{\La^{{\tilde{s}_3} -1} \tud(\tau)}_{L^2} \right)   \rmd\tau \\
& \quad + \int_0^t (1+t-\tau)^{-\sigma_p-\frac{{\tilde{s}_3} }{2}} \bigg( \norm{\tub(\tau) }_{L^\infty}\left(
 \norm{\Lambda^{\tilde{s}_3  -1} \tub(\tau) }_{H^{\ell-\tilde{s}_3+1}}+\norm{\La^{d(1-\frac{1}{p})} \tub(\tau) }_{L^2}   \right)   \\
& \qquad\quad+ \norm{\La^{\frac{d}{2}(1-\frac{1}{p})+\hal} \tub(\tau)}_{L^2}^2+ \norm{\La^{d(1-\frac{1}{p})} \tub(\tau) }_{H^1} \left( \norm{\La \tub(\tau) }_{L^2} + \norm{\tud(\tau) }_{L^2}  \right) \bigg) \rmd\tau.
\end{split}
\end{equation}
Now one may estimate the two time integrals in the right hand side of \eqref{7.104}. In view of the definition of $\m_3 (t)$, the first time integral is bounded by
\begin{equation}\label{eses3331212'}
\delta\int_0^t(1+t-\tau)^{-\sigma_p- \frac{\tilde{s}_3  }{2}}(1+\tau)^{-\sigma_p- \frac{\tilde{s}_3  }{2}} \sqrt{\m_3 (\tau)}\,d\tau
\ls \delta (1+t)^{-\sigma_p- \frac{\tilde{s}_3  }{2}} \sqrt{\m_3 (t)}.
\end{equation}
Here it has been used that $\sigma_p+\tilde{s}_3  /2>1$:  if $\tilde{s}_3  =d(1/2+1/q-1/p)+1$, then $\sigma_p+\tilde{s}_3  /2=d/(2q)+1/2>1$ since $1\le q<d$; if $\tilde{s}_3  =d(1 -1/p)+2$, then $\sigma_p+\tilde{s}_3  /2=d/4+1>1$; if $\tilde{s}_3  =d/2+1 - \epsilon$, then $\sigma_p+\tilde{s}_3  /{2}=d/(2p)+1/2 - \epsilon/2 > 1$ since $1\le p<d$ and one can take $\epsilon$ sufficiently small;  if $\tilde{s}_3  =\ell-1$, then $\sigma_p+\tilde{s}_3  /{2}= {d}/{2}( {1}/{p} - 1/2) +  {(\ell-1)}/{2}>1$ which trivially holds when $\ell>3$, and follows from $1 \leq p < {2d}/{(2(3-\ell)+d)}$ when $\ell\le 3$. To estimate the second time integral, it is noted that,  similar to \eqref{eses2221212} the first term is bounded by
\begin{equation}
(1+t)^{-\sigma_p- \frac{\tilde{s}_3  }{2}} \m_3 (t).
\end{equation}
Since $d(1-1/p)\le \tilde{s}_3$ one may use the interpolation, together with the estimates similar to \eqref{8.86}, to bound the second term by
 \begin{equation}
 \begin{split}
&  \int_0^t(1+t-\tau)^{-\sigma_p- \frac{\tilde{s}_3  }{2}}
  (1+\tau)^{- 2\sigma_p-\frac{d}{2} \left(1-\frac{1}{p}\right)-\frac{\min\{\frac{d}{2} -  {\epsilon},\tilde{s}_3  \} }{2} }
  \m_3 (\tau) \rmd\tau
 \\ &\quad\ls  (1+t)^{-\sigma_p- \frac{\tilde{s}_3  }{2}} \m_3 (t).
\end{split}
\end{equation}
Here it has been used that
$ 2 \sigma_p +    {d}/{2} (1- {1}/{p} ) +  {d}/{4}  =\sigma_p+  {d}/{2}  \geq \sigma_p + (d/2+1)/2 > \sigma_p +  {{\tilde{s}_3} }/{2}$ since $d\ge 2$ and $\tilde{s}_3 < d/2+1$.
Similarly, the last three terms are bounded by
\begin{equation}\label{tttt1}
\begin{split}
 & \int_0^t(1+t-\tau)^{-\sigma_p- \frac{\tilde{s}_3  }{2}}(1+t)^{-2\sigma_p-\frac{d}{2} \left(1-\frac{1}{p}\right)-\hal}\m_3(\tau) \rmd\tau
 \\ &\quad= \int_0^t(1+t-\tau)^{-\sigma_p- \frac{\tilde{s}_3  }{2}}(1+t)^{-\sigma_p- \frac{d}{4}-\hal}\m_3(\tau) \rmd\tau
\ls  (1+t)^{-\sigma_p- \frac{\tilde{s}_3  }{2}} {\m_3}(t),
\end{split}%
\end{equation}
since $\tilde{s}_3<d/2+1$. Thus substituting the estimates \eqref{eses3331212'}--\eqref{tttt1} into \eqref{7.104} yields
\begin{equation} \label{ttt2}
 \norm{\Lambda^{\tilde{s}_3  } \tuc (t)}_{L^2}+ \norm{\Lambda^{\tilde{s}_3  -1} \tud  (t)}_{L^2} \ls  (1+t)^{-\sigma_p- \frac{\tilde{s}_3  }{2}}\left(\ttk + \delta\sqrt{\m_3 (t)} + \m_3 (t) \right).
\end{equation}
Then by \eqref{ttt2}, similarly as \eqref{eeee1111212+}, it follows from \eqref{case11212} that
\begin{equation}
\norm{\Lambda^{\tilde{s}_3  } \tub(t)}_{H^{\ell-\tilde{s}_3  }}^2  \ls  (1+t)^{-2\sigma_p-\tilde{s}_3  }\left(K_0^2+\delta^2 {\m_3 (t)} + \m_3 ^2(t) \right).
\end{equation}

Finally, similarly as \eqref{ttt3}--\eqref{ttt5} in the case $(ii)$,
one can get
\begin{equation}
\norm{\tuc(t)}_{L^2} \ls (1+t)^{-\sigma_p} \left( \ttk + \delta \sqrt{\m_3(t)} + \m_3(t)\right)
\end{equation}
and
\begin{equation} \label{ttt6}
\norm{\tud(t)}_{L^2} \ls (1+t)^{-\sigma_p-\hal} \left( \ttk + \delta \sqrt{\m_3(t)} + \m_3(t)\right).
\end{equation}
Thus combining \eqref{ttt2}--\eqref{ttt6}, one can get that for $\delta$ and $\m_3 (0)\le \ttk$ small,
\begin{equation}\label{ggllad22}
\m_3 (t)\ls  K_0^2.
\end{equation}
This implies \eqref{7.50+}--\eqref{7.51+}.
\end{proof}

\begin{remark}
In the statement of the assertions $(ii)$ and $(iii)$ of Proposition \ref{decay decay}, it has been used implicitly that for $2d/(d+2)< p< 2$,
\begin{equation}
\norm{\Lambda \tud_0}_{L^p}\ls \norm{\tud_0}_{L^2}+\norm{\Lambda \tud_0}_{L^2} \le \ttk,
\end{equation}
which follows from an use of Lemma \ref{A3}.
\end{remark}

\section{Proof of the theorems}\label{999}

This section is devoted to prove the main theorems.
By the standard local existence theory \cite{kato1975local,majda1984local} and a continuity argument one needs only to show how to close the {\it a priori} estimates. 
\begin{proof}[Proof of Theorem \ref{thm:1.1}] Under the assumptions of Theorem \ref{thm:1.1}, one assumes a priori that $\norm{\tu(t)}_{H^\ell} \leq \delta$.
Then by the $(i)$ assertions of Propositions \ref{energy energy} and \ref{decay decay}, taking $\tilde{s}_1$ be the one in the proof of the assertion $(i)$ of Proposition \ref{decay decay}, it holds that
\begin{align}\label{end1}
  \norm{\tu(t)}_{H^\ell}
 & \ls \norm{\tu_0}_{H^\ell} +\int_0^t  \norm{ \Lambda^{\tilde{s}_1}   \tub(\tau)  }_{H^{\ell-\tilde{s}_1}}  \rmd\tau
 \\&\ls \norm{\tu_0}_{H^\ell} + K_0 \int_0^t (1+\tau)^{-\sigma_p-\frac{\tilde{s}_1}{2}} \rmd \tau\ls K_0.\nonumber
 \end{align}
This closes the a priori estimates by taking $K_0$ sufficiently small. Consequently, by the estimates \eqref{end1} and the $(i)$ assertions of Propositions \ref{enes6} and \ref{decay decay}, one can obtain \eqref{th11}--\eqref{th13} and thus conclude the proof of Theorem \ref{thm:1.1}.
\end{proof}

\begin{proof}[Proof of Theorem \ref{thm:1.2}] Under the assumptions of Theorem \ref{thm:1.2}, one assumes a priori that $\norm{\tu(t)}_{H^\ell} \leq \delta$. Then by the $(ii)$ assertions of Propositions \ref{energy energy} and \ref{decay decay}, taking $\tilde{s}_2$ be the one in the proof of the assertion $(ii)$ of Proposition \ref{decay decay}, it holds that
\begin{equation}\label{end2}
\begin{split}
  \norm{\tu(t)}_{H^\ell}
 & \ls \norm{\tu_0}_{H^\ell} +\int_0^t \left( \norm{ \Lambda^{\tilde{s}_2}   \tub(\tau)  }_{H^{\ell-\tilde{s}_2}}+\norm{ \Lambda^{\tilde{s}_2-1}   \tud(\tau)  }_{L^2} +\norm{ \tub(\tau) }_{L^\infty}^2\right)  \rmd\tau
 \\&\ls \norm{\tu_0}_{H^\ell} + K_0 \int_0^t \left((1+\tau)^{-\sigma_p-\frac{\tilde{s}_2}{2}}+(1+\tau)^{-2\sigma_p-\frac{d}{2}+\epsilon}\right) \rmd \tau\ls K_0
 \end{split}
 \end{equation}
since $d\ge 3$. This closes the a priori estimates by taking $K_0$ sufficiently small. Consequently, by the estimates \eqref{end2}, the assertion $(i)$ of Proposition \ref{enes6} and the assertion $(ii)$ of Proposition \ref{decay decay}, one can obtain \eqref{th21}--\eqref{th23} and thus conclude the proof of Theorem \ref{thm:1.2}.
\end{proof}

\begin{proof}[Proof of Theorem \ref{thm:1.3}] Under the assumptions of Theorem \ref{thm:1.3}, one assumes a priori that $\norm{\tu(t)}_{H^\ell}+\norm{\tvv_1(t)}_{L^{q}} \leq \delta$. Then by the assertion $(ii)$ of Proposition \ref{energy energy} and the assertion $(iii)$ of Proposition \ref{decay decay}, taking $\tilde{s}_3$ be the one in the proof of the assertion $(iii)$ of Proposition \ref{decay decay}, it holds that
\begin{equation}\label{end3}
\begin{split}
  \norm{\tu(t)}_{H^\ell}
 & \ls \norm{\tu_0}_{H^\ell} +\int_0^t \left( \norm{ \Lambda^{\tilde{s}_3}   \tub(\tau)  }_{H^{\ell-\tilde{s}_3}}+\norm{ \Lambda^{\tilde{s}_3-1}   \tud(\tau)  }_{L^2}+\norm{ \tub(\tau) }_{L^\infty}^2\right)  \rmd\tau
 \\&\ls \norm{\tu_0}_{H^\ell} + \tk \int_0^t \left((1+\tau)^{-\sigma_p-\frac{\tilde{s}_3}{2}}+(1+\tau)^{-2\sigma_p-\frac{d}{2}+\epsilon}\right) d \tau\ls \tk
 \end{split}
 \end{equation}
since $2\sigma_p+ {d}/{2}=d/p>1$ for $1\le p<d$ and one can take $\epsilon$ sufficiently small.
Moreover, by the estimates \eqref{lpp} of Proposition \ref{prop:6+.1} and Remark \ref{v10data}, it holds that
\begin{equation}\label{end4}
\begin{split}
\norm{\tvv_1(t)}_{L^{q}} & \ls \exp \left(C \ttk \right) \times \left( \tk + \int_0^t \norm{\tub }_{L^{2q}} \left( \norm{\tud }_{L^{2q}} + \norm{\nabla \tub }_{L^{2q}} + \norm{\tub }_{L^{4q}}^2 \right) \rmd\tau\right)
\\
& \ls \tk + \ttk  \int_0^t\left((1+\tau)^{-\frac{d}{2}\left (\frac{2}{p} - \frac{1}{q}\right) -\hal } + (1+\tau)^{-\frac{d}{2}\left(\frac{3}{p} - \frac{1}{q}\right)}  \right) \rmd\tau\ls \tk
\end{split}
\end{equation}
since $ {d}/{2}  ({2}/{p} -  {1}/{q}) +1/2\ge {d}/{(2p)}   +1/2>1 $ and $ {d}/{2}  ({3}/{p} -  {1}/{q})  \ge {d}/{p} >1 $ for $1\le p\le q<d$.
These close the a priori estimates by taking $\tk$ sufficiently small. Consequently, by the estimates \eqref{end3}, the assertion $(ii)$ of Proposition \ref{enes6} and the assertion $(iii)$ of Proposition \ref{decay decay}, one can obtain \eqref{th31}--\eqref{th33} and thus conclude the proof of Theorem \ref{thm:1.3}.
\end{proof}

\section{Application to Damping Full Euler System} \label{App_C}

In this section, we verify that the damping full Euler system of adiabatic flow satisfies all the assumptions \asmp{A1}--\asmp{A4}, \asmp{B} and \asmp{D1}--\asmp{D3}. The system can be written in terms of the density $\rho$, the velocity $v$ and the entropy $S$ of the flow as
\begin{equation}\label{euler1}
\begin{cases}
\dis\p_t \rho + \Div (\rho v) = 0, \\
\dis\p_t (\rho v) + \Div(\rho v \otimes v + p I_d) = - \rho v, \\
\dis\p_t (\rho S ) + \Div(\rho v S) = 0.
\end{cases}
\end{equation}
It is well known that the thermodynamics variables $\rho$, $p$, $S$, the temperature $\theta$ and the internal energy $e$ satisfy the relation
\begin{equation}
\rmd e = \theta \rmd S - p \rmd(\frac{1}{\rho}),
\end{equation}
and they can be determined by knowing any two of them; we may view $p, \theta, $ and $e$ as functions of $\rho$ and $S$. It is assumed that $e$ is strictly convex with respect to $\rho$ and $S$.

The system \eqref{euler1} can be rewritten into
\begin{equation}
\begin{cases}
\dis\p_t S + \sum_{k=1}^{d} v_k  S_{x_k} = 0, \\
\dis \p_t \rho + \sum_{k=1}^{d} v_k   \rho_{x_k} + \sum_{k=1}^{d} \rho  ( v_k)_{x_k}
 = 0, \\
\dis\p_t v_i + \frac{p_\rho}{\rho}    \rho_{x_i} + \sum_{k=1}^{d} v_k   ( v_i)_{x_k} + \frac{p_{{}_S}}{\rho}  S_{x_i} = - v_i,\ i=1,\dots,d.
\end{cases}
\end{equation}
Denote $u = (S, \rho, v^T)^T$. By a direct calculation, the matrix
\[
A(u,\omega) =
\begin{pmatrix}
v \cdot \omega & 0 & 0 \\
0 & v \cdot \omega & \rho \omega^T \\
\frac{p_{{}_S}}{\rho} \omega & \frac{p_\rho}{\rho} \omega & v \cdot \omega I_d
\end{pmatrix}
\]
possesses $d+2$ real eigenvalues
\[
\lambda_1(u,\omega) = \dots = \lambda_d(u,\omega) = v \cdot \omega, \quad \lambda_{d+1}(u,\omega) = v \cdot \omega + \sqrt{p_\rho}, \quad \lambda_{d+2}(u,\omega) = v \cdot \omega - \sqrt{p_\rho}
\]
and a complete set of left eigenvectors
\begin{gather*}
l_1(u,\omega) = (1,0,0^T), \quad l_k(u,\omega) = (0,0,\chi_k^T)\  (k=2,\dots,d), \\
l_{d+1}(u,\omega) = (\frac{\ps}{\sqrt{p_\rho}}, \sqrt{p_\rho}, \rho \omega ),
\quad l_{d+2}(u,\omega) = (\frac{\ps}{\sqrt{p_\rho}}, \sqrt{p_\rho}, -\rho \omega)
\end{gather*}
and right eigenvectors
\begin{gather*}
r_1(u,\omega) = \begin{pmatrix}
1 \\ -\frac{\ps}{p_\rho} \\ 0
\end{pmatrix}, \quad
r_k(u,\omega) = \begin{pmatrix}
0 \\ 0 \\ \chi_k
\end{pmatrix}\ (k = 2,\dots, d), \\
r_{d+1}(u,\omega) = \begin{pmatrix}
0 \\ \frac{1}{2\sqrt{p_\rho}} \\ \frac{\omega }{2\rho}
\end{pmatrix}, \quad
r_{d+2}(u,\omega) = \begin{pmatrix}
0 \\ \frac{1}{2\sqrt{p_\rho}} \\ -\frac{\omega}{2\rho}
\end{pmatrix}.
\end{gather*}
Here $\chi_k \ ( k=2, \dots, d)$ are $d-1$ orthogonal unit vectors that are perpendicular to $\omega$.

It is clear that the assumption \asmp{B} is satisfied. Noting that $Q(u) = (0,0,-v^T)^T$, we check that around the equilibrium state $u_\ast  = (S_\ast,\rho_*, 0^T)^T$ with $S_\ast$ and $\rho_\ast>0$ given constants,
\begin{itemize}
\item $\nabla_u Q(u) = \begin{pmatrix}
0 & 0 \\
0 & -I_d
\end{pmatrix}$.
\item Take the entropy $\bar\eta = \frac{1}{2} \rho |v|^2 + \rho e(\rho,S) $ and the entropy flux $\bar\psi^k= v^k ( \frac{1}{2} \rho |v|^2 + \rho e(\rho,S) +p(\rho,S)) $, and then define
    \[\eta(u) = \frac{1}{2} \rho |v|^2 + \rho e(\rho,S) - \rho e(\rho_\ast,S_\ast) - \rho S \theta(\rho_\ast,S_\ast)+\rho_\ast S_\ast \theta(\rho_\ast,S_\ast) - \frac{\rho}{\rho_*} p(\rho_\ast,S_\ast) + p(\rho_\ast,S_\ast).\]
\item $\nabla_u \eta(u) \cdot Q(u) = - \rho |v|^2 \leq - \frac{\rho^*}{2} |v|^2$.
\item $\nabla_u Q(u_\ast ) r_j(u_\ast ,\omega) \neq 0, \ j=2,\dots, d+2. $
\item $\nabla_u Q(u) r_1(u) \equiv 0$ and $\nabla_u \lambda_1(u,\omega) \cdot r_1(u) \equiv 0$.
\item $ l_1(u) \nabla_u Q(u) \equiv 0. $
\end{itemize}
These verify the assumptions \asmp{A1}--\asmp{A4} and \asmp{D1}--\asmp{D3}.
Thus, Theorem \ref{thm:1.3} produces the unique global solution to the damping full Euler system in {$\rd$, $d\ge 2$.


\appendix

\section{Linear transformation}\label{linear trans}

In this appendix it will be showed that for the system \eqref{1.1_Sys} under the assumptions  \asmp{A1}--\asmp{A4} and \asmp{B}, the conditions
\eqref{1.10} and \eqref{1.11} can be always satisfied after a proper linear transformation.

By \asmp{A1}, one can apply to the system \eqref{1.1_Sys} the following linear transformation
\begin{equation}
u = u(u^{*}) := \begin{pmatrix}
I_r & 0 \\
\left( \frac{\partial Q_p}{\partial u_{p'}}(0) \right)^{-1}_{r+1 \leq p,p' \leq n} \left( \frac{\partial Q_p}{\partial u_{q}}(0) \right)_{\substack{1 \leq q \leq r \\ r+1 \leq p \leq n}} & I_{n-r}
\end{pmatrix}^{-1} u^{*}.
\end{equation}
It is easy to check that the system after this transformation keeps the same form as \eqref{1.1_Sys}, and all the assumptions \asmp{A1}--\asmp{A4} and \asmp{B} \  still hold.
Moreover, one has
\begin{equation} \label{2.41111}
\frac{\partial Q_i^{*}}{\partial u^{*}_q}(0) = 0, \hs q = 1, \dots, r;\ i = 1, \dots, n.
\end{equation}
This verifies the condition \eqref{1.10}.

Now, by   \asmp{A1}, \asmp{A4} and \eqref{1.10}, it is direct to get
\begin{equation}
 r^*_1(0) \in \mathrm{span} \{e_1, \dots, e_r\},
\end{equation}
which implies that $ r^*_{p1}(0) = 0 $ $(p = r+1,\dots,n)$ and for some $i_1 \in \{1, \dots, r\}$,
\begin{equation}
r^*_{i_11}(0) \neq 0.
\end{equation}
Then one may choose $\{ i_2, \dots, i_r \} = \{1, \dots, r\} \setminus \{ i_1 \}$ and have that
\begin{equation}
\mathrm{span}\{r^*_1(0), e_{i_2}, \dots, e_{i_r}, e_{r+1}, \dots, e_n \} = \mathbb{R}^n.
\end{equation}
One can further apply to the system \eqref{1.1_Sys} the linear transformation
\begin{equation}
u^* = u^*(u^{**}):= \left( r_1^*(0), e_{i_2}, \dots,e_{i_r}, e_{r+1}, \dots, e_n \right) u^{**}.
\end{equation}
Again, it is easy to check that the transformed system keeps the same form as \eqref{1.1_Sys}, and all the assumptions \asmp{A1}--\asmp{A4}, \asmp{B}\  and the condition \eqref{1.10} still hold.
Moreover, one has
\begin{equation}
r^{**}_{1}(0) = e_1.
\end{equation}
This verifies the condition \eqref{1.11}.



\section{Analytic tools}\label{1section_appendix}

First recall the classical Sobolev interpolation of the Gagliardo-Nirenberg inequality.
\begin{lem} \label{A3}
Let  $1 \le  q,\,r \le  \infty$ and $0 \leq \alpha < \beta$. Then it holds that
\begin{equation}\label{aaa333}
\norm{\Lambda^\alpha f}_{L^p} \ls \norm{\Lambda^\beta f}_{L^r}^\theta \norm{f}_{L^q}^{1-\theta},
\end{equation}
where
\begin{equation}
 \frac{1}{p} - \frac{\alpha}{d} = \theta \left( \frac{1}{r} - \frac{\beta}{d} \right) + (1-\theta) \frac{1}{q},
\end{equation}
for all $\theta$ in the interval
\[
\frac{j}{m} \leq \theta \leq 1
\]
with the following exceptional cases.
\begin{enumerate}
\item If $\alpha = 0$, $r\beta < d$, $q = \infty$, then one needs an additional assumption that either $f$ tends to zero at infinity or $f \in L^{\tilde{q}}$ for some finite $\tilde{q} > 0$.
\item If $1 < r < \infty$, meanwhile $\beta- \alpha - d/r$ is a non negative integer, then \eqref{aaa333} holds only for $\theta$ satisfying $\alpha/\beta < \theta < 1$.
\end{enumerate}
\end{lem}
\begin{proof}
See THEOREM of Lecture II in \cite{Nirenberg}.
\end{proof}

In the lemma above the restriction $0 \leq \alpha < \beta$ is inconvenient for the use in this paper, one may employ the following variant.
\begin{lem}\label{A1}
Let $2\le p\le \infty$ and $\alpha,\beta,\gamma\in \mathbb{R}$. Then it holds that
\begin{align}\label{A.1}
\norm{\Lambda^\alpha f}_{L^p}\lesssim \norm{ \Lambda^\beta f}_{L^2}^{\theta}
\norm{ \Lambda^\gamma f}_{L^2}^{1-\theta}.
\end{align}
Here $0\le \theta\le 1$ (if $p=\infty$, then it is required that $0<\theta<1$) and $\alpha$ satisfies
\begin{align}\label{A22222}
\alpha+ d \left(\frac12-\frac{1}{p}\right)=\beta\theta+\gamma(1-\theta).
\end{align}
\end{lem}
\begin{proof}
For the case $2\le p<\infty$, it follows from the classical Sobolev inequality that
\begin{equation}\label{11in}
\norm{\Lambda^\alpha f}_{p}\lesssim \norm{ \Lambda^\zeta f}_{2}\hbox{ with }\zeta=\alpha+d\left(\frac{1}{2}-\frac{1}{p}\right).
\end{equation}
By the Parseval theorem and H\"older's inequality, we have
\begin{equation}\label{12in}
\norm{ \Lambda^\zeta f}_{2}\lesssim \norm{ \Lambda^\beta f}_{2}^{\theta}\norm{ \Lambda^\gamma f}_{2}^{1-\theta},
\end{equation}
where $0\le \theta\le 1$ is defined by \eqref{A22222}. Hence, \eqref{A.1} follows by \eqref{11in}--\eqref{12in}.

For the case $p=\infty$, one may refer to Exercise 6.1.2 in \cite{grafakos2009modern}.
\end{proof}

Recall lastly the following commutator and product estimates:
\begin{lem}\label{A2}
Let $l\ge 0$ and define the commutator
\begin{align}\label{commutator}
\left[\Lambda^l,g\right]h=\Lambda^l(gh)-g\Lambda^lh.
\end{align}
Then it holds that
\begin{align}\label{commutator estimate}
\norm{\left[\Lambda^l,g\right]h}_{L^{p_0}} \ls\norm{\na g}_{L^{p_1}}
\norm{\Lambda^{l-1}h}_{L^{p_2}} +\norm{\Lambda^l g}_{L^{p_3}}\norm{h}_{L^{p_4}}.
\end{align}
In addition, for $l\ge0$,
\begin{align}\label{product estimate}
\norm{\Lambda^l(gh)}_{L^{p_0}} \ls\norm{g}_{L^{p_1}}
\norm{\Lambda^{l}h}_{L^{p_2}} +\norm{\Lambda^l g}_{L^{p_3}} \norm{ h}_{L^{p_4}}.
\end{align}
Here $p_0,p_2,p_3\in(1,\infty)$ and
\begin{align*}
\frac1{p_0}=\frac1{p_1}+\frac1{p_2}=\frac1{p_3}+\frac1{p_4}.
\end{align*}
\end{lem}
\begin{proof}
One may refer to Lemma 3.1 in \cite{ju2004existence}.
\end{proof}

\section*{Acknowledgements}

The authors are deeply grateful to Professor Zhouping Xin for his great guidance and constant encouragement.

\bibliographystyle{plain}

\bibliography{Ref20151010}

%
%
%

\end{document}